\documentclass[reqno,A4paper]{amsart}

\usepackage{amsmath}
\usepackage{amssymb}
\usepackage{amsthm}
\usepackage{enumerate}
\usepackage{mathrsfs} 
\usepackage{eqlist}
\usepackage{array}

\usepackage[english]{babel}                             
\usepackage[latin1]{inputenc}
\usepackage{comment}
\usepackage{ae}
\usepackage[leqno]{amsmath}  

\usepackage[dotinlabels]{titletoc}     

\usepackage{amssymb,latexsym,verbatim} 
\usepackage{amstext}
\usepackage{amsfonts}
\usepackage{amsbsy}
\usepackage{amsopn}
\usepackage{enumerate}
\usepackage{txfonts}
\usepackage{amsmath, amsthm, amssymb}

\usepackage{color}
\usepackage[usenames,dvipsnames,svgnames,table]{xcolor}
\usepackage[numbers,sort&compress]{natbib}
\allowdisplaybreaks
\newtheorem{theorem}{Theorem}
\newtheorem{definition}[theorem]{Definition}
\newtheorem{lemma}[theorem]{Lemma}

\newtheorem{proposition}[theorem]{Proposition}
\newtheorem{remark}[theorem]{Remark}
\numberwithin{equation}{section}
\numberwithin{theorem}{section}

\renewcommand{\epsilon}{\varepsilon}
\renewcommand{\rho}{\varrho}

\def\Xint#1{\mathchoice
{\XXint\displaystyle\textstyle{#1}}%
{\XXint\textstyle\scriptstyle{#1}}%
{\XXint\scriptstyle\scriptscriptstyle{#1}}%
{\XXint\scriptscriptstyle\scriptscriptstyle{#1}}%
\!\int}
\def\XXint#1#2#3{{\setbox0=\hbox{$#1{#2#3}{\int}$ }
\vcenter{\hbox{$#2#3$ }}\kern-.6\wd0}}

\def\dashint{\Xint-}

\usepackage{graphicx} 

\def\YYint#1#2#3{{\setbox0=\hbox{$#1{#2#3}{\iint}$}
    \vcenter{\hbox{$#2#3$}}\kern-.51\wd0}}

\numberwithin{equation}{section}

\begin{document}

\title[Partial regularity for degenerate parabolic systems]%
{Partial regularity for degenerate parabolic systems with non-standard growth and discontinuous coefficients
}
\author[Qifan Li]%
{Qifan Li*}

\newcommand{\acr}{\newline\indent}

\address{\llap{*\,}Department of Mathematics\acr
                   School of Sciences\acr
                   Wuhan University of Technology\acr
                   430070, 122 Luoshi Road,
                   Wuhan, Hubei\acr
                   P. R. China}
\email{qifan\_li@yahoo.com, qifan\_li@whut.edu.cn}

\subjclass[2010]{35K40, 35K65, 35K67, 35K92, 35B65.} 
\keywords{Partial regularity theory, Quasilinear parabolic system, Non-standard growth condition.}

\begin{abstract}
This article studies the partial H\"older continuity of weak solutions to certain degenerate parabolic
systems whose model is the differentiable parabolic $p(x,t)$-Laplacian system,
\begin{equation*}\partial_t u-\operatorname{div}[\mu(z)(1+|Du|^2)^{\frac{p(z)-2}{2}}Du]=0,\qquad p(z)\geq2.\end{equation*}
Here, the exponential function $p(z)$ satisfies a logarithmic continuity
condition.
We show that if $\mu(z)$ satisfies a certain VMO-type condition, then $u$ is locally H\"older
continuous
except for a measure zero set.
\end{abstract}
\maketitle
\section{Introduction}
The aim of this paper is to establish a partial regularity result for weak solutions to parabolic systems of the type
\begin{equation*}\partial_t u-\operatorname{div}\left[\mu(z)A(z,Du)\right]=0,\end{equation*}
where the coefficient $\mu(z)$ is discontinuous and the
vector field $A(z,Du)$ exhibits non-standard $p(x,t)$-growth conditions. The partial regularity theory
for weak solutions of parabolic systems was first studied by Campanato \cite{C}.
Unlike the case of single equation, in general not every weak solution of parabolic system is everywhere H\"older continuous
and we can hope for is a result ensuring H\"older continuity outside a Lebesgue measure zero set.
For H\"older continuous coefficients, the partial regularity theory for parabolic systems with standard $p$-growth
has been established by Duzzar, Mingione, Steffen \cite{DMS} in the superquadratic case and Scheven \cite{S}
in the subquadratic case. In \cite{DMS,S} the authors proved that
the gradient of the weak solution is partial H\"older continuous.
Moreover, in \cite{BFM,FG} it has been proved that weak solutions to parabolic systems are
H\"older continuous except a Lebesgue measure zero set, provided that the coefficients are merely continuous. Recently,
Mons \cite{M} proved the partial regularity result for the VMO coefficients, which allows
the coefficients to be discontinuous.

In recent years, there has been tremendous interest in developing regularity theory
for the parabolic systems with non-standard $p(x,t)$-growth. However, limited work has been done in
the partial regularity problem
for this kind of parabolic systems. The first result was established  by
Acerbi, Mingione and Seregin \cite{AcMS}. In \cite{AcMS} the authors obtained a Hausdorff dimension estimate of the singular set
for the parabolic systems related to a class
of non-Newtonian fluids. Subsequently,
Duzaar and Habermann \cite{DH} studied the partial regularity problem with H\"older continuous coefficients by using the
$A$-caloric approximation method. Motivated by this work, we are interested in extending the main result in \cite{M}
to the variable exponent case. Our main result states that any
weak solution to degenerate parabolic systems with non-standard growth and VMO coefficients
is partially H\"older continuous for any H\"older exponents.

Our approach is in the spirit of \cite{BDHS,BFM,DMS,DH,M} which uses the $A$-caloric approximation method.
However, the treatment of non-standard growth condition is considerably more delicate.
The higher integrability estimate is a standard tool in the proof and the higher integrability exponents play an important
role in determining the parameters. From the higher integrability result, we shall derive an $L^p\log^\gamma L$ estimate which is
standard ingredient in the partial regularity proof.
Contrary to \cite{DH}, gradients of the affine functions used in our proof
may not be bounded and the method of intrinsic scaling developed by DiBenedetto, Kinnunen and Lewis \cite{KL1} is necessary.
In the context of non-standard growth condition,
we have to work with a non-standard version
of the intrinsic geometry and establish the Caccioppoli inequality and decay estimates on the scaled parabolic cylinders.

An outline of this paper is as follows. In Sect. 2, we provide
some preliminary material and state the main
result. We also give the characterization of the singular sets in this section.
In Sect. 3, we derive an estimate for gradients of solutions in $L^p\log^\gamma L$ space by using the
higher integrability estimates. Sect. 4 is devoted to the study of Poincar\'e type inequality for the weak solution.
Furthermore, we obtain an alternative characterization of the regular points from the
Poincar\'e inequality.
In Sect. 5 we establish the Caccioppoli-type estimate which is a reverse Poincar\'e inequality.
In Sect. 6 we use the $A$-caloric approximation method to derive a decay estimate. Finally the proof of the
partial regularity result is presented in Sect. 7 by using an iteration method.
\section{Preliminary material}
In the present section, we set up notations and give the statement of the main result.
Throughout the paper, we assume that $\Omega$ is an open bounded domain in $\mathbb{R}^n$ with $n\geq 2$.
We write $\{e_i\}_{i=1}^n$ for the standard basis of $\mathbb{R}^n$.
For $T>0$,
let $\Omega_T=\Omega\times(-T,0)$. Given a point $z_0=(x_0,t_0)\in \mathbb{R}^{n+1}$
and $r>0$, we set $B_r(x_0)=\{x\in\mathbb{R}^n:\ |x-x_0|<r\}$, $\Lambda_r(t_0)=(t_0-r^2,t_0)$ and $Q_r(z_0)=B_r(x_0)\times \Lambda_r(t_0)$.
For $\lambda>0$, we define the intrinsic parabolic cylinder $Q_r^{(\lambda)}(z_0)$ by
$Q^{(\lambda)}_r(z_0):=B_r(x_0)\times\Lambda_r^{(\lambda)}(t_0)$ where $ \Lambda_r^{(\lambda)}(t_0):=(t_0-\lambda^{2-p_0}r^2,t_0)$
and $p_0=p(z_0)$.
If the reference point $z_0$ is the origin, then
we omit in our notation the point $z_0$ and write
$B_r$, $\Lambda_r^{(\lambda)}$ and $Q_{r}^{(\lambda)}$ for $B_r(0)$, $\Lambda_r^{(\lambda)}(0)$
and $Q_{r}^{(\lambda)}(0)$. Specifically, if $\lambda=1$, then we abbreviate $\Lambda_r=\Lambda_r^{(1)}$ and $Q_{r}=Q_{r}^{(1)}$.
Given a function $f\in L^1(Q,\mathbb{R}^N)$, with $Q\subset\mathbb{R}^{n+1}$ and $N\in \mathbb{N}$ we define
\begin{equation*}(f)_Q=\ \dashint_Q|f|\,\mathrm{d}z:=\frac{1}{|Q|}\ \int_Q|f|\,\mathrm{d}z.\end{equation*}
For $u\in L^2(Q_r^{(\lambda)},\mathbb{R}^N)$ with $Q_r^{(\lambda)}\subset\Omega_T$,
we denote by $l_{z_0,\ r}^{(\lambda)}:\mathbb{R}^n\mapsto\mathbb{R}^N$ the unique affine map minimizing the functional
\begin{equation*}F[l]=\ \dashint_{Q_r^{(\lambda)}}|u-l|^2\,\mathrm{d}z\end{equation*}
among all affine maps $l=l(x_0)+Dl\cdot (x-x_0)$ which are independent of $t$.
In this work we are concerned with the quasilinear parabolic systems of the divergence form
\begin{equation}\label{parabolic}\partial_t u-\operatorname{div}\left[\mu(z)A(z,Du)\right]=0,\end{equation}
where $\mu:\mathbb{R}^{n+1}\mapsto\mathbb{R}$ is a coefficient and $u:\mathbb{R}^{n+1}\mapsto\mathbb{R}^N$ is an integrable map.
We assume that $(z,w)\mapsto A(z,w)$ and $(z,w)\mapsto \partial_wA(z,w)$ are continuous in $\Omega_T\times \mathbb{R}^{N\times n}$
and satisfy the following non-standard $p(z)$-growth and ellipticity conditions:
 \begin{equation}\label{A}
	\begin{cases}
	 \big\langle\partial_wA(z,w)\tilde w,\tilde w \big \rangle\geq \sqrt{\nu}(1+|w|^2)^{\frac{p(z)-2}{2}}|\tilde w|^2,\\
	|A(z,w)|+(1+|w|^2)^{\frac{1}{2}}|\partial_wA(z,w)|\leq \sqrt{L}(1+|w|^2)^{\frac{p(z)-1}{2}},
	\end{cases}
\end{equation}
for all $z\in\Omega_T$, $w$, $\tilde w\in\mathbb{R}^{N\times n}$.
Here, $\nu$ and $L$ are fixed structural parameters. In this paper, we only consider the degenerate case where $p(z)\geq2$.
For the exponent function $p:\Omega_T\mapsto[2,+\infty)$, we assume that it is continuous with a modulus of continuity
$\omega_p:\Omega_T\mapsto[0,1]$.
More precisely, we assume that for any $z_1=(x_1,t_1),z_2=(x_2,t_2)\in\Omega_T$,
 \begin{equation}\label{pz1pz2}
|p(z_1)-p(z_2)|\leq \omega_p(d_{\mathcal{P}}(z_1,z_2)),\end{equation}
where $d_{\mathcal{P}}(z_1,z_2)=\max\left\{|x_1-x_2|,\sqrt{|t_1-t_2|}\right\}$. Moreover, we shall assume $ 2\leq p(z)\leq \gamma_2$
for a fixed constant $\gamma_2>0$. The modulus of continuity $\omega_p$ is assumed to be a concave, non-decreasing function satisfying the
following logarithmic continuity condition:
  \begin{equation}\label{omegap}\lim_{\rho\downarrow0}\omega_p(\rho)\log\left(\frac{1}{\rho}\right)=0.\end{equation}
  Concerning the coefficients $\mu(z)$, we impose a
certain VMO-type condition. More precisely,
 we assume that $\sqrt{\nu}\leq \mu(z)\leq \sqrt{L}$ for all $z\in \Omega_T$ and satisfy the following vanishing mean oscillation condition:
 \begin{equation}\label{VMO}\lim_{r\to 0}v(r)=0,\end{equation}
 where
 \begin{equation*}v(r)=\sup\left\{\ \dashint_{Q_{\rho,\sigma}(z_0)\cap\Omega_T}
 |\,\mu(z)-(\mu)_{Q_{\rho,\sigma}(z_0)\cap\Omega_T}|\,\mathrm {d}z:\quad
 \max\{\rho,\sqrt{\sigma}\}\leq r,\quad z_0\in\Omega_T\right\}\end{equation*}
 and $Q_{\rho,\sigma}(z_0)=B_\rho(x_0)\times(t_0-\sigma,t_0)$.
Furthermore, we shall assume that the partial map $z\mapsto A(z,w)$ satisfies the following continuity condition:
 \begin{equation}\begin{split}\label{Az1z2}|A(z_1,w)-A(z_2,w)|\leq &L\omega_p(d_{\mathcal{P}}(z_1,z_2))
 \left[(1+|w|)^{p(z_1)-1}+(1+|w|)^{p(z_2)-1}\right]
 \\&\times\left[1+\log(1+|w|)\right]\end{split}\end{equation}
 for any $z_1$, $z_2\in \Omega_T$ and $w\in \mathbb{R}^{N\times n}$. Finally, we shall assume that the partial
map $z\mapsto \partial_wA(z,w)$ is continuous in the sense that there exists a bounded, concave
and non-decreasing modulus of continuity $\hat\omega_a(\cdot)$ such that $\hat\omega_a(0)=0$ and
  \begin{equation}\begin{split}\label{DAz1z2}|\partial_wA(z,w_1)-\partial_w A(z,w_2)|\leq L\hat\omega_a\left(\frac{|w_1-w_2|}
  {1+|w_1|+|w_2|}\right) (1+|w_1|+|w_2|)^{p(z)-2}\end{split}\end{equation}
  holds for any $w_1,w_2\in \mathbb{R}^{N\times n}$ and $z\in \Omega_T$.
 Now we give the definition of a weak solution to the parabolic system \eqref{parabolic}.
 \begin{definition}A function $u\in L^1(\Omega_T,\mathbb{R}^N)$ is called weak solution to the parabolic system \eqref{parabolic} if and only
 if $u\in C^0([-T,0];L^2(\Omega,\mathbb{R}^N))$, $|u|^{p(\cdot)}$, $|Du|^{p(\cdot)}\in L^1(\Omega_T)$ and
\begin{equation}\label{weaksolution}\int_{\Omega_T}\left[u\cdot\partial_t \varphi-\mu(z)\big\langle A(z,Du), D\varphi\big\rangle\right]
\mathrm {d}z=0\end{equation}
holds,
whenever $\varphi\in C_0^\infty(\Omega_T,\mathbb{R}^N)$.
 \end{definition}
We are now in a position to state our main theorem which also
present a characterizations of singular sets.
 \begin{theorem}\label{main}Let $u$ be a weak solution of the parabolic system \eqref{parabolic}, where the assumptions \eqref{A}-\eqref{DAz1z2}
 are in force. Then there exists an open subset $\Omega_0\subset\Omega_T$
 with $|\Omega_T\setminus \Omega_0|=0$ such that $u\in C_{\mathrm{loc}}^{0;\alpha,\alpha/2}(\Omega_0,\mathbb{R}^N)$,
where $\alpha\in(0,1)$. More precisely, we have that the
 singular set fulfils $\Omega_T\setminus \Omega_0\subset \Sigma_1\cup\Sigma_2$, where $\Sigma_1$ and $\Sigma_2$ are defined in the following way:
  \begin{equation*}\begin{split}\Sigma_1:&=\left\{z_0\in \Omega_T:\liminf\limits_{\rho\downarrow0}\ \dashint_{Q_\rho(z_0)}
  |Du-(Du)_{Q_\rho(z_0)}|\,\mathrm {d}z>0
  \right\},
  \\ \Sigma_2:&=\left\{z_0\in \Omega_T:\limsup\limits_{\rho\downarrow0}\left(|Du|^{p(\cdot)}\right)_{Q_\rho(z_0)}=+\infty
  \right\}.\end{split}\end{equation*}
 \end{theorem}
 We conclude this section by pointing out that the Lebesgue measure $\mathscr{L}^{n+1}(\Sigma_1\cup\Sigma_2)=0$ by the Lebesgue differentiation
 theorem and this also implies that $|\Omega_T\setminus \Omega_0|=0$.
 \section{Higher integrability}
 A key ingredient in the partial regularity proof is the higher integrability for the gradient of the solutions to the
 parabolic system \eqref{parabolic}. In this section, we collect some higher integrability results and obtain
 a logarithmic estimate for the gradients. Since $|Du|^{p(\cdot)}\in L^1(\Omega_T)$,
we assume in this paper that
 \begin{equation}\label{E}\int_{\Omega_T}(|Du|+1)^{p(z)}\,\mathrm {d}z\leq E,\end{equation}
 for some constant $E>1$.
 Next, we recall the following higher integrability results for the degenerate case $p(z)\geq 2$,
 which were obtained from \cite[Theorem 2.2]{BD}.
 \begin{proposition}\label{higher integrability0}(\cite[Theorem 2.2]{BD})
 There exists $\epsilon_0=\epsilon_0(n,N,\nu,L,\gamma_2)\in(0,1)$ such that
 the following holds: Whenever $u$ is a weak solution of the parabolic system \eqref{parabolic},
 where the assumptions \eqref{A}-\eqref{DAz1z2}
 are in force, then there holds
\begin{equation}\label{higher integrability initial}|Du|^{p(\cdot)(1+\epsilon_0)}\in L_{\mathrm{loc}}^1(\Omega_T).\end{equation}
  Moreover, if \eqref{E} holds, then there exists a radius $\rho_*=\rho_*(n,N,\nu,L,\gamma_2,E)>0$ such that
  for any parabolic cylinder $Q_{2\rho}(z_0)\subset\Omega_T$ with $\rho\leq\rho_*$ and $\epsilon\in(0,\epsilon_0]$
  there holds
   \begin{equation}\label{higher integrability} \dashint_{Q_\rho(z_0)}|Du|^{p(z)(1+\epsilon)}\,\mathrm {d}z\leq c
   \left[\Big(\ \,\dashint_{Q_{2\rho}(z_0)}|Du|^{p(z)}\,\mathrm {d}z\Big)^{1+\frac{1}{2}\epsilon p_0}+1\right],\end{equation}
   for a constant $c$ depending only upon $n$, $N$, $\nu$, $L$, $\gamma_2$ and $E$.
 \end{proposition}
Our first goal is to improve the higher integrability
\eqref{higher integrability initial} to a fixed scaled parabolic cylinder and derive some properties. Let $\epsilon_0\in(0,1)$ be the exponent as in
Proposition \ref{higher integrability0}. Fix $z_0\in\Omega_T$, $\rho>0$ and $\lambda\geq1$, we now consider the oscillation of $p(z)$
on the parabolic cylinder $Q_{\rho}^{(\lambda)}(z_0)$ and define
 \begin{equation}\label{p1p2}p_1=\inf _{Q_{\rho}^{(\lambda)}(z_0)}p(z)\qquad\text{and}\qquad p_2=\sup _{Q_{\rho}^{(\lambda)}(z_0)}p(z).\end{equation}
Since $p_0=p(z_0)\geq2$ and $\lambda\geq1$, we infer from \eqref{pz1pz2} and \eqref{p1p2} that
  \begin{equation}\label{p2-p1}p_2-p_1\leq \max_{z_1,z_2\in
  Q_{\rho}^{(\lambda)}(z_0)}\omega_p(d_{\mathcal{P}}(z_1,z_2)) \leq \omega_p(\rho).\end{equation}
This also implies that $p_2-p_1\leq1$. Moreover, we conclude from \eqref{omegap} and \eqref{p2-p1} that
    \begin{equation}\label{exp p2-p1}\rho^{-(p_2-p_1)}\leq \rho^{-\omega_p(\rho)}=e^{\omega_p(\rho)\log\left(\frac{1}{\rho}\right)}
    \leq e.\end{equation}
 At this point, we choose $\tilde\rho_0=\tilde\rho_0(\epsilon_0)>0$ small enough to have
 \begin{equation}\label{rho0}
 \tilde \rho_0\leq \rho_*\qquad\text{and}\qquad\omega_p(16\tilde\rho_0)\leq \epsilon_0,\end{equation}
 where $\rho_*$ is the radius in Proposition \ref{higher integrability}.
 From \eqref{rho0}, we deduce from \eqref{higher integrability initial} that for any $\rho<\tilde \rho_0$ there holds
\begin{equation}Du\in L^{p_2}(Q_\rho^{(\lambda)}(z_0),\mathbb{R}^{N\times n}).\end{equation}
Our task now is to establish a relationship between the exponential function and the scaling factor.
To this end, we introduce the quantity
 \begin{equation}\label{phi}\Phi^{(\lambda)}(z_0,\rho,l)=\ \dashint_{Q_{\rho}^{(\lambda)}(z_0)}
 \left(\frac{|Du-Dl|}{1+|Dl|}\right)^2+\left(\frac{|Du-Dl|}{1+|Dl|}\right)^{p_0}\,\mathrm {d}z,\end{equation}
 where $p_0=p(z_0)$ and $l$ is an affine function.
Specifically, if $\lambda=1$, then we simply write $\Phi(z_0,\rho,l)$
for $\Phi^{(1)}(z_0,\rho,l)$.
The next lemma provides some useful estimates for the treatment of the non-standard growth.
  \begin{lemma}\label{lambda} Let $z_0\in\Omega_T$ and $\tilde\rho_0$ satisfies \eqref{rho0}. Let $p_1$ and $p_2$ be the exponents defined in
  \eqref{p1p2}.
  Suppose that $\lambda\geq1$ and $\lambda\leq 1+|Dl|\leq M\lambda$ for some $M\geq1$.
  If $\Phi^{(\lambda)}(z_0,\rho,l)\leq \frac{1}{16}$ and $\rho\leq\tilde\rho_0$, then we have the following properties:
  \begin{itemize}
  \item [(1)] $\frac{1}{2}(1+|Dl|)\leq1+|(Du)_{Q_{\rho}^{(\lambda)}(z_0)}|\leq3(1+|Dl|)$,
\end{itemize}
 \begin{itemize}
  \item [(2)] $\dashint_{Q_\rho^{(\lambda)}(z_0)}|Du|^{p_0}\,\mathrm {d}z\leq c(1+|Dl|)^{p_0}$,
  \end{itemize}
  \begin{itemize}
  \item [(3)] $\lambda^{p_2-p_1}\leq c$ and $(1+|Dl|)^{p_2-p_1}\leq c$,
\end{itemize}
 \begin{itemize}
  \item [(4)] $1+|Dl|\leq c\rho^{-\gamma_2\frac{(n+2)}{2}}$,
\end{itemize}
  where the constant $c$ depends only upon $n$, $E$, $\gamma_2$ and $M$.
   \end{lemma}
   \begin{proof} We first observe that
   the first two claims follows directly from the proof of \cite[Lemma 3.3]{H} and the assumption $\Phi^{(\lambda)}(z_0,\rho,l)\leq\frac{1}{16}$.
   For the proof of the third claim, we apply \eqref{E} and \eqref{exp p2-p1} to obtain
   \begin{equation}\begin{split}\label{lemma3.2 1}\lambda^{p_2-p_1}&\leq (1+|Dl|)^{p_2-p_1}\leq 2\Big(1+|(Du)_{Q_{\rho}^{(\lambda)}(z_0)}|
\Big)^{p_2-p_1}
   \\&\leq c\left(1+\Big(\ \dashint_{Q_{\rho}^{(\lambda)}(z_0)}|Du|^{p(z)}\,\mathrm {d}z\Big)^{\frac{1}{p_1}}\right)^{p_2-p_1}
   \\&\leq c\frac{1}{\rho^{(n+2)(p_2-p_1)/p_1}}\lambda^{\frac{p_0-2}{p_1}(p_2-p_1)}(1+E)^{p_2-p_1}\leq c\lambda^{\frac{p_0-2}{p_1}(p_2-p_1)},
   \end{split}\end{equation}
   where the constant $c$ depending only upon $n$ and $E$. Since $p_2-p_1\leq 1$, we infer from \eqref{lemma3.2 1} that
     $\lambda^{p_2-p_1}\leq c(n,E)^{\frac{p_1}{2-(p_2-p_1)}}\leq c(n,E)^{\gamma_2}.$
Furthermore, $(1+|Dl|)^{p_2-p_1}\leq c(n,\gamma_2,E,M)$ follows by the assumption $1+|Dl|\leq M\lambda$.
Next, we consider the proof of the fourth claim.
In view of the first claim, we use the H\"older's inequality and \eqref{E} to obtain
 \begin{equation}\begin{split}\label{lemma3.2 2}
1+|Dl|\leq 2&\big(\,1+|(Du)_{Q_{\rho}^{(\lambda)}(z_0)}|\,\big)
\leq 2+\frac{2}{\lambda^{\frac{2-p_0}{p_1}}\rho^{\frac{n+2}{2}}}\Big(\int_{Q_{\rho}^{(\lambda)}(z_0)}|Du|^{p_1}\,\mathrm {d}z
\Big)^{\frac{1}{p_1}}\\
&\leq \frac{c(E)\lambda^{\frac{p_0-2}{p_1}}}{\rho^{\frac{n+2}{2}}}\leq \frac{c(E)}{\rho^{\frac{n+2}{2}}}(1+|Dl|)^{\frac{p_0-2}{p_1}},
\end{split}\end{equation}
since $p_0\geq 2$ and $1\leq\lambda\leq 1+|Dl|$. Since $p_2-p_1\leq 1$, we infer from \eqref{lemma3.2 2} that
\begin{equation*}1+|Dl|\leq c\rho^{-\frac{n+2}{2}\frac{p_1}{p_1-p_0+2}}\leq c\rho^{-\gamma_2\frac{(n+2)}{2}}\end{equation*}
and this proves the fourth claim. We have thus proved the lemma.
   \end{proof}
Our next destination is to refine the estimate \eqref{higher integrability} on the scaled parabolic cylinder $Q_{\rho}^{(\lambda)}(z_0)$.
To this end,
we establish the following version of the higher integrability estimate for non-uniformly parabolic cylinders.
 \begin{lemma}\label{fitst higher} Let $z_0\in\Omega_T$ and $Q_{2\rho}^{(\lambda)}(z_0)\subset\Omega_T$. Let $p_1$ and $p_2$ be the
exponents defined in
  \eqref{p1p2}.
 Assume that
 $\Phi^{(\lambda)}(z_0,\rho,l)\leq \frac{1}{16}$, $\lambda\geq1$ and $\lambda\leq 1+|Dl|\leq M\lambda$ for some $M\geq1$.
 Then, there exists a constant $\delta_0=\delta_0(n,E,\gamma_2,\nu,L,M)>0$ such that the following holds: Whenever $\delta\in(0,\delta_0]$,
 there exists a radius $\hat\rho_0=\hat\rho_0(n,E,\gamma_2,\nu,L,M,\delta)>0$ such that
 for any $\rho\leq\hat\rho_0$
 the following inequality holds
    \begin{equation}\label{higher integrability1} \dashint_{Q_{\rho/2}^{(\lambda)}(z_0)}|Du|^{p_1(1+\delta)}\,\mathrm {d}z\leq c
   \left[\Big(\ \,\dashint_{Q_{\rho}^{(\lambda)}(z_0)}|Du|^{p_1}\,\mathrm {d}z\Big)^{1+\beta(\delta)}+1\right],\end{equation}
    where the constant $c$ depends only upon  $n$, $E$, $\gamma_2$, $\nu$, $L$ and $M$. Here, the constant
  $\beta(\delta)$ is defined by
  \begin{equation}\label{betadelta}\beta(\delta)=\frac{\delta}{1-\frac{p_2-p_1}{\delta p_1}(1+\frac{\delta p_2}{2})}.\end{equation}
 \end{lemma}
    \begin{proof} To start with, we first choose $\rho<\tilde\rho_0$ where the radius $\tilde\rho_0$ satisfies \eqref{rho0}.
This choice of $\rho$ allows us to use
    Lemma \ref{lambda}.
 Let $\frac{1}{2}\rho<s<t<\rho$ and set $\hat z=(\hat x,\hat t)\in Q_s^{(\lambda)}(z_0)$. Moreover, we set $r=t-s$ and consider the
 uniform parabolic cylinder
\begin{equation}\label{new cylinder}\tilde Q_r^{(\lambda)}(\hat z):=B_r(\hat x)\times(\hat t-\lambda^{2-p_0}r^2,\hat t\,),\end{equation}
where $p_0=p(z_0)$. By abuse of notation, we continue to write $Q_r^{(\lambda)}(\hat z)$ for $\tilde Q_r^{(\lambda)}(\hat z)$.
It can be easily seen that $Q_r^{(\lambda)}(\hat z)\subset Q_t^{(\lambda)}(z_0)$.
    We divide our proof in two steps.

    Step 1: Obtaining a higher integrability type estimate on $Q_r^{(\lambda)}(\hat z)$, where $\hat z\in Q_s^{(\lambda)}(z_0)$.
    Introducing
the change of variables and the new functions
 \begin{equation}\label{change of variable}
	\begin{cases}
	\ \tilde u(y,\tau)=\lambda^{-1}u(\hat x+y,\hat t+\lambda^{2-p_0}\tau),\\
	\ \tilde A(y,\tau,w)=\lambda^{1-p_0}A(\hat x+y,\hat t+\lambda^{2-p_0}\tau,\lambda w),\\
\ \tilde \mu(y,\tau)=\mu(\hat x+y,\hat t+\lambda^{2-p_0}\tau),\\
\ \tilde p(y,\tau)=p(\hat x+y,\hat t+\lambda^{2-p_0}\tau),
	\end{cases}
\end{equation}
we rewrite the parabolic system \eqref{parabolic} in terms of the
new variables $(y,\tau)$ and new functions $\tilde u$ as follows:
\begin{equation}\label{new parabolic}\partial_\tau \tilde u-\operatorname{div}_y\left[\tilde \mu(y,\tau) \tilde A(y,\tau,D_y\tilde u)\right]=0,
\quad\text{in}\quad
Q_{2r}(0).\end{equation}
Our task now is to establish a reverse H\"older inequality similar to \eqref{higher integrability} for the map $\tilde u$. To this end,
we first observe that $\tilde u\in C([-4r^2,0];L^2(B_{2r}(0);\mathbb{R}^N))$ as well as
$|\tilde u|^{\tilde p(\cdot)}$, $|D_y\tilde u|^{\tilde p(\cdot)}\in L^1(Q_{2r}(0))$.
Moreover, we check that
\begin{equation*}\label{new weaksolution}\int_{Q_{2r}(0)}\left[\tilde u\cdot\partial_\tau \tilde\varphi-\tilde\mu(y,\tau)\big\langle \tilde A(
y,\tau,D_y\tilde u), D_y\tilde \varphi\big\rangle\right]
\,\mathrm {d}y\mathrm {d}\tau=0\end{equation*}
holds,
whenever $\tilde\varphi(y,\tau)\in C_0^\infty(Q_{2r}(0),\mathbb{R}^N)$. It follows that $\tilde u$ is indeed a weak solution to
\eqref{new parabolic}.
The next thing to do in the proof is to check that
the conditions in \cite[Theorem 2.2]{BD}
are fulfilled.
Since $\rho<\tilde\rho_0$ and $\lambda\geq1$, we use \eqref{E} and Lemma \ref{lambda} (3) to deduce that
\begin{equation}\begin{split}\label{new e}\int_{Q_{2r}(0)}|D_y\tilde u|^{\tilde p(y,\tau)}\,\mathrm {d}y\mathrm {d}\tau
&\leq \lambda^{p_2-p_1-2}\int_{Q_{2\rho}^{(\lambda)}(z_0)}|Du|^{p(z)}\,\mathrm {d}z
\leq c(n,E)\frac{E}{\lambda^2}\leq c_1(n,E),\end{split}\end{equation}
which is analogue to \eqref{E}.
Next, we check the continuity condition for the new exponent function $\tilde p$. From \eqref{pz1pz2} and \eqref{change of variable}$_4$,
we see that
\begin{equation}\begin{split}|\label{new p1p2}\tilde p(y_1,\tau_1)-\tilde p(y_2,\tau_2)|&\leq \omega_p\Big(\max\Big\{|y_1-y_2|,
\lambda^{\frac{2-p_0}{2}}\sqrt{|\tau_1-\tau_2|}\Big
\}\,
\Big)\\&\leq \omega_p(d_{\mathcal{P}}((y_1,\tau_1),(y_2,\tau_2))),\end{split}\end{equation}
since $p_0\geq2$ and $\lambda\geq1$. This shows that the exponent function $\tilde p$ satisfies \eqref{pz1pz2}.
Moreover, we consider the structural conditions for the vector field $\tilde A(y,\tau,w)$.
From Lemma \ref{lambda} (2), \eqref{A}$_2$ and \eqref{change of variable}$_2$, we obtain
\begin{equation}\begin{split}\label{new a1}|\tilde A(y,\tau,w)|&\leq \sqrt{L}\lambda^{1-p_0}(1+|\lambda w|^2)^{\frac{p(z)-1}{2}}
\\&\leq \sqrt{L}\lambda^{p(z)-p_0}
(1+|w|^2)^{\frac{p(z)-1}{2}}\leq L_1(1+|w|^2)^{\frac{\tilde p(y,\tau)-1}{2}},\end{split}\end{equation}
since $z=(\hat x+y,\hat t+\lambda^{2-p_0}\tau)$ and
$p(z)=\tilde p(y,\tau)$. Here, the constant $L_1$ depends only upon $L$, $n$, $E$, $\gamma_2$ and $M$.
This establishes the growth condition for $\tilde A$.
In view of \eqref{A}$_1$ and \eqref{change of variable}$_2$, we apply Lemma \ref{lambda} (2) to deduce that
\begin{equation}\begin{split}\label{new a2}\big\langle\partial_w\tilde A(y,\tau,w)\tilde w,\tilde w \big \rangle
&=\lambda^{2-p_0}\big\langle\partial_w A(\hat x+y,\hat t+\lambda^{2-p_0}\tau,\lambda w)\tilde w,\tilde w \big \rangle
\\&\geq \sqrt{\nu}\lambda^{p(z)-p_0}|w|^{p(z)-2}|\tilde w|^2\geq \nu_1|w|^{\tilde p(y,\tau)-2}|\tilde w|^2,
\end{split}\end{equation}
where the constant $\nu_1$ depends only upon $\nu$, $n$, $E$, $\gamma_2$ and $M$.
It follows from \eqref{new a2} that
\begin{equation}\begin{split}\label{new a3}
\big\langle\tilde A(y,\tau,w),w \big \rangle&
=\int_0^1\big\langle\partial_w\tilde A(y,\tau,sw)  w, w \big \rangle\,\mathrm{d}s+\big\langle\tilde A(y,\tau,0),w \big \rangle
\\&\geq \nu_1|w|^{p(z)}\int_0^1s^{p(z)-2}\,\mathrm{d}s-L_1|w|\geq c|w|^{\tilde p(y,\tau)}-c(L_1),
\end{split}\end{equation}
where we used Young's inequality in the last step. This shows the ellipticity condition for $\tilde A$.
Furthermore, we infer from \eqref{change of variable}$_3$ that $\sqrt{\nu}\leq\tilde\mu\leq\sqrt{L}$.
At this point, we observe from \eqref{new e}-\eqref{new a3} that the parabolic system \eqref{new parabolic} fulfills the desired
structure assumptions
and the results of \cite[Theorem 2.2]{BD} apply to weak solutions of \eqref{new parabolic}. More precisely,
there exist a constant $\delta_0>0$ and a radius $r_0>0$, depending only upon  $n$, $E$, $\gamma_2$, $\nu$, $L$ and $M$, such that
 for any $\hat r\leq r_0$ and $\delta\in(0,\delta_0]$
there holds
 \begin{equation}\label{reverse Holder inequality}
 \dashint_{Q_{\hat r/2}(0)}|D_y\tilde u|^{\tilde p(y,\tau)(1+\delta)}\,\mathrm {d}y\mathrm {d}\tau\leq c
   \left[\Big(\ \,\dashint_{Q_{\hat r}(0)}|D_y\tilde u|^{\tilde p(y,\tau)}\,\mathrm {d}y\mathrm {d}\tau
   \Big)^{1+\frac{1}{2}\delta p(\hat z)}+1\right],\end{equation}
   where the constant $c$ depends only upon  $n$, $E$, $\gamma_2$, $\nu$, $L$ and $M$.
Let $\delta\in (0,\delta_0]$ be a fixed constant, we now choose $\hat\rho_0=\hat\rho_0(\delta)>0$ small enough to have
 \begin{equation}\label{hatrho0}
\hat\rho_0\leq \min\{\tilde\rho_0,r_0\}\qquad\text{and}\qquad\omega_p(\hat\rho_0)\leq \frac{
\delta}{1+\frac{1}{2}\delta\gamma_2}.\end{equation}
 At this stage, we conclude from \eqref{hatrho0} that for any $\rho\leq\hat\rho_0$, there holds
$r=t-s\leq\frac{1}{2}\rho<r_0$ and $p_2<p_1(1+\delta)$. Therefore, the reverse H\"older inequality
\eqref{reverse Holder inequality} holds for $\hat r=r$. This gives
 \begin{equation*}
 \dashint_{Q_{r/2}(0)}|D_y\tilde u|^{p_1(1+\delta)}\,\mathrm {d}y\mathrm {d}\tau\leq c
   \left[\Big(\ \,\dashint_{Q_{r}(0)}|D_y\tilde u|^{p_2}\,\mathrm {d}y\mathrm {d}\tau
   \Big)^{1+\frac{1}{2}\delta p_2}+1\right].\end{equation*}
 Rescaling back to $u$, we deduce
   \begin{equation*}\begin{split}\dashint_{Q_{r/2}^{(\lambda)}(\hat z)}|Du|^{p_1(1+\delta)}
 \,\mathrm {d}z&\leq c\lambda^{\delta p_1-\frac{1}{2}\delta p_1p_2}
 \Big(\ \dashint_{Q_r^{(\lambda)}(\hat z)}|Du|^{p_2}\,\mathrm {d}z
   \Big)^{1+\frac{1}{2}\delta p_2}+c\lambda^{p_1(1+\delta)}.
   \end{split}\end{equation*}
From Lemma \ref{lambda}, we obtain
$\lambda\leq 1+|Dl|\leq 2\big(\,1+|(Du)_{Q_{\rho}^{(\lambda)}(z_0)}|\,\big)$ and therefore
\begin{equation*}\begin{split}\dashint_{Q_{r/2}^{(\lambda)}(\hat z)}|Du|^{p_1(1+\delta)}
 \,\mathrm {d}z
   &\leq c\lambda^{\delta p_1-\frac{1}{2}\delta p_1p_2}\Big(\ \dashint_{Q_r^{(\lambda)}(\hat z)}|Du|^{p_2}\,\mathrm {d}z
   \Big)^{1+\frac{1}{2}\delta p_2}+c\big(\,1+|(Du)_{Q_{\rho}^{(\lambda)}(z_0)}|\,\big)^{p_1(1+\delta)}.
   \end{split}\end{equation*}
In view of $p_1\leq p_2\leq p_1(1+\delta)$,
the interpolation inequality allows us to conclude that
 \begin{equation*}\begin{split}\dashint_{Q_{r/2}^{(\lambda)}(\hat z)}&|Du|^{p_1(1+\delta)}
 \,\mathrm {d}z\\ &\leq  c\Big(\ \dashint_{Q_r^{(\lambda)}(\hat z)}|Du|^{p_1(1+\delta)}\,\mathrm {d}z
   \Big)^{\frac{p_2-p_1}{\delta p_1}\left(1+\delta\frac{p_2}{2}\right)}
   \lambda^{\delta p_1-\frac{1}{2}\delta p_1p_2}\Big(\ \dashint_{Q_r^{(\lambda)}(\hat z)}|Du|^{p_1}\,\mathrm {d}z
   \Big)^{\left(1-\frac{p_2-p_1}{\delta p_1}\right)\left(1+\delta\frac{p_2}{2}\right)}
   \\&\quad +c|(Du)_{Q_{\rho}^{(\lambda)}(z_0)}|^{p_1(1+\delta)}+c,
 \end{split}\end{equation*}
 since
 \begin{equation*}\frac{1}{p_2}=\frac{\Theta}{p_1(1+\delta)}+\frac{1-\Theta}{p_1},
 \qquad\text{where}\qquad \Theta=\frac{p_2-p_1}{p_2}\frac{\delta+1}{\delta}.\end{equation*}
Furthermore, from \eqref{hatrho0}, we see that
 \begin{equation*}\sigma:=\frac{1}{\frac{p_2-p_1}{\delta p_1}(1+\delta\frac{p_2}{2})}>1.\end{equation*}
This enables us to use the Young's inequality with $\sigma$ and $\sigma/(\sigma-1)$ to obtain
 \begin{equation*}\begin{split}\int_{Q_{r/2}^{(\lambda)}(\hat z)}|Du|^{p_1(1+\delta)}
 \,\mathrm {d}z&\leq  \kappa\ \int_{Q_r^{(\lambda)}(\hat z)}|Du|^{p_1(1+\delta)}\,\mathrm {d}z
   \\&+\frac{c(\kappa)}{|Q_r^{(\lambda)}(\hat z)|^{\hat \beta(\delta
   )}}\lambda^{\frac{1}{1-\frac{p_2-p_1}{\delta p_1}\left(1+\frac{\delta p_2}{2}\right)}(\delta p_1-\frac{1}{2}\delta p_1p_2)}
   \Big(\ \int_{Q_r^{(\lambda)}(\hat z)}|Du|^{p_1}\,\mathrm {d}z
   \Big)^{1+\hat \beta(\delta)}
   \\&\quad +c|Q_{r/2}^{(\lambda)}(\hat z)|
   |(Du)_{Q_{\rho}^{(\lambda)}(z_0)}|^{p_1(1+\delta)}+c|Q_{r/2}^{(\lambda)}(\hat z)|,
 \end{split}\end{equation*}
 where $\kappa\leq1$ and
  \begin{equation*}\hat\beta(\delta)=\frac{\frac{\delta p_2}{2}}{1-\frac{p_2-p_1}{\delta p_1}(1+\frac{\delta p_2}{2})}.\end{equation*}

  Step 2: Proof of \eqref{higher integrability1} by using a covering argument. To start with,
  we cover the set $Q_s^{(\lambda)}(z_0)$ by a finite number of the
  uniform parabolic cylinders $\{Q_{r/2}^{(\lambda)}(\hat z_i)\}_{i=1}^M$ of the type
  \eqref{new cylinder}
  such that only a finite number
  of cylinders $Q_{r/2}^{(\lambda)}(\hat z_i)$ intersect.
    Therefore, we find that for any $\frac{1}{2}\rho<s<t<\rho$ there holds
\begin{equation*}\begin{split}
\int_{Q_s^{(\lambda)}(z_0)}|Du|^{p_1(1+\delta)}
 \,\mathrm {d}z\leq & \kappa N\ \int_{Q_t^{(\lambda)}(z_0)}|Du|^{p_1(1+\delta)}\,\mathrm {d}z
   \\&+\frac{c(\kappa, N)}{(\lambda^{2-p_0}(t-s)^{n+2})^{\hat\beta(\delta
   )}}\lambda^{\frac{1}{1-\frac{p_2-p_1}{\delta p_1}\left(1+\frac{\delta p_2}{2}\right)}(\delta p_1-\frac{1}{2}
   \delta p_1p_2)}\Big(\ \int_{Q_\rho^{(\lambda)}(z_0)}|Du|^{p_1}\,\mathrm {d}z
   \Big)^{1+\hat\beta(\delta)}
   \\&+cN|Q_{\rho}^{(\lambda)}(z_0)|
   |(Du)_{Q_{\rho}^{(\lambda)}(z_0)}|^{p_1(1+\delta)}+c|Q_{\rho}^{(\lambda)}(z_0)|.
\end{split}\end{equation*}
At this stage, we can use an iteration lemma (cf. \cite[Lemma 2.2]{DMS})
to re-absorb the first integral of the right-hand side into the left.
This gives
\begin{equation*}\begin{split}
\int_{Q_{\rho/2}^{(\lambda)}(z_0)}|Du|^{p_1(1+\delta)}
 \,\mathrm {d}z&\leq  \lambda^{\frac{1}{1-\frac{p_2-p_1}{\delta p_1}\left(1+\frac{\delta p_2}{2}\right)}
 (\delta p_1-\frac{1}{2}\delta p_1p_2)}\frac{c(\kappa, N)}{|Q_\rho^{(\lambda)}(z_0)|^{\hat \beta(\delta
   )}}\Big(\ \int_{Q_\rho^{(\lambda)}(z_0)}|Du|^{p_1}\,\mathrm {d}z
   \Big)^{1+\hat\beta(\delta)}
   \\&+c|Q_{\rho}^{(\lambda)}(z_0)|
   |(Du)_{Q_{\rho}^{(\lambda)}(z_0)}|^{p_1(1+\delta)}+c|Q_{\rho}^{(\lambda)}(z_0)|.
\end{split}\end{equation*}
Therefore, dividing by $|Q_{\rho}^{(\lambda)}(z_0)|$, we obtain
\begin{equation}\begin{split}\label{RHI}
\dashint_{Q_{\rho/2}^{(\lambda)}(z_0)}|Du|^{p_1(1+\delta)}
 \,\mathrm {d}z&\leq  c \lambda^{\frac{1}{1-\frac{p_2-p_1}{\delta p_1}\left(1+\frac{\delta p_2}{2}\right)}(\delta p_1-\frac{1}{2}\delta p_1p_2)}
 \Big(\ \dashint_{Q_\rho^{(\lambda)}(z_0)}|Du|^{p_1}\,\mathrm {d}z
   \Big)^{1+\hat\beta(\delta)}
   \\&+c
   |(Du)_{Q_{\rho}^{(\lambda)}(z_0)}|^{p_1(1+\delta)}+c.
\end{split}\end{equation}
On the other hand, we infer from the H\"older's inequality and Lemma \ref{lambda} (2) that
\begin{equation}\begin{split}\label{estimatestep}\dashint_{Q_\rho^{(\lambda)}(z_0)}|Du|^{p_1}\,\mathrm {d}z\leq
\Big(\ \dashint_{Q_\rho^{(\lambda)}(z_0)}|Du|^{p_0}\,\mathrm {d}z\Big)^{\frac{p_1}{p_0}}\leq c(1+|Dl|)^{p_1}\leq c(M,\gamma_2)\lambda^{p_1},
\end{split}\end{equation}
since $1+|Dl|\leq M\lambda$.
In view of
$\lambda\leq 1+|Dl|\leq 2\big(\,1+|(Du)_{Q_{\rho}^{(\lambda)}(z_0)}|\,\big)$ , we conclude from \eqref{estimatestep} that
\begin{equation}\begin{split}\label{estimatestep1}
&\lambda^{\frac{1}{1-\frac{p_2-p_1}{\delta p_1}\left(1+\frac{\delta p_2}{2}\right)}(\delta p_1-\frac{1}{2}\delta p_1p_2)}
 \Big(\ \dashint_{Q_\rho^{(\lambda)}(z_0)}|Du|^{p_1}\,\mathrm {d}z
   \Big)^{1+\hat\beta(\delta)}
\\&\quad\leq c(\delta,M,\gamma_2)\lambda^{\frac{1}{1-\frac{p_2-p_1}{\delta p_1}\left(1+\frac{\delta p_2}{2}\right)}
 (\delta p_1-\frac{1}{2}\delta p_1p_2)}
 \lambda^{p_1(1+\hat\beta(\delta))}
 \\&\quad= c(\delta,M,\gamma_2)\lambda^{p_1(1+\beta(\delta))}
 \leq  c\left[\Big(\ \,\dashint_{Q_{\rho}^{(\lambda)}(z_0)}|Du|^{p_1}\,\mathrm {d}z\Big)^{1+\beta(\delta)}+1\right],
   \end{split}\end{equation}
 where $\beta(\delta)$ satisfies \eqref{betadelta}.
 Plugging \eqref{estimatestep1} into \eqref{RHI} and taking into account
$\delta\leq \beta(\delta)$, we obtain the desired estimate \eqref{higher integrability1}. The proof of the lemma is now complete.
\end{proof}
\begin{remark}\label{first remark}
From the proof of Lemma \ref{fitst higher}, we conclude that for any fixed $\delta\in(0,\delta_0]$ the radius
$\hat\rho_0$ can be chosen according to \eqref{hatrho0} in order to obtain the higher integrability estimate \eqref{higher integrability1}.
\end{remark}
With the help of the preceding two lemmas, we can now prove the following extension of \cite[Lemma 2.9]{H}
to the parabolic case.
\begin{lemma}\label{importanthigher}Let $\theta\in (0,\frac{1}{2}]$, $M\geq1$ and $\gamma\geq1$ be fixed.
Suppose that $\Phi^{(\lambda)}(z_0,\rho,l)\leq \frac{1}{16}$, $\lambda\geq1$ and $\lambda\leq 1+|Dl|\leq M\lambda$.
Then, there exists a radius $\rho_0=\rho_0(n,E,\gamma_2,\gamma,\nu,L,M)>0$ such that
 for any $\rho\leq\rho_0$ and $q\in [p_1,p_1+6\omega_p(2\rho)]$,
there holds
\begin{equation}\label{higher integrability2}\dashint_{Q_{\theta \rho}^{(\lambda)}(z_0)}(1+|Dl|+|Du|)^q
\log^\gamma(e+|Dl|+|Du|)\,\mathrm {d}z\leq c\log^\gamma\left(\frac{1}{\rho}\right)(1+|Dl|)^{p_0},
\end{equation}
where the constant $c$ depends only upon  $n$, $E$, $\gamma_2$, $\gamma$, $\nu$, $L$, $M$ and $\theta$. Here,
$p_0=p(z_0)$ and $p_1$ is the minimum of $p(z)$ on $Q_{\rho}^{(\lambda)}(z_0)$.
\end{lemma}
\begin{proof}We prove the estimate \eqref{higher integrability2} by decomposing the integral of the left-hand side into two parts.
To this aim,
we decompose
$Q_{\theta \rho}^{(\lambda)}(z_0)=Q^+\cup Q^-$, where
\begin{equation*}\begin{split}Q^+&=\left\{z\in Q_{\theta \rho}^{(\lambda)}(z_0):\ |Du|\geq 1+|Dl|\right\}
\quad\text{and}\quad Q^-=Q_{\theta \rho}^{(\lambda)}(z_0)\setminus Q^+.
\end{split}\end{equation*}
We first consider the estimate on the set $Q^-$.
From Lemma \ref{lambda} (4) and \eqref{exp p2-p1}, we conclude that
\begin{equation*}\begin{split}\label{I-}I_-:&=\frac{1}{|Q_{\theta \rho}^{(\lambda)}(z_0)|}
\int_{Q^-}(1+|Dl|+|Du|)^q
\log^\gamma(e+|Dl|+|Du|)\,\mathrm {d}z\\&\leq c(1+|Dl|)^q\log^\gamma\big[2(1+|Dl|)\big]
\\&\leq c(1+|Dl|)^{p_0}\Big(\,\rho^{-(n+2)\frac{\gamma_2}{2}}\Big)^{q-p_0}\log^\gamma\Big[\,c\rho^{-(n+2)\frac{\gamma_2}{2}}\Big]
\\&\leq c\log^\gamma\left(\frac{1}{\rho}\right)(1+|Dl|)^{p_0},
\end{split}\end{equation*}
since $q\in [p_1,p_1+6\omega_p(2\rho)]$. Moreover, we estimate the left-hand side of \eqref{higher integrability2}
on the set $Q^+$ and obtain a decomposition as follows:
\begin{equation*}\begin{split}\label{I+12}I_+:&=\frac{1}{|Q_{\theta \rho}^{(\lambda)}(z_0)|}
\int_{Q^+}(1+|Dl|+|Du|)^q
\log^\gamma(e+|Dl|+|Du|)\,\mathrm {d}z
\\&\leq c\frac{1}{|Q_{\theta \rho}^{(\lambda)}(z_0)|}
\int_{Q_{\theta \rho}^{(\lambda)}(z_0)}|Du|^q
\log^\gamma(e+|Du|)\,\mathrm {d}z
\\&\leq c\ \dashint_{Q_{\theta \rho}^{(\lambda)}(z_0)}|Du|^q
\log^\gamma\Big(e+\frac{|Du|^q}{(|Du|^q)_{Q_{\theta\rho}^{(\lambda)}(z_0)}}\Big)\,\mathrm {d}z
\\&\quad+c\log^\gamma\Big(e+(|Du|^q)_{Q_{\theta\rho}^{(\lambda)}(z_0)}\Big)\ \dashint_{Q_{\theta \rho}^{(\lambda)}(z_0)}|Du|^q
\,\mathrm {d}z
=:I_+^{(1)}+I_+^{(2)},
\end{split}\end{equation*}
with the obvious meaning of $I_+^{(1)}$ and $I_+^{(2)}$. Next, we consider the estimate for $I_+^{(1)}$.
To start with,
we apply Proposition \ref{LlogL} with $(m,p,\sigma,f)$ replaced by $(n+1,q,\delta_0,|Du|)$, where $\delta_0$ is the exponent in Lemma
\ref{fitst higher}.
This gives
\begin{equation*}\begin{split}
I_+^{(1)}&\leq c\Big(\ \dashint_{Q_{\theta \rho}^{(\lambda)}(z_0)}|Du|^{q(1+\delta_0)}
\,\mathrm {d}z\Big)^{\frac{1}{1+\delta_0}}
\\&\leq c(\theta)\Big(\ \dashint_{Q_{\rho/2}^{(\lambda)}(z_0)}|Du|^{p_1\left(1+\delta_0+3\omega_p(2\rho)(1+\delta_0)\right)}
\,\mathrm {d}z+1\Big)^{\frac{1}{1+\delta_0}},
\end{split}\end{equation*}
since $\theta\leq\frac{1}{2}$.
At this point, let $r_0=r_0(n,E,\gamma_2,\nu,L,M)$ be the radius considered in the proof of Lemma \ref{fitst higher} and hence
the radius $\rho_0$ can be
determined a priori only in terms of $\epsilon_0$, $\delta_0$, $\tilde\rho_0$ and $r_0$, such that
\begin{equation}\label{rho1}\rho_0\leq\min\{\tilde \rho_0,r_0\}\qquad\text{and}\qquad \omega_p(2\rho_0)\leq\min\left\{\frac{\delta_0}
{1+13\delta_0\gamma_2},\frac{\epsilon_0}{3}\right\}.\end{equation}
From Remark \ref{first remark},
the conditions for
Lemma \ref{higher integrability} are satisfied with $\delta=\delta_0+3\omega_p(2\rho)(1+\delta_0)$.
We now apply Lemma \ref{higher integrability}
to conclude that for any $\rho\leq\rho_0$, there holds
\begin{equation*}\begin{split}
I_+^{(1)}\leq c
\Big(\ \dashint_{Q_{ \rho}^{(\lambda)}(z_0)}|Du|^{p_1}
\,\mathrm {d}z\Big)^{\frac{1+\beta(\delta)}{1+\delta_0}}+c,
\end{split}\end{equation*}
where the constant $c$ depends only upon  $n$, $E$, $\gamma_2$, $\gamma$, $\nu$, $L$, $M$ and $\theta$.
According to Lemma \ref{lambda} (2), we use the H\"older's inequality and \eqref{betadelta} to deduce
\begin{equation}\begin{split}\label{ddd}
\Big(\ \dashint_{Q_{ \rho}^{(\lambda)}(z_0)}|Du|^{p_1}
\,\mathrm {d}z\Big)^{\frac{1+\beta(\delta)}{1+\delta_0}-1}\leq c(M)(1+|Dl|)^{c\omega_p(2\rho)}\leq c
\rho^{-c\omega_p(\rho)}\leq c,\end{split}\end{equation}
where the constant $c$ depends only upon  $n$, $E$, $\gamma_2$ and $M$. Finally,
we use \eqref{ddd}, H\"older's inequality and Lemma \ref{lambda} (2) to obtain
\begin{equation}\begin{split}\label{I+1}
I_+^{(1)}\leq c
\Big(\ \dashint_{Q_{ \rho}^{(\lambda)}(z_0)}|Du|^{p_1}
\,\mathrm {d}z\Big)+c\leq c(1+|Dl|)^{p_1}
\leq c(1+|Dl|)^{p_0}.
\end{split}\end{equation}
To estimate $I_+^{(2)}$, we infer from \eqref{rho1} that the inequality $q\leq p_1+6\omega_p(2\rho)\leq p_1(1+\epsilon_0)$ holds.
In view of $p_0\geq2$ and $\lambda\leq 1+|Dl|$, we conclude from \eqref{E},
Proposition \ref{higher integrability0} and Lemma \ref{lambda} (4) that
\begin{equation*}\begin{split}
(|Du|^q)_{Q_{\theta\rho}^{(\lambda)}(z_0)}&\leq c(\theta)\ \dashint_{Q_{\rho/2}^{(\lambda)}(z_0)}|Du|^{q}
\,\mathrm {d}z
\\&
\leq c\lambda^{p_0-2}
   \left[\rho^{-n-2}\Big(\ \,\int_{Q_{\rho}}|Du|^{p(z)}\,\mathrm {d}z\Big)^{1+\frac{1}{2}\epsilon p_0}+1\right]
   \\&
\leq c(E)\rho^{-\gamma_2(\gamma_2-2)\frac{n+2}{2}-n-2}.
\end{split}\end{equation*}
By the H\"older's inequality, we adopt the same procedure as in the proof of \eqref{I+1}. This yields the estimate
\begin{equation*}\begin{split}\label{I+2}I_+^{(2)}\leq c\log^\gamma\left(\frac{1}{\rho}\right)\ \dashint_{Q_{\theta \rho}^{(\lambda)}(z_0)}|Du|^q
\,\mathrm {d}z\leq
c\log^\gamma\left(\frac{1}{\rho}\right)(1+|Dl|)^{p_0},
\end{split}\end{equation*}
for a constant $c$ depending only on the data $n$, $E$, $\gamma_2$, $\gamma$, $\nu$, $L$, $M$ and $\theta$.
which proves the lemma. Combining the estimates for $I_-$, $I_+^{(1)}$ and $I_+^{(2)}$, we obtain the desired estimate
\eqref{higher integrability2}.
\end{proof}
\section{Poincar\'e type inequality}
The aim of this section is to derive some smallness conditions for weak solutions near the regular points.
These smallness conditions will be the starting point for the proof of the main result in Sect. 7.
In order to reduce the alternative characterization of the regular
points, we will need a Poincar\'e type inequality for the weak solutions.
To achieve this, we first prove the following gluing lemma, which concerns weighted mean
values of weak solutions on different time slices..
\begin{lemma}\label{first poincare} Let $z_0\in\Omega_T$ and let $l$ be an affine function. Let $\rho_0>0$ be the radius in
Lemma \ref{importanthigher}.
Assume that $\rho<\rho_0$ and $\Phi(z_0,2\rho,l)\leq \frac{1}{16}$.
Then, for any $r$, $s\in\Lambda_\rho(t_0)$, there exists a constant $c$ depending only upon $n$, $N$, $E$, $\gamma_2$, $\nu$ and $L$
such that the inequality
\begin{equation}
\begin{split}\label{poincare1}\Big|&\int_{B_\rho(x_0)}(u_i(\cdot,r)-u_i(\cdot,s))\psi(x)\,\mathrm {d}x\Big|
\\&\leq c|r-s|^{\frac{1}{p_0}}\|D\psi\|_{L^{p_0}}|Q_\rho|^{\frac{p_0-1}{p_0}}\Big(\ \dashint_{Q_\rho(z_0)}|Du-Dl|^{p_0}
\,\mathrm {d}z\Big)^{\frac{p_0-1}{p_0}}
\\&\quad+c|r-s|^{\frac{1}{p_0}}(1+|Dl|)^{p_0-2}
\|D\psi\|_{L^{p_0}}|Q_\rho|^{\frac{p_0-1}{p_0}}\Big(\ \dashint_{Q_\rho(z_0)}|Du-Dl|^{p_0}
\,\mathrm {d}z\Big)^{\frac{1}{p_0}}
\\&\quad+c\left(\omega_p(\rho)\log\left(\frac{1}{\rho}\right)+v(\rho)\right)^{\frac{p_0-1}{p_0}}
|r-s|^{\frac{1}{p_0}}(1+|Dl|)^{p_0-1}
\|D\psi\|_{L^{p_0}}|Q_\rho|^{\frac{p_0-1}{p_0}}
\end{split}
\end{equation}
holds,
where $\psi\in C_0^\infty(B_\rho(x_0))$, $i\in \{1,\cdots,N\}$ and $p_0=p(z_0)$.
\end{lemma}
\begin{proof}
Without loss of generality, we may assume that $z_0=0$.
Let $i\in \{1,\cdots,N\}$ and $0<h<\frac{1}{2}(r-s)$. In the weak formulation \eqref{weaksolution} we choose the test function
\begin{equation*}\varphi_h(x,t)=\zeta_h(t)\psi(x)e_i,\end{equation*}
where $\psi\in C_0^\infty(B_\rho)$ and $\zeta_h$ is a Lipschitz function
given by
 \begin{equation}\label{zeta}
	\zeta_h(t)=\begin{cases}
	\ 0,\quad &t\leq s,\\
	\ \frac{1}{h}(t-s),\quad &s<t<s+h,\\
\ -\frac{1}{h}(t-r),\quad &r-h<t<r,\\
\ 0,&t\geq r,
	\end{cases}
\end{equation}
where $-\rho^2<s<r<0$.
For the vector field $A=(A^1,\cdots ,A^N)\in \mathbb{R}^{n\times N}$, we deduce
\begin{equation*}\int_{-\rho^2}^0\int_{B_\rho} u_i\cdot\partial_t \zeta_h\psi\,\mathrm {d}x\mathrm {d}t=
\int_{-\rho^2}^0\int_{B_\rho}\zeta_h\mu(z) A^i(z,Du)\cdot D\psi
\,\mathrm {d}x\mathrm {d}t.\end{equation*}
From the definition of $\zeta_h$, we pass to the limit $h\downarrow 0$ to infer that
\begin{equation*}\int_{B_\rho}(u_i(\cdot,r)-u_i(\cdot,s))\psi(x)\,\mathrm {d}x=
\int_s^r\int_{B_\rho}\mu(z) A^i(z,Du)\cdot D\psi\,\mathrm {d}x\mathrm {d}t.
\end{equation*}
We now proceed to estimate the integral with respect to space from the right-hand side of the above equation by
\begin{equation*}\begin{split}
\int_{B_\rho\times\{t\}}&\mu(z) A^i(z,Du)\cdot D\psi\,\mathrm {d}x
\\&=\int_{B_\rho\times\{t\}}\mu(z) [A^i(z,Du)-A^i(z,Dl)]\cdot D\psi\,\mathrm {d}x
\\&\quad+\int_{B_\rho\times\{t\}}[\mu(z)-(\mu)_{Q_\rho}] A^i(z,Dl)\cdot D\psi\,\mathrm {d}x
\\&\quad+\int_{B_\rho\times\{t\}}(\mu)_{Q_\rho}[ A^i(z,Dl)-A^i(0,Dl)]\cdot D\psi\,\mathrm {d}x
\\&=:L_1+L_2+L_3,
\end{split}
\end{equation*}
with the obvious meaning of $L_1$, $L_2$ and $L_3$.
At this point, we define
$p_1=\inf _{Q_{\rho}}p(z)$ and $p_2=\sup _{Q_{\rho}}p(z)$.
By the fundamental theorem of calculus, we infer from \eqref{A}$_2$ that
\begin{equation*}\begin{split}
|L_1|&\leq\sqrt{L}\Big|\int_{B_\rho\times\{t\}}\int_0^1\partial_w A^i(z,Dl+s(Du-Dl))\cdot(Du-Dl)
\,\mathrm{d}s\cdot D\psi\,\mathrm {d}x\Big|
\\&\leq \sqrt{L}\int_{B_\rho\times\{t\}}(1+|Dl|+|Du-Dl|)^{p_2-2}|Du-Dl||D\psi|\,\mathrm {d}x
\\&\leq c\int_{B_\rho\times\{t\}}(1+|Dl|)^{p_2-2}|Du-Dl||D\psi|\,\mathrm {d}x
\\&\qquad +c\int_{B_\rho\times\{t\}}|Du-Dl|^{p_2-1}|D\psi|\,\mathrm {d}x
=:L_1^{(1)}+L_1^{(2)},
\end{split}
\end{equation*}
since $p_2\geq2$.
We note that the choice of $\rho<\rho_0$ allows us to use
Lemma \ref{lambda} with $\lambda=1$. Now we come to the estimate of $L_1^{(1)}$.
By Lemma \ref{lambda} (3) and
the H\"older's inequality, we obtain
\begin{equation*}\begin{split}
\int_s^rL_1^{(1)}\,\mathrm {d}t
&\leq c(1+|Dl|)^{p_0-2}|r-s|^{\frac{1}{p_0}}\|D\psi\|_{L^{p_0}}
\Big(\int_s^r \int_{B_\rho}|Du-Dl|^{\frac{p_0}{p_0-1}}
\,\mathrm {d}z\Big)^{\frac{p_0-1}{p_0}}\\&
\leq c(1+|Dl|)^{p_0-2}|r-s|^{\frac{1}{p_0}}\|D\psi\|_{L^{p_0}}
|Q_\rho|^{\frac{p_0-1}{p_0}}\Big(\ \dashint_{Q_\rho}|Du-Dl|^{p_0}
\,\mathrm {d}z\Big)^{\frac{1}{p_0}},
\end{split}
\end{equation*}
since $p_0\geq2$. To estimate $L_1^{(2)}$,
we use the H\"older's inequality to deduce
\begin{equation*}\begin{split}
L_1^{(2)}\leq c\|D\psi\|_{L^{p_0}}\Big(\int_{B_\rho\times\{t\}}|Du-Dl|^{\frac{p_0}
{p_0-1}(p_2-1)}\,\mathrm {d}x\Big)^{\frac{p_0-1}{p_0}}.
\end{split}\end{equation*}
In the case $|Du-Dl|\geq1$, we use the fundamental theorem of calculus and \eqref{p2-p1} to obtain
\begin{equation}\begin{split}\label{differenceDu}
|Du&-Dl|^{\frac{p_0}
{p_0-1}(p_2-1)}-|Du-Dl|^{p_0}
\\&\leq |Du-Dl|^{\frac{p_0}
{p_0-1}(p_2-1)}\Big(p_0\frac{p_2-p_0}{p_0-1}\Big)\log(|Du-Dl|)
\\&\leq 2\omega_p(\rho)|Du-Dl|^{\frac{p_0}
{p_0-1}(p_2-1)}\log(1+|Du-Dl|),
\end{split}\end{equation}
since $p_0\geq2$.
In the case $|Du-Dl|<1$, we observe that
$|Du-Dl|^{\frac{p_0}
{p_0-1}(p_2-1)}\leq |Du-Dl|^{p_0}.$ Using this together with
\eqref{differenceDu}, we conclude that
\begin{equation*}\begin{split}L_1^{(2)}\leq &c\|D\psi\|_{L^{p_0}}\Big(\int_{B_\rho\times\{t\}}|Du-Dl|^{p_0
}\,\mathrm {d}x\Big)^{\frac{p_0-1}{p_0}}
\\&+c\|D\psi\|_{L^{p_0}}\Big(\omega_p(\rho)\int_{B_\rho\times\{t\}}|Du-Dl|^{\frac{p_0}
{p_0-1}(p_2-1)}\log(1+|Du-Dl|)\,\mathrm {d}x\Big)^{\frac{p_0-1}{p_0}}.
\end{split}\end{equation*}
Noting that $p_0\frac{p_2-1}{p_0-1}\leq p_1+2\omega_p(\rho)$, we can apply Lemma \ref{importanthigher} with
$\theta=\frac{1}{2}$, $\gamma=1$ and $\lambda=1$. Therefore, we apply Lemma \ref{importanthigher} and
the H\"older's inequality to obtain
\begin{equation*}\begin{split}\int_s^rL_1^{(2)}\,\mathrm {d}t&\leq c|r-s|^{\frac{1}{p_0}}\|D\psi\|_{L^{p_0}}
|Q_\rho|^{\frac{p_0-1}{p_0}}\Big(\ \dashint_{Q_\rho}|Du-Dl|^{p_0}
\,\mathrm {d}z\Big)^{\frac{p_0-1}{p_0}}
\\&\quad+c\|D\psi\|_{L^{p_0}}(r-s)^{\frac{1}{p_0}}\Big(\int_{Q_\rho}
\omega_p(\rho)|Du-Dl|^{p_0\frac{p_2-1}{p_0-1}}\log(1+|Du|+|Dl|)
\,\mathrm {d}z\Big)^{\frac{p_0-1}{p_0}}
\\&\leq c|r-s|^{\frac{1}{p_0}}\|D\psi\|_{L^{p_0}}
|Q_\sigma|^{\frac{p_0-1}{p_0}}\Big(\ \dashint_{Q_\rho}|Du-Dl|^{p_0}
\,\mathrm {d}z\Big)^{\frac{p_0-1}{p_0}}
\\&\quad+c\|D\psi\|_{L^{p_0}}|r-s|^{\frac{1}{p_0}}|Q_\rho|^{\frac{p_0-1}
{p_0}}(1+|Dl|)^{p_0-1}\Big[\omega_p(\rho)\log\Big(\frac{1}{\rho}\Big)\Big]
^{\frac{p_0-1}{p_0}}.
\end{split}\end{equation*}
Consequently, we infer that
\begin{equation}\begin{split}\label{L1}
\int_s^rL_1\,\mathrm {d}t&\leq c(1+|Dl|)^{p_0-2}|r-s|^{\frac{1}{p_0}}|Q_\rho|^{\frac{p_0-1}{p_0}}
\Big(\ \dashint_{Q_\rho}|Du-Dl|^{p_0}\,\mathrm {d}z\Big)^{\frac{1}{p_0}}
\|D\psi\|_{L^{p_0}}
\\&\quad+c|r-s|^{\frac{1}{p_0}}\|D\psi\|_{L^{p_0}}
|Q_\rho|^{\frac{p_0-1}{p_0}}\Big(\ \dashint_{Q_\rho}|Du-Dl|^{p_0}
\,\mathrm {d}z\Big)^{\frac{p_0-1}{p_0}}
\\&\quad+c\|D\psi\|_{L^{p_0}}|r-s|^{\frac{1}{p_0}}|Q_\rho|^{\frac{p_0-1}
{p_0}}(1+|Dl|)^{p_0-1}\Big[\omega_p(\rho)\log\Big(\frac{1}{\rho}\Big)\Big]
^{\frac{p_0-1}{p_0}}.
\end{split}\end{equation}
We now turn our
attention to the estimate of $L_2$. Since $\sqrt{\nu}\leq \mu(z)\leq \sqrt{L}$, we conclude from
Lemma \ref{lambda} (3) and the growth condition \eqref{A}$_2$ that
\begin{equation}\begin{split}\label{L2}
\int_s^rL_2\,\mathrm {d}t&=\int_s^r\int_{B_\rho}[\mu(z)-(\mu)_{Q_\rho}] A^i(z,Dl)\cdot D\psi\,\mathrm {d}x\mathrm {d}t
\\&\leq \sqrt{L}(1+|Dl|)^{p_2-1}\int_s^r\int_{B_\rho}|\mu(z)-(\mu)_{Q_\rho}||D\psi|\,\mathrm {d}x\mathrm {d}t
\\&\leq c(1+|Dl|)^{p_0-1}(r-s)^{\frac{1}{p_0}}\|D\psi\|_{L^{p_0}}|Q_\rho|^{\frac{p_0-1}{p_0}}v(\rho)^{\frac{p_0-1}{p_0}}.
\end{split}\end{equation}
Finally, we consider the estimate for $L_3$. Recalling that $\sqrt{\nu}\leq \mu(z)\leq \sqrt{L}$, we infer from
Lemma \ref{lambda} (3), (4) and the continuity condition \eqref{Az1z2} to obtain
\begin{equation}\begin{split}\label{L3}
\int_s^rL_3\,\mathrm {d}t&=\int_s^r\int_{B_\rho}(\mu)_{Q_\rho}[ A^i(z,Dl)-A^i(0,Dl)]\cdot D\psi\,\mathrm {d}x\mathrm {d}t
\\&\leq \sqrt{L}\int_s^r\int_{B_\rho}\omega_p(\rho)\Big[(1+|Dl|^2)^{\frac{p(z)-1}{2}}
+(1+|Dl|^2)^{\frac{p_0-1}{2}}\Big]|D\psi|\,\mathrm {d}x\mathrm {d}t
\\ &\qquad \times\Big[1+\log(1+|Dl|^2)\Big]
\\&\leq c(1+|Dl|)^{p_0-1}\omega_p(\rho)\log\Big(\frac{1}{\rho}\Big)
(r-s)^{\frac{1}{p_0}}\|D\psi\|_{L^{p_0}}|Q_\rho|^{\frac{p_0-1}{p_0}}.
\end{split}\end{equation}
Combining the estimates \eqref{L1}-\eqref{L3}, we obtain the desired estimate \eqref{poincare1}. This completes the proof.
\end{proof}
With the help of Lemma \ref{first poincare},
we are now ready for the proof of Poincar\'e inequality for weak solutions to the parabolic system \eqref{parabolic}.
\begin{lemma}\label{poincare lemma} Let $z_0\in\Omega_T$ and let $l$ be an affine function. Let $\rho_0>0$ be the radius in
Lemma \ref{importanthigher}.
Assume that $\rho<\rho_0$ and $\Phi(z_0,2\rho,l)\leq \frac{1}{16}$.
Then, there exists a constant $c$ depending only upon $n$, $E$, $\gamma_2$, $\nu$ and $L$
such that the inequality
\begin{equation}\begin{split}\label{Poincare2}&\dashint_{Q_\rho(z_0)}\Big|\frac{u-(u)_{Q_\rho(z_0)}-Dl\cdot (x-x_0)}{\rho}\Big|^2
+\Big|\frac{u-(u)_{Q_\rho(z_0)}-Dl\cdot (x-x_0)}{\rho}\Big|^{p_0}\,\mathrm {d}z
\\&\leq c(1+|Dl|)^{p_0(p_0-2)}\left[\ \dashint_{Q_{\rho}(z_0)}|Du-Dl|^{p_0}\,\mathrm {d}z+\Big(\ \dashint_{Q_{\rho}(z_0)
}|Du-Dl|^{p_0}\,\mathrm {d}z
\Big)^{p_0-1}\right]
\\&\quad +c(1+|Dl|)^{p_0(p_0-2)}\left[\ \Big(\ \dashint_{Q_{\rho}(z_0)}|Du-Dl|^{p_0}\,\mathrm {d}z\Big)^{\frac{2}{p_0}(p_0-1)}
+\Big(\ \dashint_{Q_{\rho}(z_0)}|Du-Dl|^{p_0}\,\mathrm {d}z
\Big)^{\frac{2}{p_0}}\right]
\\&\quad +c(1+|Dl|)^{p_0(p_0-1)}\left[v^{\frac{2(p_0-1)}{p_0}}(\rho)+\Big(\omega_p(\rho)\log\Big(\frac{1}{\rho}\Big)\Big)
^{\frac{2(p_0-1)}{p_0}}\right]
\end{split}\end{equation}
holds, where $p_0=p(z_0)$.
\end{lemma}
\begin{proof}
Without loss of generality, we may assume that $z_0=0$.
We choose a radial function $\psi\in C_0^\infty(B_1)$
with $\int_{B_1}\psi(x)\,\mathrm{d}x=1$. Let $\psi_\rho(x)=\rho^{-n}\psi(x/\rho)$. For $t\in (-\rho^2,0)$, we define
some different types of mean values
\begin{equation*}\tilde u_\rho(t)=\ \dashint_{B_\rho}u(x,t)\,\mathrm{d}x,
\quad \tilde u^{\psi}(t)=\ \dashint_{B_\rho}u(x,t)\psi_\rho(x)\,\mathrm{d}x\quad\text{and}\quad
(u)^{\psi}=\ \dashint_{-\rho^2}^0\tilde u^{\psi}(t)\,\mathrm{d}t.
\end{equation*}
For any fixed $t\in(-\rho^2,0)$, we infer from Lemma \ref{first poincare} that
\begin{equation}\begin{split}\label{difference1}
|\tilde u^{\psi}(t)&-(u)^{\psi}|\leq\ \dashint_{-\rho^2}^0|\tilde u^{\psi}(t)-
\tilde u^{\psi}(s)|\,\mathrm{d}s
\\&\leq c\sum_{i=1}^N\ \dashint_{-\rho^2}^0\Big|\int_{B_\rho}(u_i(x,t)-u_i(x,s))\psi_\rho(x)\,\mathrm{d}x\Big|\,\mathrm{d}s
\\&\leq c\rho\Big(\ \dashint_{Q_\rho}|Du-Dl|^{p_0}\,\mathrm{d}z\Big)^{\frac{p_0-1}{p_0}}
+c\rho(1+|Dl|)^{p_0-2}\Big(\ \dashint_{Q_\rho}|Du-Dl|^{p_0}\,\mathrm{d}z\Big)^{\frac{1}{p_0}}
\\&\quad+c\rho(1+|Dl|)^{p_0-1}\Big(v^{\frac{p_0-1}{p_0}}(\rho)+\Big(\omega_p(\rho)\log\Big(\frac{1}{\rho}\Big)\Big)
^{\frac{p_0-1}{p_0}}\Big).
\end{split}\end{equation}
On the other hand, we apply the Poincar\'e inequality retrieved from \cite[Theorem 12.36]{Leoni} to infer that there exists a constant
$c$ depending only upon $n$ such that
\begin{equation}\begin{split}\label{independent of p}
\dashint_{B_\rho}|u(x,t)-\tilde u_\rho(t)-Dl\cdot x|^{p_0}
\,\mathrm{d}x
\leq c\rho^{p_0}\ \dashint_{B_{\rho}}|Du-Dl|^{p_0}\,\mathrm{d}x.
\end{split}\end{equation}
We emphasize that the constant $c$ in \eqref{independent of p} is independent of $p(z)$ and therefore the estimate \eqref{independent of p} is
suitable for our purpose. To this end,
we use \eqref{independent of p} slicewise to deduce
\begin{equation}\begin{split}\label{difference2}
|\tilde u^{\psi}(t)-\tilde u_\rho(t)|^{p_0}&=\Big|\int_{B_\rho}(u(x,t)-\tilde u_\rho(t)-Dl\cdot x)\psi_\rho(x)\,\mathrm{d}x\Big|^{p_0}
\\& \leq c(n)\rho^{p_0}\ \dashint_{B_{\rho}}|Du-Dl|^{p_0}\,\mathrm{d}x,
\end{split}\end{equation}
for any $t\in(-\rho^2,0)$.
Furthermore,
the inequality \eqref{difference2} also implies that
\begin{equation}\begin{split}\label{difference3}
|(u)^{\psi}-u_{Q_\rho}|^{p_0}\leq c(n)\rho^{p_0}\ \dashint_{Q_{\rho}}|Du-Dl|^{p_0}\,\mathrm{d}z.
\end{split}\end{equation}
Combining the estimates \eqref{difference1}-\eqref{difference3}, we conclude from the triangle inequality that
\begin{equation*}\begin{split}
&\dashint_{Q_\rho}
\Big|\frac{u-(u)_{Q_\rho}-Dl\cdot x}{\rho}\Big|^{p_0}\,\mathrm {d}z
\\&\leq c\ \dashint_{Q_\rho}
\Big|\frac{u-\tilde u_\rho(t)-Dl\cdot x}{\rho}\Big|^{p_0}\,\mathrm {d}z+c\rho^{-p_0}\sup_{-\rho^2<t<0}|\tilde u^{\psi}(t)-(u)^{\psi}|^{p_0}
\\&\quad + c\rho^{-p_0}\sup_{-\rho^2<t<0}|\tilde u^{\psi}(t)-\tilde u_\rho(t)|^{p_0}+c\rho^{-p_0}|(u)^{\psi}-u_{Q_\rho}|^{p_0}
\\&\leq c\ \dashint_{Q_{\rho}}|Du-Dl|^{p_0}\,\mathrm{d}x+c\Big(\ \dashint_{Q_\rho}|Du-Dl|^{p_0}\,\mathrm{d}z\Big)^{p_0-1}
\\&\quad+c(1+|Dl|)^{p_0(p_0-2)}\ \dashint_{Q_\rho}|Du-Dl|^{p_0}\,\mathrm{d}z
\\&\quad+c(1+|Dl|)^{p_0(p_0-1)}\Big(v^{p_0-1}(\rho)+\Big(\omega_p(\rho)\log\Big(\frac{1}{\rho}\Big)\Big)
^{p_0-1}\Big).
\end{split}\end{equation*}
Therefore,
we have proved the desired estimate \eqref{Poincare2}
by the H\"older's inequality. The proof of the lemma is now complete.
\end{proof}
Before giving the precise statement of the smallness conditions near the regular points, we introduce the excess functional
for weak solutions.
To this end, we
let $z_0\in\Omega_T$, $\rho>0$ and let $l$ be an affine function. Moreover, we assume that $1\leq \lambda\leq 1+|Dl|$
and the excess functional is defined by
\begin{equation}\label{psi}
\Psi_\lambda(z_0,\rho,l)=\ \dashint_{Q_\rho^{(\lambda)}(z_0)}
\Big|\frac{u-l}{\rho(1+|Dl|)}\Big|^2\,\mathrm {d}z+\ \dashint_{Q_\rho^{(\lambda)}(z_0)}
\Big|\frac{u-l}{\rho(1+|Dl|)}\Big|^{p_0}\,\mathrm {d}z,
\end{equation}
where $p_0=p(z_0)$. Specifically, if $\lambda=1$, then we simply write $\Psi(z_0,\rho,l)$
for $\Psi_1(z_0,\rho,l)$.
Similarly, if $\lambda=1$, then we simply write $l_{z_0,\ \rho}$
for $l_{z_0,\ \rho}^{(1)}$.
We are now in a position to give a new characterization for the regular points in terms of the excess functional and
the following proposition is our main result in this section.
\begin{proposition}\label{start} Let $\Sigma_1$ and $\Sigma_2$ be the sets defined in Theorem \ref{main}
and fix $\mathfrak z_0\in \Omega_T\setminus(\Sigma_1\cup\Sigma_2)$. Then there exist a universal constant $M_0>0$ and
a constant $\hat\epsilon=\hat\epsilon(n,N,\nu,L,E,\gamma_2,M_0)>0$
such that the following holds: Whenever $\epsilon_*<\hat\epsilon$ is a fixed constant, then
there exists a radius $\rho=\rho(n,N,\nu,L,E,\gamma_2,M_0,\epsilon_*)>0$ such that
for any $z_0\in Q_{\rho/8}(\mathfrak z_0)$ there holds
 \begin{equation}\label{smalless initial 0}
\begin{cases}\,&1+|Dl_{z_0,\ \rho}|\leq M_0,\\
\,&\Psi(z_0,\rho,l_{z_0,\ \rho})\leq\epsilon_*,\\
\,&\Phi(z_0,\rho,l_{z_0,\ \rho})\leq\frac{1}{16},
\\\,&v(\rho)+\omega_p(\rho)\log\Big(\frac{1}{\rho}\Big)\leq \epsilon_*
	\end{cases}
\end{equation}
and $\rho<\rho_0$ where $\rho_0>0$ is the radius in
Lemma \ref{importanthigher}.
\end{proposition}
\begin{proof}We first observe from the definition of $\Sigma_1$ and $\Sigma_2$ that
for $\mathfrak z_0\in\Omega_T\setminus(\Sigma_1\cup\Sigma_2)
$, there exist a constant $M_*>1$ and a sequence $\{\rho_i\}_{i=1}^\infty$ with $\rho_i\downarrow 0$ such that
 \begin{equation*}\begin{split}\lim_{i\to\infty}\ \dashint_{Q_{\rho_i}(\mathfrak z_0)}
  |Du-(Du)_{Q_{\rho_i}(\mathfrak z_0)}|\,\mathrm {d}z=0
\qquad\text{and}\qquad \limsup\limits_{i\to\infty}\left(|Du|^{p(\cdot)}\right)_{Q_{\rho_i}(\mathfrak z_0)}\leq M_*.\end{split}\end{equation*}
From this, we conclude that for any fixed $0<\epsilon_*<1$, then there exists a radius $\hat\rho>0$ such that
for any $\rho_j<\hat\rho$ there holds
  \begin{equation*}\begin{split}\dashint_{Q_{\rho_j}(\mathfrak z_0)}
  |Du-(Du)_{Q_{\rho_j}(\mathfrak z_0)}|\,\mathrm {d}z<2^{-n-3}\epsilon_*^{\frac{\gamma_2(1+\delta)-1}{\delta}}\qquad\text{and}\qquad
\dashint_{Q_{\rho_j}(\mathfrak z_0)}|Du|^{p(\cdot)}\,\mathrm{d}z\leq 2 M_*,\end{split}\end{equation*}
where
$\delta=\frac{1}{4}\epsilon_0$ and $\epsilon_0$ is the exponent defined in Proposition \ref{higher integrability0}.
At this stage, let $\rho=\rho_j/8<\hat\rho/8$ be a fixed radius which will be determined later.
For any fixed point $z_0\in Q_{\rho/8}(\mathfrak z_0)$, we infer that
 \begin{equation}\begin{split}\label{small1}\dashint_{Q_{4\rho}(z_0)}
  |Du-(Du)_{Q_{4\rho}(z_0)}|\,\mathrm {d}z\leq 2\ \dashint_{Q_{4\rho}(z_0)}
  |Du-(Du)_{Q_{\rho_j}(\mathfrak z_0)}|\,\mathrm {d}z\leq \epsilon_*^{\frac{\gamma_2(1+\delta)-1}{\delta}}
\end{split}\end{equation}
and
 \begin{equation}\begin{split}\label{small2}
\dashint_{Q_{4\rho}(z_0)}|Du|^{p(\cdot)}\,\mathrm{d}z\leq 2^{n+2}\ \dashint_{Q_{\rho_j}(\mathfrak z_0)}|Du|^{p(\cdot)}\,\mathrm{d}z\leq
2^{n+3}M_*.
\end{split}\end{equation}
We first consider the proof of \eqref{smalless initial 0}$_1$. To this end,
we choose a radial function $\psi\in C_0^\infty(B_\rho(x_0))$
with $|D\psi|\leq c\rho^{-1}$. Moreover, we define $u^{\psi}(t)=\ \dashint_{B_\rho}u(x,t)\psi(x)\,\mathrm{d}x$
and apply the Poincar\'e inequality slicewise to obtain
 \begin{equation*}\begin{split}
 \dashint_{Q_\rho(z_0)}\Big|\frac{u-(u)_{Q_\rho(z_0)}}{\rho}\Big|\,\mathrm{d}z\leq c\ \dashint_{Q_\rho(z_0)}
 |Du|\,\mathrm{d}z+c\rho^{-1}
 \sup_{t,\tau\in(t_0-\rho^2,t_0)}|\tilde u^\psi(t)-\tilde u^\psi(\tau)|.
 \end{split}\end{equation*}
Once again, we define
$p_1=\inf _{Q_{8\rho}(\mathfrak z_0)}p(z)$ and $p_2=\sup _{Q_{8\rho}(\mathfrak z_0)}p(z)$.
We now follow the proof of \cite[Lemma 5]{BFM} to deduce that for a.e. $t,\tau\in (t_0-\rho^2,t_0)$ there holds
  \begin{equation*}\begin{split}
  |u^\psi(t)-u^\psi(\tau)|&\leq \sqrt{L}\|D\psi\|_{L^\infty}\rho^2
 \ \dashint_{Q_{\rho}(z_0)}(1+|Du|)^{p(z)-1}\,\mathrm{d}z
 \\&\leq c\rho  \ \dashint_{Q_{4\rho}(z_0)}(1+|Du|)^{p_2-1}\,\mathrm{d}z
  \\&\leq c\rho  \ \dashint_{Q_{4\rho}(z_0)}(1+|Du|)^{p(z)}\,\mathrm{d}z\leq c\rho M_*,
  \end{split}\end{equation*}
  since $p_2-p_1\leq\omega_p(4\rho)\leq 1$.
Consequently, we conclude that there exists a constant $\hat c$ depending only upon
$n$, $N$, $\nu$, $L$ and $\gamma_2$ such that the inequality
   \begin{equation*}\begin{split}
\dashint_{Q_\rho(z_0)}\Big|\frac{u-(u)_{Q_\rho(z_0)}}{\rho}\Big|\,\mathrm{d}z\leq \hat c+\hat cM_*+\hat c
\ \dashint_{Q_{4\rho}(z_0)}|Du|^{p(z)}\,\mathrm{d}z\leq 3\hat cM_*
     \end{split}\end{equation*}
holds.
Recalling that $l_{z_0,\ \rho}$ is the unique affine function minimizing $l\to\ \dashint_{Q_\rho(z_0)}|u-l|^2\,\mathrm{d}z$, we
apply \cite[(2.1)]{BFM} to deduce that
  \begin{equation*}\begin{split}
  |Dl_{z_0,\ \rho}|&=\frac{n+2}{\rho^2}\Big|\ \dashint_{Q_\rho(z_0)}(u-(u)_{Q_\rho(z_0)})\otimes(x-x_0)\,\mathrm{d}z\Big|
  \\&\leq (n+2)\
\dashint_{Q_\rho(z_0)}\Big|\frac{u-(u)_{Q_\rho(z_0)}}{\rho}\Big|\,\mathrm{d}z\leq 3(n+2)\hat cM_*.
    \end{split}\end{equation*}
This proves \eqref{smalless initial 0}$_1$ with the choice $M_0=(1+3(n+2)\hat c)M_*$.
To prove \eqref{smalless initial 0}$_2$,
the aim is to determine the upper bound of the constant $\epsilon_*<\hat\epsilon$ and the radius $\rho$.
We first claim that for $\epsilon_*$ and $\rho$ sufficiently small there holds
    \begin{equation}\begin{split}\label{claim1/16}
    \Phi(z_0, 2\rho,l)\leq\frac{1}{16},
      \end{split}\end{equation}
where $l=(u)_{Q_{\rho}(z_0)}+(Du)_{Q_{2\rho}(z_0)}(x-x_0)$. Initially, we set $\rho<\rho_0$, where $\rho_0>0$ is the radius in
Lemma \ref{importanthigher}. Note
that this choice can be justified by the limitation $\rho_i\downarrow 0$.
For the proof of \eqref{claim1/16}, we apply the interpolation inequality to deduce
\begin{equation}\begin{split}\label{interpolation}
&\Big(\ \dashint_{Q_{2\rho}(z_0)}|Du-(Du)_{Q_{2\rho}(z_0)}|^{p_0}\,\mathrm{d}z\Big)^{\frac{1}{p_0}}
\\&\leq\ \Big(\ \dashint_{Q_{2\rho}(z_0)}|Du-(Du)_{Q_{2\rho}(z_0)}|^{p_0(1+\delta)}\,\mathrm{d}z\Big)
^{\frac{1}{p_0(1+\delta)}\Theta}
\Big(\ \dashint_{Q_{2\rho}(z_0)}|Du-(Du)_{Q_{2\rho}(z_0)}|\,\mathrm{d}z\Big)^{1-\Theta}
\end{split}\end{equation}
where $p_0=p(z_0)$, $\delta=\frac{1}{4}\epsilon_0$ and the factor $\Theta$ can be bounded above by a constant $\hat \Theta$
depending only upon $\epsilon_0$ and $\gamma_2$,
\begin{equation*}\Theta=\frac{(p_0-1)(1+\delta)}{p_0(1+\delta)-1}\leq \frac{(\gamma_2-1)(1+\delta
)}{\gamma_2(1+\delta)-1}=:\hat\Theta.\end{equation*}
Since $\rho<\rho_0$, we observe from \eqref{rho1} and \eqref{rho0} that $p_2-p_1\leq \epsilon_0\leq \frac{1}{2}p_1\epsilon_0$.
Using this together with \eqref{higher integrability} and \eqref{small2} we obtain
\begin{equation}\begin{split}\label{higher bound}
\dashint_{Q_{2\rho}(z_0)}|Du|^{p_0(1+\delta)}\,\mathrm{d}z&\leq\
\dashint_{Q_{2\rho}(z_0)}(1+|Du|)^{p(z)(1+\epsilon_0)}\,\mathrm{d}z
\\&\leq c
   \left[\Big(\ \,\dashint_{Q_{4\rho}(z_0)}|Du|^{p(z)}\,\mathrm {d}z\Big)^{1+\frac{1}{2}\epsilon_0 p_0}+1\right]\leq c(M_0).
\end{split}\end{equation}
Consequently, we infer from \eqref{small1}, \eqref{small2}, \eqref{interpolation} and \eqref{higher bound} that
there exists a constant $\hat c_1\geq 1$ depending only upon $n$, $N$, $\nu$, $L$, $\gamma_2$ and $M_0$ such that
the inequality
\begin{equation*}\begin{split}
&\Big(\ \dashint_{Q_{2\rho}(z_0)}|Du-(Du)_{Q_{2\rho}(z_0)}|^{p_0}\,\mathrm{d}z\Big)^{\frac{1}{p_0}}
\leq c(\gamma_2,M_0)^{\hat\Theta}
\Big(\ \dashint_{Q_{4\rho}(z_0)}|Du-(Du)_{Q_{4\rho}(z_0)}|\,\mathrm{d}z\Big)^{1-\hat\Theta}
\leq \hat c_1\epsilon_*
\end{split}\end{equation*}
holds.
This estimate together with the H\"older's inequality ensures us that
 \begin{equation*}\begin{split}\Phi(z_0, 2\rho,l)&=\dashint_{Q_{2\rho}(z_0)}
 \left(\frac{|Du-Dl|}{1+|Dl|}\right)^2+\left(\frac{|Du-Dl|}{1+|Dl|}\right)^{p_0}\,\mathrm {d}z
 \\&\leq 2\Big(\ \dashint_{Q_{2\rho}(z_0)}|Du-(Du)_{Q_{2\rho}(z_0)}|^{p_0}\,\mathrm{d}z\Big)^{\frac{2}{p_0}}
\leq 2\hat c_1^2\epsilon_*^2,\end{split}\end{equation*}
since $p_0\geq 2$.
 At this point, we choose $\hat \epsilon\leq \frac{1}{32}\hat c_1^{-2}$,
which proves the claim \eqref{claim1/16} for $\epsilon_*\leq \hat \epsilon$ and $\rho<\rho_0$.
In view of \eqref{claim1/16} and $\rho<\rho_0$, we can
apply Lemma \ref{poincare lemma} with $l=(u)_{Q_{\rho}(z_0)}+(Du)_{Q_{2\rho}(z_0)}(x-x_0)$ to obtain
\begin{equation*}\begin{split}\label{poincare3}
\dashint_{Q_\rho(z_0)}&
\Big|\frac{u-(u)_{Q_\rho(z_0)}-(Du)_{Q_{2\rho}(z_0)} (x-x_0)}{\rho}\Big|^2\,\mathrm {d}z
 +\ \dashint_{Q_\rho(z_0)}
\Big|\frac{u-(u)_{Q_\rho(z_0)}-(Du)_{Q_{2\rho}(z_0)} (x-x_0)}{\rho}\Big|^{p_0}\,\mathrm {d}z
\\&\leq c(1+|(Du)_{Q_{2\rho}(z_0)}|)^{p_0(p_0-2)}\Big(\ \dashint_{Q_{\rho}}|Du-(Du)_{Q_{2\rho}(z_0)}|^{p_0}\,\mathrm {d}z
\Big)^{\frac{2}{p_0}}
\\&\quad +c(1+|(Du)_{Q_{2\rho}(z_0)}|)^{p_0(p_0-1)}\left[v^{\frac{2(p_0-1)}{p_0}}
(\rho)+\Big(\omega_p(\rho)\log\Big(\frac{1}{\rho}\Big)\Big)
^{\frac{2(p_0-1)}{p_0}}\right]
\\&\leq c(n,N,\gamma_2,M_0)\epsilon_*^2+c(\gamma_2,M_0)
\left[v^{\frac{2}{\gamma_2}}
(\rho)+\Big(\omega_p(\rho)\log\Big(\frac{1}{\rho}\Big)\Big)
^{\frac{2}{\gamma_2}}\right].
\end{split}\end{equation*}
To proceed further, we apply the quasi-minimality of the affine function $l_{z_0,\ \rho}$
from \cite[Lemma 2.2]{DH}
and conclude that there exists a constant
$\hat c_2\geq 1$ depending only upon $n$, $N$, $\nu$, $L$, $\gamma_2$ and $M_0$ such that
the inequality
\begin{equation*}\begin{split}
\Psi(z_0, \rho,l_{z_0,\ \rho})&=\ \dashint_{Q_\rho(z_0)}
\Big|\frac{u-l_{z_0,\ \rho}}{\rho(1+|Dl_{z_0,\ \rho}|)}\Big|^2\,\mathrm {d}z+\ \dashint_{Q_\rho(z_0)}
\Big|\frac{u-l_{z_0,\ \rho}}{\rho(1+|Dl_{z_0,\ \rho}|)}\Big|^{p_0}\,\mathrm {d}z
\\&\leq c\ \dashint_{Q_\rho(z_0)}
\Big|\frac{u-(u)_{Q_\rho(z_0)}-(Du)_{Q_{2\rho}(z_0)}\cdot (x-x_0)}{\rho}\Big|^2\,\mathrm {d}z\\&+c\ \dashint_{Q_\rho(z_0)}
\Big|\frac{u-(u)_{Q_\rho(z_0)}-(Du)_{Q_{2\rho}(z_0)}\cdot (x-x_0)}{\rho}\Big|^{p_0}\,\mathrm {d}z
\\&\leq \hat c_2\epsilon_*^2+\hat c_2
\left[v^{\frac{2}{\gamma_2}}
(\rho)+\Big(\omega_p(\rho)\log\Big(\frac{1}{\rho}\Big)\Big)
^{\frac{2}{\gamma_2}}\right]
\end{split}\end{equation*}
holds. At this stage, we choose $\hat\epsilon>0$ such that $\hat\epsilon<\frac{1}{2}\hat c_2^{-1}$. Moreover,
for a fixed $\epsilon_*<\hat\epsilon$, we choose $\rho<\rho_0$ small enough to have
  \begin{equation}\begin{split}\label{radius}
\hat c_2
\left[v^{\frac{2}{\gamma_2}}
(\rho)+\Big(\omega_p(\rho)\log\Big(\frac{1}{\rho}\Big)\Big)
^{\frac{2}{\gamma_2}}\right]<\frac{1}{2}\epsilon_*,
\end{split}\end{equation}
since $\rho=\rho_j$ for some $j\in\mathbb{N}$ and $\rho_j\downarrow 0$.
This establishes the desired inequality \eqref{smalless initial 0}$_2$ and \eqref{smalless initial 0}$_4$.
Finally, we come to the proof
of \eqref{smalless initial 0}$_3$. Applying \cite[Lemma 2.5]{M} with $\lambda=1$, $A_{z_0,\ \rho}^{(1)}=Dl_{z_0,\ \rho}$,
$\xi=(u)_{Q_\rho(z_0)}$ and $w=(Du)_{Q_{2\rho}(z_0)}$, we obtain
  \begin{equation*}\begin{split}
  |(Du)_{Q_{2\rho}(z_0)}-Dl_{z_0,\ \rho}|^2&\leq\frac{n(n+2)}{\rho^2}\ \dashint_{Q_{\rho}(z_0)}
  \Big|u-(u)_{Q_\rho(z_0)}-(Du)_{Q_{2\rho}(z_0)}\cdot (x-x_0)\Big|^2\,\mathrm{d}z
\leq c\epsilon_*.
   \end{split}\end{equation*}
   Therefore, we conclude that there exists a constant $\hat c_3\geq 1$ depending only upon $n$, $N$, $\nu$, $L$, $\gamma_2$ and $M_0$ such that
the inequality
 \begin{equation*}\begin{split}
 \Phi(z_0, \rho,l_{z_0,\ \rho})&=\dashint_{Q_{\rho}(z_0)}
 \left(\frac{|Du-Dl_{z_0,\ \rho}|}{1+|Dl_{z_0,\ \rho}|}\right)^2+\left(\frac{|Du-Dl_{z_0,\ \rho}|}
 {1+|Dl_{z_0,\ \rho}|}\right)^{p_0}\,\mathrm {d}z
 \\&\leq c\Big(\ \dashint_{Q_{2\rho}(z_0)}|Du-(Du)_{Q_{2\rho}(z_0)}|^{p_0}\,\mathrm{d}z\Big)^{\frac{2}{p_0}}
 +c|(Du)_{Q_{2\rho}(z_0)}-Dl_{z_0,\ \rho}|^2
\\&\qquad\qquad+c|(Du)_{Q_{2\rho}(z_0)}-Dl_{z_0,\ \rho}|^{p_0}
 \\&
\leq \hat c_3\epsilon_*\leq \hat c_3\hat \epsilon
 \end{split}\end{equation*}
holds.
 At this point, we choose $\hat\epsilon>0$ small enough to have $\hat \epsilon<\frac{1}{16}\hat c_3^{-1}$. This proves
\eqref{smalless initial 0}$_3$. In conclusion, we have proved \eqref{smalless initial 0} with the choice
 \begin{equation*}\hat \epsilon=\min\Big\{\frac{1}{32}\hat c_1^{-2},\frac{1}{2}\hat c_2^{-1},
\frac{1}{16}\hat c_3^{-1}\Big\},\end{equation*}
while for any fixed $\epsilon_*<\hat\epsilon$, the radius $\rho<\min\{\frac{1}{8}
\hat\rho,\rho_0\}$ can be determined via \eqref{radius}. The proof of the proposition
is now complete.
\end{proof}
\section{Caccioppoli inequality}
The aim of this section is to establish Caccioppoli type estimate for the weak solution to the parabolic system \eqref{parabolic}.
In the context of the problem for discontinuous coefficients, the energy estimate should be established
on non-uniformly parabolic cylinders and the bounds should be independent of $|Dl|$.
The next lemma is our main result in this section.
\begin{lemma}
Let $z_0\in\Omega_T$ and let $l$ be an affine function. Let $\rho_0>0$ be the radius in Lemma \ref{importanthigher}.
Assume that $\rho<\rho_0$, $\lambda\geq1$ and $\lambda\leq 1+|Dl|\leq M\lambda$ for some $M\geq1$.
Moreover, suppose that $\Phi^{(\lambda)}(z_0,\rho,l)\leq\frac{1}{16}$. Then,
there exists a constant $c$ depending only upon $n$, $N$, $E$, $\gamma_2$, $\nu$, $L$ and $M$
such that the inequality
\begin{equation}\begin{split}\label{Cac}\sup_{t_0-\lambda^{2-p_0}\frac{1}{16}\rho^2<t<t_0}&\
\dashint_{B_{\rho/4}(x_0)}\frac{|u(\cdot,t)-l|^2}{\lambda^{2-p_0}\rho^2}
\,\mathrm{d}x+\ \dashint_{Q^{(\lambda)}_{\rho/4}(z_0)}(1+|Dl|+|Du|)^{p_0-2}|Du-Dl|^2\,\mathrm{d}z
\\&\leq c(1+|Dl|)^{p_0}\Big(\ \dashint_{Q_\rho^{(\lambda)}(z_0)}
\Big|\frac{u-l}{\rho(1+|Dl|)}\Big|^2\,\mathrm {d}z+\ \dashint_{Q_\rho^{(\lambda)}(z_0)}
\Big|\frac{u-l}{\rho(1+|Dl|)}\Big|^{p_0}\,\mathrm {d}z\Big)
\\&\quad+c(1+|Dl|)^{p_0}\Big(v(\rho)+\Big(\omega_p(\rho)\log\Big(\frac{1}{\rho}\Big)\Big)^2\Big)\end{split}\end{equation}
holds, where $p_0=p(z_0)$.
\end{lemma}
\begin{proof}
There is no loss of generality in assuming $z_0=0$.
Once again, we set
$p_1=\inf _{Q_{\rho}^{(\lambda)}}p(z)$ and $p_2=\sup _{Q_{\rho}^{(\lambda)}}p(z)$.
For fixed $\frac{1}{4}\rho<s_1<s_2<\frac{1}{2}\rho$
we consider the concentric parabolic cylinders $Q_{\frac{1}{4}\rho}^{(\lambda)}\subset Q_{s_1}^{(\lambda)}
\subset Q_{s_2}^{(\lambda)}\subset Q_{\frac{1}{2}\rho}^{(\lambda)}\subset Q_{\rho}^{(\lambda)}$.
We choose a cut-off function $\phi=\phi(x)\in C_0^\infty(B_{s_2})$ with $\phi\equiv 1$ on $B_{s_1}$, $0\leq\phi\leq1$
and $|D\phi|\leq c(s_2-s_1)^{-1}$. Moreover, we define a Lipschitz function $\zeta_\epsilon\in W_0^{1,\infty}(\Lambda_\rho^{(\lambda)}
,[0,1])$ via
 \begin{equation}\label{zetaCac}
	\ \zeta_\epsilon(t)=\begin{cases}0,\quad &\text{on}
\ (-\lambda^{2-p_0}\rho^2,\ -\lambda^{2-p_0}s_2^2]\\
	\ \frac{1}{\lambda^{2-p_0}(s_2^2-s_1^2)}(t+\lambda^{2-p_0}s_2^2),\quad &\text{on}\ (-\lambda^{2-p_0}s_2^2,\
-\lambda^{2-p_0}s_1^2]\\
	\ 1,\quad &\text{on}\ (-\lambda^{2-p_0}s_1^2,\
\tilde t]\\
\ \frac{1}{\epsilon}(\tilde t+\epsilon-t),\quad &\text{on}\ (\tilde t,\
\tilde t+\epsilon]\\
\ 0,&\text{on}\ (
\tilde t+\epsilon,\ 0)
	\end{cases}
\end{equation}
for a fixed $\tilde t\in\Lambda_{s_1}^{(\lambda)}$ and $\epsilon\in (0,|\tilde t|)$. In the weak formulation \eqref{weaksolution} we
choose the test function $\varphi(x,t)=\phi\zeta_\epsilon(u-l)$. We first observe that
\begin{equation*}\int_{Q_\rho^{(\lambda)}}l\cdot\partial_t\varphi\,\mathrm{d}z=0\qquad\text{and}\qquad
\int_{Q_\rho^{(\lambda)}}(\mu)_{Q_\rho^{(\lambda)}}\big\langle A(0,Dl), D\varphi\big\rangle\,\mathrm{d}z=0.\end{equation*}
This leads us to
\begin{equation}\begin{split}\label{I-V}
\uppercase\expandafter{\romannumeral1}:&=\int_{Q_\rho^{(\lambda)}}\phi(x)\zeta
_\epsilon(t)\mu(z)\big\langle A(z,Du)- A(z,Dl), Du-Dl\big\rangle\,\mathrm{d}z
\\&=-\int_{Q_\rho^{(\lambda)}}\zeta_\epsilon(t)\mu(z)\big\langle A(z,Du)- A(z,Dl), (u-l)\otimes D\phi\big\rangle\,\mathrm{d}z
\\&\quad-\int_{Q_\rho^{(\lambda)}}[\mu(z)-(\mu)_{Q_\rho^{(\lambda)}}]\big\langle A(z,Dl), D\varphi\big\rangle\,\mathrm{d}z
\\&\quad-\int_{Q_\rho^{(\lambda)}}(\mu)_{Q_\rho^{(\lambda)}}\big\langle A(z,Dl)- A(0,Dl), D\varphi\big\rangle\,\mathrm{d}z
\\&\quad+\int_{Q_\rho^{(\lambda)}}(u-l)\cdot\partial_t\varphi\,\mathrm{d}z=:\uppercase\expandafter{\romannumeral2}+
\uppercase\expandafter{\romannumeral3}+
\uppercase\expandafter{\romannumeral4}+
\uppercase\expandafter{\romannumeral5},
\end{split}\end{equation}
with the obvious meaning of $\uppercase\expandafter{\romannumeral1}$-$\uppercase\expandafter{\romannumeral5}$.
Now we are going to estimate the first terms $\uppercase\expandafter{\romannumeral1}$.
By the fundamental theorem of calculus, we infer from the ellipticity condition \eqref{A}$_1$ that
\begin{equation*}\begin{split}
\uppercase\expandafter{\romannumeral1}&=\int_{Q_\rho^{(\lambda)}}\phi(x)\zeta_\epsilon
(t)\mu(z)\int_0^1\big\langle \partial_wA(z,Dl+s(Du-Dl))(Du-Dl), Du-Dl\big\rangle\,\mathrm{d}s\,\mathrm{d}z
\\&\geq \nu\int_{Q_\rho^{(\lambda)}}\phi\zeta_\epsilon\int_0^1(1+|Dl+s(Du-Dl)|)^{p(z)-2}\,\mathrm{d}s\ |Du-Dl|^2\,\mathrm{d}z
\\&\geq c_1(\nu,\gamma_2)\int_{Q_\rho^{(\lambda)}}\phi\zeta_\epsilon(1+|Dl|+|Du|)^{p(z)-2}\ |Du-Dl|^2\,\mathrm{d}z.
\end{split}\end{equation*}
In the last line we have used \cite[Lemma 2.4]{BDHS} to estimate the integral from below.
Furthermore, we set $\bar \nu=c_1(\nu,\gamma_2)$ and decompose the right-hand side by
$\uppercase\expandafter{\romannumeral1}\geq \bar\nu(\uppercase\expandafter{\romannumeral1}_1+\uppercase\expandafter{\romannumeral1}_2),$
where
\begin{equation*}
\uppercase\expandafter{\romannumeral1}_1:=\int_{Q_\rho^{(\lambda)}}\phi\zeta_\epsilon(1+|Dl|+|Du|)^{p_0-2}\ |Du-Dl|^2\,\mathrm{d}z
\end{equation*}
and
\begin{equation*}
\uppercase\expandafter{\romannumeral1}_2:=\int_{Q_\rho^{(\lambda)}}\phi\zeta_\epsilon\Big((1+|Dl|+|Du|)^{p(z)-2}
-(1+|Dl|+|Du|)^{p_0-2}\Big)\ |Du-Dl|^2\,\mathrm{d}z
.
\end{equation*}
We now come to the estimate of $\uppercase\expandafter{\romannumeral1}_2$. First, we use
the fundamental theorem of calculus and Young's inequality to obtain
for any $\kappa\in(0,1)$ that
\begin{equation}\begin{split}\label{I2}
\uppercase\expandafter{\romannumeral1}_2&\leq c\omega_p(\rho)\int_{Q_{s_2}^{(\lambda)}}(1+|Du|+|Dl|)^{p_2-1}\log(1+|Dl|+|Du|)\ |Du-Dl|
\,\mathrm{d}z
\\&\leq \kappa\int_{Q_{s_2}^{(\lambda)}}(1+|Du|+|Dl|)^{p_0-2}|Du-Dl|^2 \,\mathrm{d}z
\\&\qquad\qquad+c(\kappa)\omega_p^2(\rho)\int_{Q_{\rho/2}^{(\lambda)}}(1+|Du|+|Dl|)^{2p_2-p_0}\log^2(1+|Dl|+|Du|)\,\mathrm{d}z
\\&=:\uppercase\expandafter{\romannumeral1}_3+\uppercase\expandafter{\romannumeral1}_4,
\end{split}\end{equation}
with the obvious meaning of $\uppercase\expandafter{\romannumeral1}_3$ and $\uppercase\expandafter{\romannumeral1}_4$. Since
$2p_2-p_0\leq p_1+2\omega_p(\rho)$, we apply
Lemma \ref{importanthigher} with $\gamma=2$ and $\theta=\frac{1}{2}$ to obtain
\begin{equation}\begin{split}\label{I4}
\uppercase\expandafter{\romannumeral1}_4&\leq c\omega_p^2(\rho)\int_{Q_{\rho/2}^{(\lambda)}}(1+|Du|+|Dl|)^{p_0+2\omega_p(\rho)}
\log^2(1+|Dl|+|Du|)\,\mathrm{d}z
\\&\leq c|Q_{\rho}^{(\lambda)}|\ (1+|Dl|)^{p_0}\ \Big(\omega_p(\rho)\log\Big(\frac{1}{\rho}\Big)\Big)^2.
\end{split}\end{equation}
At this satge, we choose $\kappa=\frac{1}{12}$ in \eqref{I2}. It follows from \eqref{I2} and \eqref{I4}
that for any $\tilde t\in (-\lambda^{2-p_0} s_1^2,0)$ there holds
\begin{equation}\begin{split}\label{estimate for I}
\lim_{\epsilon\downarrow0}
\uppercase\expandafter{\romannumeral1}&\geq
\bar\nu\int_{-\lambda^{2-p_0}s_1^2}^{\tilde t}\int_{B_{s_1}}(1+|Dl|+|Du|)^{p_0-2}\ |Du-Dl|^2\,\mathrm{d}x
\mathrm{d}t
\\&\quad-\frac{\bar\nu}{12}\int_{Q_{s_2}^{(\lambda)}}(1+|Du|+|Dl|)^{p_0-2}|Du-Dl|^2 \,\mathrm{d}z
\\&\quad -c|Q_{\rho}^{(\lambda)}|\ (1+|Dl|)^{p_0}\ \Big(\omega_p(\rho)\log\Big(\frac{1}{\rho}\Big)\Big)^2.
\end{split}\end{equation}
To estimate $\uppercase\expandafter{\romannumeral2}$, we use the mean value theorem from calculus, growth condition \eqref{A}$_2$
and the Young's inequality to obtain
for any $\kappa\in(0,1)$ that
\begin{equation*}\begin{split}
|\uppercase\expandafter{\romannumeral2}|
&=|\int_{Q_\rho^{(\lambda)}}\zeta_\epsilon(t)\mu(z)\big\langle A(z,Du)- A(z,Dl), (u-l)\otimes D\phi\big\rangle\,\mathrm{d}z|
\\&\leq
L\int_{Q_{s_2}^{(\lambda)}}\Big|\int_0^1\partial_wA(z,Dl+s(Du-Dl))
\cdot(Du-Dl)\,\mathrm{d}s\Big| \cdot\Big|\frac{u-l}{s_2-s_1}\Big|\,\mathrm{d}z
\\&\leq c\int_{Q_{s_2}^{(\lambda)}}(1+|Dl|+|Du|)^{p_2-2}|Du-Dl|\cdot\Big|\frac{u-l}{s_2-s_1}\Big|\,\mathrm{d}z
\\&\leq \kappa\int_{Q_{s_2}^{(\lambda)}}(1+|Dl|+|Du|)^{2p_2-p_0-2}|Du-Dl|^2\,\mathrm{d}z
\\&\qquad\qquad+c(\kappa)\int_{Q_{s_2}^{(\lambda)}}
(1+|Dl|+|Du|)^{p_0-2}\cdot\Big|\frac{u-l}{s_2-s_1}\Big|^2\,\mathrm{d}z
\\&=:\uppercase\expandafter{\romannumeral2}_1+
\uppercase\expandafter{\romannumeral2}_2,
\end{split}\end{equation*}
with the obvious meaning of $\uppercase\expandafter{\romannumeral2}_1$ and $\uppercase\expandafter{\romannumeral2}_2$.
In order to estimate $\uppercase\expandafter{\romannumeral2}_1$, we set $\kappa=\bar\nu/12$ and decompose
\begin{equation*}\begin{split}
\uppercase\expandafter{\romannumeral2}_1&= \frac{\bar\nu}{12}\int_{Q_{s_2}^{(\lambda)}}(1+|Dl|+|Du|)^{p_0-2}|Du-Dl|^2\,\mathrm{d}z
\\&\quad+\frac{\bar\nu}{12}\int_{Q_{s_2}^{(\lambda)}}\Big[(1+|Dl|+|Du|)^{2p_2-p_0-2}-(1+|Dl|+|Du|)^{p_0-2}\Big]\cdot|Du-Dl|^2\,\mathrm{d}z.
\end{split}\end{equation*}
It suffices to treat the second term on the right-hand side.
To this end, we use the mean value theorem from calculus and Young's inequality to obtain
\begin{equation*}\begin{split}
\frac{\bar\nu}{12}\int_{Q_{s_2}^{(\lambda)}}&\Big[(1+|Dl|+|Du|)^{2p_2-p_0-2}-(1+|Dl|+|Du|)^{p_0-2}\Big]\cdot|Du-Dl|^2\,\mathrm{d}z
\\&\leq c\omega_p(\rho)\int_{Q_{s_2}^{(\lambda)}}(1+|Dl|+|Du|)^{2p_2-p_0-2}|Du-Dl|^2\log(1+|Dl|+|Du|)\,\mathrm{d}z
\\&\leq c\int_{Q_{s_2}^{(\lambda)}}\omega_p(\rho)(1+|Dl|+|Du|)^{2p_2-\frac{3}{2}p_0}\ \log(1+|Dl|+|Du|)
\\&\qquad\qquad\times (1+|Dl|+|Du|)^{\frac{p_0}{2}-1}|Du-Dl|\,\mathrm{d}z
\\&\leq c\omega_p^2(\rho)\int_{Q_{\rho/2}^{(\lambda)}}(1+|Dl|+|Du|)^{4p_2-3p_0}\log^2(1+|Dl|+|Du|)\,\mathrm{d}z
\\&\qquad\qquad+\frac{\bar\nu}{12}\int_{Q_{s_2}^{(\lambda)}}(1+|Dl|+|Du|)^{p_0-2}|Du-Dl|^2\,\mathrm{d}z.
\end{split}\end{equation*}
Since $4p_2-3p_0\leq p_1+4\omega_p(\rho)$, we apply Lemma \ref{importanthigher} with $\gamma=2$ and $\theta=\frac{1}{2}$
to deduce
\begin{equation*}\begin{split}
\frac{\bar\nu}{12}\int_{Q_{s_2}^{(\lambda)}}&\Big[(1+|Dl|+|Du|)^{2p_2-p_0-2}-(1+|Dl|+|Du|)^{p_0-2}\Big]\cdot|Du-Dl|^2\,\mathrm{d}z
\\&\leq
\frac{\bar\nu}{12}\int_{Q_{s_2}^{(\lambda)}}(1+|Dl|+|Du|)^{p_0-2}|Du-Dl|^2\,\mathrm{d}z
+c|Q_{\rho}^{(\lambda)}|(1+|Dl|)^{p_0}\Big(\omega_p(\rho)\log\Big(\frac{1}{\rho}\Big)\Big)^2,
\end{split}\end{equation*}
and this leads us to
\begin{equation*}\begin{split}
\uppercase\expandafter{\romannumeral2}_1\leq
\frac{\bar\nu}{6}\int_{Q_{s_2}^{(\lambda)}}(1+|Dl|+|Du|)^{p_0-2}|Du-Dl|^2\,\mathrm{d}z
+c|Q_{\rho}^{(\lambda)}|(1+|Dl|)^{p_0}\Big(\omega_p(\rho)\log\Big(\frac{1}{\rho}\Big)\Big)^2.
\end{split}\end{equation*}
We now turn our
attention to the estimate of $\uppercase\expandafter{\romannumeral2}_2$.
Taking into account that $p_0\geq 2$, we use the Young's inequality with exponents $\frac{p_0}{2}$ and $\frac{p_0}{p_0-2}$ to obtain
\begin{equation*}\begin{split}
\uppercase\expandafter{\romannumeral2}_2&\leq c\int_{Q_{s_2}^{(\lambda)}}
(1+|Dl|)^{p_0-2}\cdot\Big|\frac{u-l}{s_2-s_1}\Big|^2+
|Du-Dl|^{p_0-2}\cdot\Big|\frac{u-l}{s_2-s_1}\Big|^2\,\mathrm{d}z
\\&\leq c(1+|Dl|)^{p_0-2}\int_{Q_{s_2}^{(\lambda)}}
\Big|\frac{u-l}{s_2-s_1}\Big|^2\,\mathrm{d}z+c\int_{Q_{s_2}^{(\lambda)}}
\Big|\frac{u-l}{s_2-s_1}\Big|^{p_0}\,\mathrm{d}z
\\&\qquad +\frac{\bar\nu}{12}\int_{Q_{s_2}^{(\lambda)}}|Du-Dl|^{p_0}\,\mathrm{d}z
\\&\leq c(1+|Dl|)^{p_0-2}\int_{Q_{s_2}^{(\lambda)}}
\Big|\frac{u-l}{s_2-s_1}\Big|^2\,\mathrm{d}z+c\int_{Q_{s_2}^{(\lambda)}}
\Big|\frac{u-l}{s_2-s_1}\Big|^{p_0}\,\mathrm{d}z
\\&\qquad +\frac{\bar\nu}{12}\int_{Q_{s_2}^{(\lambda)}}(1+|Dl|+|Du|)^{p_0-2}|Du-Dl|^2\,\mathrm{d}z.
\end{split}\end{equation*}
Therefore, combining the estimates for $\uppercase\expandafter{\romannumeral2}_1$
and $\uppercase\expandafter{\romannumeral2}_2$, we conclude that
\begin{equation}\begin{split}\label{estimate for II}
|\uppercase\expandafter{\romannumeral2}|&\leq c(1+|Dl|)^{p_0-2}\int_{Q_{s_2}^{(\lambda)}}
\Big|\frac{u-l}{s_2-s_1}\Big|^2\,\mathrm{d}z+c\int_{Q_{s_2}^{(\lambda)}}
\Big|\frac{u-l}{s_2-s_1}\Big|^{p_0}\,\mathrm{d}z
\\&\qquad +\frac{\bar\nu}{4}\int_{Q_{s_2}^{(\lambda)}}(1+|Dl|+|Du|)^{p_0-2}|Du-Dl|^2\,\mathrm{d}z
\\&\qquad+c|Q_{\rho}^{(\lambda)}|(1+|Dl|)^{p_0}\Big(\omega_p(\rho)\log\Big(\frac{1}{\rho}\Big)\Big)^2.
\end{split}\end{equation}
Now we come to the estimate of $\uppercase\expandafter{\romannumeral3}$.
We first note that the smallness condition $\rho<\rho_0$ allows us to use Lemma \ref{lambda}.
To this end,
we use the growth condition \eqref{A}$_2$ and Lemma \ref{lambda} (3) to obtain
\begin{equation*}\begin{split}
|\uppercase\expandafter{\romannumeral3}|&=
|\int_{Q_\rho^{(\lambda)}}[\mu(z)-(\mu)_{Q_\rho^{(\lambda)}}]\big\langle A(z,Dl), D\varphi\big\rangle\,\mathrm{d}z|
\\&\leq c(1+|Dl|)^{p_2-1}\int_{Q_{s_2}^{(\lambda)}}|\mu(z)-(\mu)_{Q_\rho^{(\lambda)}}|
|D\varphi|\,\mathrm{d}z
\\&\leq c(1+|Dl|)^{p_0-1}\int_{Q_{s_2}^{(\lambda)}}|\mu(z)-(\mu)_{Q_\rho^{(\lambda)}}|
|Du-Dl|\,\mathrm{d}z
\\&\quad+c(1+|Dl|)^{p_0-1}\int_{Q_{s_2}^{(\lambda)}}|\mu(z)-(\mu)_{Q_\rho^{(\lambda)}}|
\Big|\frac{u-l}{s_2-s_1}\Big|\,\mathrm{d}z
\\&=:\uppercase\expandafter{\romannumeral3}_1+\uppercase\expandafter{\romannumeral3}_2,
\end{split}\end{equation*}
with the obvious meaning of $\uppercase\expandafter{\romannumeral3}_1$ and $\uppercase\expandafter{\romannumeral3}_2$.
Since $p_0\geq 2$, we apply the Young's inequality to obtain
\begin{equation*}\begin{split}
\uppercase\expandafter{\romannumeral3}_1&\leq c(1+|Dl|)^{\frac{p_0}{2}}
\int_{Q_{s_2}^{(\lambda)}}(1+|Dl|+|Du|)^{\frac{p_0}{2}-1}|\mu(z)-(\mu)_{Q_\rho^{(\lambda)}}|
|Du-Dl|\,\mathrm{d}z
\\&\leq \frac{\bar\nu}{12}\int_{Q_{s_2}^{(\lambda)}}(1+|Dl|+|Du|)^{p_0-2}|Du-Dl|^2\,\mathrm{d}z
+c(1+|Dl|)^{p_0}\int_{Q_{s_2}^{(\lambda)}}|\mu(z)-(\mu)_{Q_\rho^{(\lambda)}}|^2\,\mathrm{d}z
\\&\leq \frac{\bar\nu}{12}\int_{Q_{s_2}^{(\lambda)}}(1+|Dl|+|Du|)^{p_0-2}|Du-Dl|^2\,\mathrm{d}z
+c(1+|Dl|)^{p_0}|Q_\rho^{(\lambda)}|v(\rho),
\end{split}\end{equation*}
where we used
$\sqrt{\nu}\leq \mu(z)\leq \sqrt{L}$ in the last step.
To estimate $\uppercase\expandafter{\romannumeral3}_2$, we use the H\"older's inequality to deduce
\begin{equation*}\begin{split}
\uppercase\expandafter{\romannumeral3}_2&=c\int_{Q_{s_2}^{(\lambda)}}\left[(1+|Dl|)^{\frac{p_0}{2}}|\mu(z)-(\mu)_{Q_\rho^{(\lambda)}}|\right]
\times
\left[(1+|Dl|)^{\frac{p_0}{2}-1}\Big|\frac{u-l}{s_2-s_1}\Big|\right]\,\mathrm{d}z
\\&\leq c(1+|Dl|)^{p_0}|Q_\rho^{(\lambda)}|v(\rho)+(1+|Dl|)^{p_0-2}\int_{Q_{s_2}^{(\lambda)}}\Big|\frac{u-l}{s_2-s_1}\Big|^2\,\mathrm{d}z.
\end{split}\end{equation*}
Consequently, we conclude from the estimates for $\uppercase\expandafter{\romannumeral3}_1$
and $\uppercase\expandafter{\romannumeral3}_2$ that
\begin{equation}\begin{split}\label{estimate for III}
|\uppercase\expandafter{\romannumeral3}|&\leq \frac{\bar\nu}{12}\int_{Q_{s_2}^{(\lambda)}}(1+|Dl|+|Du|)^{p_0-2}|Du-Dl|^2\,\mathrm{d}z
\\&\qquad +c(1+|Dl|)^{p_0}|Q_\rho^{(\lambda)}|v(\rho)+(1+|Dl|)^{p_0-2}\int_{Q_{s_2}^{(\lambda)}}\Big|\frac{u-l}{s_2-s_1}\Big|^2\,\mathrm{d}z.
\end{split}\end{equation}
Next, we consider the estimate for $\uppercase\expandafter{\romannumeral4}$. To this end, we apply
the continuity condition \eqref{Az1z2}, Lemma \ref{lambda} (3), (4)
and take into account that $\sqrt{\nu}\leq \mu(z)\leq \sqrt{L}$. This yields
\begin{equation*}\begin{split}
|\uppercase\expandafter{\romannumeral4}|&=
|\int_{Q_\rho^{(\lambda)}}(\mu)_{Q_\rho^{(\lambda)}}\big\langle A(z,Dl)- A(0,Dl), D\varphi\big\rangle\,\mathrm{d}z|
\\&\leq c\sqrt{L}\omega_p(\rho)\int_{Q_{s_2}^{(\lambda)}}(1+|Dl|)^{p_2-1}(1+\log(1+|Dl|))|D\varphi|\,\mathrm{d}z
\\&\leq c\omega_p(\rho)\log\Big(\frac{1}{\rho}\Big)(1+|Dl|)^{p_0-1}\int_{Q_{s_2}^{(\lambda)}}|Du-Dl|\,\mathrm{d}z
\\&\qquad+c\omega_p(\rho)\log\Big(\frac{1}{\rho}\Big)(1+|Dl|)^{p_0-1}\int_{Q_{s_2}^{(\lambda)}}\Big|\frac{u-l}{s_2-s_1}\Big|\,\mathrm{d}z
\\&=:\uppercase\expandafter{\romannumeral4}_1+\uppercase\expandafter{\romannumeral4}_2,
\end{split}\end{equation*}
with the obvious meaning of $\uppercase\expandafter{\romannumeral4}_1$ and $\uppercase\expandafter{\romannumeral4}_2$.
To estimate $\uppercase\expandafter{\romannumeral4}_1$, we use the Young's inequality to deduce
\begin{equation*}\begin{split}
\uppercase\expandafter{\romannumeral4}_1&\leq c\int_{Q_{s_2}^{(\lambda)}}\left[\omega_p(\rho)\log\Big(\frac{1}{\rho}\Big)(1+|Dl|)^{\frac{p_0}
{2}}\right]\cdot\left[(1+|Dl|+|Du|)^{\frac{p_0}{2}-1}|Du-Dl|\right]\,\mathrm{d}z
\\&\leq \frac{\bar\nu}{12}\int_{Q_{s_2}^{(\lambda)}}(1+|Dl|+|Du|)^{p_0-2}|Du-Dl|^2\,\mathrm{d}z
+
c|Q_{\rho}^{(\lambda)}|(1+|Dl|)^{p_0}\Big(\omega_p(\rho)\log\Big(\frac{1}{\rho}\Big)\Big)^2,
\end{split}\end{equation*}
since $p_0\geq2$.
Similarly, in order to estimate $\uppercase\expandafter{\romannumeral4}_2$, we use the Young's inequality again and deduce
\begin{equation*}\begin{split}
\uppercase\expandafter{\romannumeral4}_2&
\leq c\int_{Q_{s_2}^{(\lambda)}}(1+|Dl|)^{p_0-2}\Big|\frac{u-l}{s_2-s_1}\Big|^2\,\mathrm{d}z+
c|Q_{\rho}^{(\lambda)}|(1+|Dl|)^{p_0}\Big(\omega_p(\rho)\log\Big(\frac{1}{\rho}\Big)\Big)^2.
\end{split}\end{equation*}
Consequently, we infer that
\begin{equation}\begin{split}\label{estimate for IV}
\uppercase\expandafter{\romannumeral4}&\leq \frac{\bar\nu}{12}\int_{Q_{s_2}^{(\lambda)}}(1+|Dl|+|Du|)^{p_0-2}|Du-Dl|^2\,\mathrm{d}z
\\&\quad+
c|Q_{\rho}^{(\lambda)}|(1+|Dl|)^{p_0}\Big(\omega_p(\rho)\log\Big(\frac{1}{\rho}\Big)\Big)^2
+c\int_{Q_{s_2}^{(\lambda)}}(1+|Dl|)^{p_0-2}\Big|\frac{u-l}{s_2-s_1}\Big|^2\,\mathrm{d}z.
\end{split}\end{equation}
Finally, we address the estimate of $\uppercase\expandafter{\romannumeral5}$.
In the following we shall proceed formally by assuming that the time derivative exists, since
the arguments can be
made rigorous by the use of the
Steklov averages. Therefore, we find that
\begin{equation*}\begin{split}
\uppercase\expandafter{\romannumeral5}&=
\int_{Q_\rho^{(\lambda)}}|u-l|^2\phi\partial_t\zeta_\epsilon\,\mathrm{d}z+
\frac{1}{2}\int_{Q_\rho^{(\lambda)}}\phi\zeta_\epsilon\partial_t(|u-l|^2)\,\mathrm{d}z
\\&=\frac{1}{2}
\int_{Q_\rho^{(\lambda)}}|u-l|^2\phi\partial_t\zeta_\epsilon\,\mathrm{d}z
\\&\leq\frac{1}{2}\lambda^{p_0-2}\int_{-\lambda^{2-p_0}s_2^2}^{-\lambda^{2-p_0}s_1^2}\int_{B_\rho}
\Big|\frac{u-l}{s_2-s_1}\Big|^2\phi\,\mathrm{d}x\mathrm{d}t
-\frac{1}{2\epsilon}\int_{\tilde t}^{\tilde t+\epsilon}\int_{B_\rho}|u-l|^2\phi\,\mathrm{d}x\mathrm{d}t.
\end{split}\end{equation*}
Recalling that $\lambda\leq 1+|Dl|$, then we conclude
that for any time level $\tilde t\in (-\lambda^{2-p_0} (\rho/4)^2,0)$ there holds
\begin{equation}\begin{split}\label{estimate for V}
\lim_{\epsilon\downarrow 0}\uppercase\expandafter{\romannumeral5}&
\leq\frac{1}{2}(1+|Dl|)^{p_0-2}\int_{Q_\rho^{(\lambda)}}
\Big|\frac{u-l}{s_2-s_1}\Big|^2\,\mathrm{d}z
 -\frac{1}{2}\int_{B_{\rho/4}}|u(\cdot,\tilde t)-l|^2\,\mathrm{d}x,
\end{split}\end{equation}
since $p_0\geq2$.
Combining \eqref{estimate for I}, \eqref{estimate for II}, \eqref{estimate for III},
\eqref{estimate for IV} and \eqref{estimate for V}, we conclude that
for any fixed $\frac{1}{4}\rho<s_1<s_2<\frac{1}{2}\rho$ and any time level $\tilde t\in (-\lambda^{2-p_0} (\rho/4)^2,0)$ there holds
\begin{equation}\begin{split}\label{guocheng}
\int_{B_{\rho/4}}&|u(\cdot,\tilde t)-l|^2\,\mathrm{d}x+
\bar\nu\int_{-\lambda^{2-p_0}s_1^2}^{\tilde t}\int_{B_{s_1}}(1+|Dl|+|Du|)^{p_0-2}\ |Du-Dl|^2\,\mathrm{d}x
\\&\leq \frac{\bar\nu}{2}\int_{Q_{s_2}^{(\lambda)}}(1+|Dl|+|Du|)^{p_0-2}|Du-Dl|^2\,\mathrm{d}z
\\&\quad+
c(1+|Dl|)^{p_0-2}\int_{Q_{\rho}^{(\lambda)}}
\Big|\frac{u-l}{s_2-s_1}\Big|^2\,\mathrm{d}z+c\int_{Q_{\rho}^{(\lambda)}}
\Big|\frac{u-l}{s_2-s_1}\Big|^{p_0}\,\mathrm{d}z
\\&\quad+c|Q_{\rho}^{(\lambda)}|(1+|Dl|)^{p_0}\Big(v(\rho)+\Big(\omega_p(\rho)\log\Big(\frac{1}{\rho}\Big)\Big)^2\Big)
.
\end{split}\end{equation}
Taking into account that $\tilde t\in(-\lambda^{2-p_0} (\rho/4)^2,0)$ is arbitrary,
we first pass to the limit $\tilde t\uparrow0$ in \eqref{guocheng}
and then take the supremum for $\tilde t\in(-\lambda^{2-p_0} (\rho/4)^2,0)$ in \eqref{guocheng} again. This gives
\begin{equation*}\begin{split}
\sup_{\tilde t\in(-\lambda^{2-p_0} (\rho/4)^2,0)}\int_{B_{\rho/4}}&|u(\cdot,\tilde t)-l|^2\,\mathrm{d}x+
\bar\nu\int_{Q_{s_1}^{(\lambda)}}(1+|Dl|+|Du|)^{p_0-2}\ |Du-Dl|^2\,\mathrm{d}x
\\&\leq \frac{\bar\nu}{2}\int_{Q_{s_2}^{(\lambda)}}(1+|Dl|+|Du|)^{p_0-2}|Du-Dl|^2\,\mathrm{d}z
\\&\quad+
c(1+|Dl|)^{p_0-2}\int_{Q_{\rho}^{(\lambda)}}
\Big|\frac{u-l}{s_2-s_1}\Big|^2\,\mathrm{d}z+c\int_{Q_{\rho}^{(\lambda)}}
\Big|\frac{u-l}{s_2-s_1}\Big|^{p_0}\,\mathrm{d}z
\\&\quad+c|Q_{\rho}^{(\lambda)}|(1+|Dl|)^{p_0}\Big(v(\rho)+\Big(\omega_p(\rho)\log\Big(\frac{1}{\rho}\Big)\Big)^2\Big)
.
\end{split}\end{equation*}
At this point, we can use an iteration lemma (cf. \cite[Lemma 2.2]{DMS})
to re-absorb the first integral of the right-hand side into the left.
This leads us to
\begin{equation*}\begin{split}
\sup_{\tilde t\in(-\lambda^{2-p_0} (\rho/4)^2,0)}\int_{B_{\rho/4}}&|u(\cdot,\tilde t)-l|^2\,\mathrm{d}x+
\int_{Q_{\rho/4}^{(\lambda)}}(1+|Dl|+|Du|)^{p_0-2}\ |Du-Dl|^2\,\mathrm{d}x
\\&\leq
c(1+|Dl|)^{p_0-2}\int_{Q_{\rho}^{(\lambda)}}
\Big|\frac{u-l}{\rho}\Big|^2\,\mathrm{d}z+c\int_{Q_{\rho}^{(\lambda)}}
\Big|\frac{u-l}{\rho}\Big|^{p_0}\,\mathrm{d}z
\\&\quad+c|Q_{\rho}^{(\lambda)}|(1+|Dl|)^{p_0}\Big(v(\rho)+\Big(\omega_p(\rho)\log\Big(\frac{1}{\rho}\Big)\Big)^2\Big),
\end{split}\end{equation*}
which proves the desired estimate \eqref{Cac}. Dividing by $|Q_{\rho}^{(\lambda)}|$, it is now obvious that the lemma holds.
\end{proof}
\section{Decay estimate}
This section is devoted to the study of a decay estimate which plays a
crucial role in the partial regularity proof. We first show that the weak solution minus an affine function is
approximatively $A$-caloric.
Before giving the precise statement of this result we
introduce the concept of the hybrid excess functional.
For $z_0\in\Omega_T$, $\rho\in(0,1)$, $\lambda\geq1$ and
an affine function $l:\mathbb{R}^n\to\mathbb{R}^N$, we define the hybrid excess functional
$\Psi^*_\lambda(z_0,\rho,l)$ by
\begin{equation}\label{hybrid excess}\Psi^*_\lambda(z_0,\rho,l)=\Psi_\lambda(z_0,\rho,l)+v(\rho)+\omega_p(\rho)\log\Big(\frac{1}{\rho}\Big),
\end{equation}
where $\Psi_\lambda(z_0,\rho,l)$ is defined in \eqref{psi}.
We
are now in a position to state the following linearization result.
\begin{lemma}\label{linearity lemma}
Let $z_0\in\Omega_T$ and let $l$ be an affine function. Let $\rho_0>0$ be the radius in Lemma \ref{importanthigher}.
Assume that $\rho<\rho_0$, $\lambda\geq1$ and $\lambda\leq 1+|Dl|\leq M\lambda$ for some $M\geq1$.
Moreover, suppose that $\Phi^{(\lambda)}(z_0,\rho,l)\leq\frac{1}{16}$.
Then,
there exists a constant $c$ depending only upon $n$, $N$, $E$, $\gamma_2$, $\nu$, $L$ and $M$
such that
for any $\varphi\in C_0^\infty(Q_{\rho/4}^{(\lambda)}(z_0))$ there holds
\begin{equation}\begin{split}\label{linearity}
\Big|&\ \dashint_{Q_{\rho/4}^{(\lambda)}(z_0)}(u-l)\cdot\partial_t\varphi-
\big\langle(\mu)_{Q_{\rho/4}^{(\lambda)}(z_0)} A(z_0,Dl)\cdot(Du-Dl), D\varphi\big\rangle\,\mathrm{d}z\Big|\\
&\leq c(1+|Dl|)^{p_0-1}\Big[\Psi^*_\lambda(z_0,\rho,l)+
\omega_a^{\frac{1}{2}}
\Big(\sqrt{\Psi^*_\lambda(z_0,\rho,l)}
\Big)\sqrt{\Psi^*_\lambda(z_0,\rho,l)}\ \Big]
\sup_{Q_{\rho/4}^{(\lambda)}(z_0)}|D\varphi|,
\end{split}\end{equation}
where $p_0=p(z_0)$.
\end{lemma}
\begin{proof} Without loss of generality, we may prove the lemma in the case that
 $z_0=0$ and $\sup_{Q_{\rho/4}^{(\lambda)}}|D\varphi|=1$.
For simplicity of notation, we write $\Phi^{(\lambda)}(z_0,\rho,l)$, $\Psi_\lambda(z_0,\rho,l)$ and $\Psi^*_\lambda(z_0,\rho,l)$
for $\Phi^{(\lambda)}(\rho)$, $\Psi_\lambda(\rho)$ and $\Psi^*_\lambda(\rho)$, respectively. To start with, we first
observe that
\begin{equation*}\int_{Q_{\rho/4}^{(\lambda)}}l\cdot\partial_t\varphi\,\mathrm{d}z=0\qquad\text{and}\qquad
\int_{Q_{\rho/4}^{(\lambda)}}(\mu)_{Q_{\rho/4}^{(\lambda)}}\big\langle A(0,Dl), D\varphi\big\rangle\,\mathrm{d}z=0,\end{equation*}
since $\varphi\in C_0^\infty(Q_{\rho/4}^{(\lambda)})$.
From the identities above and weak formulation \eqref{weaksolution}, we have the decomposition
\begin{equation*}\begin{split}
&\ \dashint_{Q_{\rho/4}^{(\lambda)}}(u-l)\cdot\partial_t\varphi-
\big\langle(\mu)_{Q_{\rho/4}^{(\lambda)}} A(0,Dl)\cdot(Du-Dl), D\varphi\big\rangle\,\mathrm{d}z
=\uppercase\expandafter{\romannumeral1}+\uppercase\expandafter{\romannumeral2}+\uppercase\expandafter{\romannumeral3},
\end{split}\end{equation*}
where
\begin{equation*}\begin{split}
\uppercase\expandafter{\romannumeral1}&=\ \dashint_{Q_{\rho/4}^{(\lambda)}}\big\langle(\mu-(\mu)_{Q_{\rho/4}^{(\lambda)}})
A(z,Du),D\varphi\big\rangle\,\mathrm{d}z,
\\ \uppercase\expandafter{\romannumeral2}&=\ \dashint_{Q_{\rho/4}^{(\lambda)}}\big\langle(\mu)_{Q_{\rho/4}^{(\lambda)}}
(A(z,Du)-A(0,Du)),D\varphi\big\rangle\,\mathrm{d}z
\end{split}\end{equation*}
and
\begin{equation*}\begin{split}
\uppercase\expandafter{\romannumeral3}&=\ \dashint_{Q_{\rho/4}^{(\lambda)}}\big\langle(\mu)_{Q_{\rho/4}^{(\lambda)}}
(A(0,Du)-A(0,Dl)),D\varphi\big\rangle\,\mathrm{d}z
\\&\quad-\ \dashint_{Q_{\rho/4}^{(\lambda)}}\big\langle(\mu)_{Q_{\rho/4}^{(\lambda)}}
\partial_wA(0,Dl)\cdot(Du-Dl),D\varphi\big\rangle\,\mathrm{d}z.
\end{split}\end{equation*}
Now we are going to estimate the terms $\uppercase\expandafter{\romannumeral1}$-$\uppercase\expandafter{\romannumeral3}$.
Once again, we set
$p_1=\inf _{Q_{\rho}^{(\lambda)}}p(z)$ and $p_2=\sup _{Q_{\rho}^{(\lambda)}}p(z)$.
For the estimate of $\uppercase\expandafter{\romannumeral1}$, we use the growth condition \eqref{A}$_2$ and decompose
\begin{equation*}\begin{split}
|\uppercase\expandafter{\romannumeral1}|&\leq\ \dashint_{Q_{\rho/4}^{(\lambda)}}|\mu-(\mu)_{Q_{\rho/4}^{(\lambda)}}|
|A(z,Du)||D\varphi|\,\mathrm{d}z
\\&\leq \sqrt{L} \ \dashint_{Q_{\rho/4}^{(\lambda)}}|\mu-(\mu)_{Q_{\rho/4}^{(\lambda)}}|(1+|Du|)^{p(z)-1}\,\mathrm{d}z
\\&\leq c\ \dashint_{Q_{\rho/4}^{(\lambda)}}|\mu-(\mu)_{Q_{\rho/4}^{(\lambda)}}|(1+|Dl|)^{p_2-1}\,\mathrm{d}z
\\&+c\ \dashint_{Q_{\rho/4}^{(\lambda)}}|\mu-(\mu)_{Q_{\rho/4}^{(\lambda)}}||Du-Dl|^{p(z)-1}\,\mathrm{d}z
=:\uppercase\expandafter{\romannumeral1}_1+\uppercase\expandafter{\romannumeral1}_2,
\end{split}\end{equation*}
with the obvious meaning of $\uppercase\expandafter{\romannumeral1}_1$ and $\uppercase\expandafter{\romannumeral1}_2$.
To estimate $\uppercase\expandafter{\romannumeral1}_1$,
we use the definition of $v(\rho)$ to deduce
$\uppercase\expandafter{\romannumeral1}_1\leq c(1+|Dl|)^{p_0-1}v(\rho).$
To estimate $\uppercase\expandafter{\romannumeral1}_2$, we decompose
\begin{equation*}\begin{split}
\uppercase\expandafter{\romannumeral1}_2&=c\frac{1}{|Q_{\rho/4}^{(\lambda)}|}
\int_{Q_{\rho/4}^{(\lambda)}}|\mu-(\mu)_{Q_{\rho/4}^{(\lambda)}}||Du-Dl|^{p(z)-1}\chi_{\{|Du-Dl|\leq1\}}\,\mathrm{d}z
\\&\quad+c\frac{1}{|Q_{\rho/4}^{(\lambda)}|}
\int_{Q_{\rho/4}^{(\lambda)}}|\mu-(\mu)_{Q_{\rho/4}^{(\lambda)}}||Du-Dl|^{p(z)-1}\chi_{\{|Du-Dl|>1\}}\,\mathrm{d}z
\\&\leq cv(\rho)+\uppercase\expandafter{\romannumeral1}_3,
\end{split}\end{equation*}
where
\begin{equation*}\uppercase\expandafter{\romannumeral1}_3=\
\dashint_{Q_{\rho/4}^{(\lambda)}}|\mu-(\mu)_{Q_{\rho/4}^{(\lambda)}}||Du-Dl|^{p_2-1}\chi_{\{|Du-Dl|>1\}}\,\mathrm{d}z.\end{equation*}
To proceed further, we apply the H\"older's inequality with exponents $p_0$ and $\frac{p_0}{p_0-1}$
to conclude that
\begin{equation}\begin{split}\label{II3}
\uppercase\expandafter{\romannumeral1}_3
&\leq c\Big( \ \dashint_{Q_{\rho/4}^{(\lambda)}}|Du-Dl|^{\frac{p_0}{p_0-1}(p_2-1)}\chi_{\{|Du-Dl|>1\}}\,\mathrm{d}z\Big)^{\frac{p_0-1}{p_0}}
\\&\quad \times \Big( \ \dashint_{Q_{\rho/4}^{(\lambda)}}|\mu-(\mu)_{Q_{\rho/4}^{(\lambda)}}|^{p_0}\,\mathrm{d}z\Big)^{\frac{1}{p_0}}
\\&\leq cv(\rho)^{\frac{1}{p_0}}
\Big( \ \dashint_{Q_{\rho/4}^{(\lambda)}}|Du-Dl|^{\frac{p_0}{p_0-1}(p_2-1)}\chi_{\{|Du-Dl|>1\}}\,\mathrm{d}z\Big)^{\frac{p_0-1}{p_0}},
\end{split}\end{equation}
since $\sqrt{\nu}\leq \mu(z)\leq \sqrt{L}$. Moreover,
for any $z\in\{z\in Q_{\rho/4}^{(\lambda)}:\ |Du-Dl|>1\}$, we use the mean value theorem from calculus to find that
\begin{equation*}\begin{split}
|Du-Dl|^{\frac{p_0}{p_0-1}(p_2-1)}-|Du-Dl|^{p_0}&\leq \frac{p_0}{p_0-1}(p_2-p_0)|Du-Dl|^{\frac{p_0}{p_0-1}(p_2-1)}\log|Du-Dl|
\\&\leq 2\omega_p(\rho)|Du-Dl|^{p_1+2\omega_p(\rho)}\log|Du-Dl|,
\end{split}\end{equation*}
since $p_0\geq 2$ and therefore $\frac{p_0}{p_0-1}(p_2-1)\leq p_1+2\omega_p(\rho)$.
Plugging this into \eqref{II3},
we apply Lemma \ref{importanthigher} with $\gamma=1$ and $\theta=\frac{1}{4}$
to deduce
\begin{equation*}\begin{split}
\uppercase\expandafter{\romannumeral1}_3&\leq cv(\rho)^{\frac{1}{p_0}}
\Big( \ \dashint_{Q_{\rho/4}^{(\lambda)}}|Du-Dl|^{p_0}\,\mathrm{d}z\Big)^{\frac{p_0-1}{p_0}}
\\&\quad+cv(\rho)^{\frac{1}{p_0}}
\Big( \ \dashint_{Q_{\rho/4}^{(\lambda)}}\omega_p(\rho)(1+|Du|+|Dl|)^{p_0(1+2\omega_p(\rho))}\log(1+|Du|+|Dl|)
\,\mathrm{d}z\Big)^{\frac{p_0-1}{p_0}}
\\&\leq c(1+|Dl|)^{p_0-1}v(\rho)^{\frac{1}{p_0}}
\Phi^{(\lambda)}\Big(\frac{\rho}{4}\Big)^{\frac{p_0-1}{p_0}}+cv(\rho)^{\frac{1}{p_0}}\Big[
\omega_p(\rho)\log\Big(\frac{1}{\rho}\Big)(1+|Dl|)^{p_0}\Big]^{\frac{p_0-1}{p_0}}.
\end{split}\end{equation*}
From the Caccioppoli's inequality \eqref{Cac}, we see that $\Phi^{(\lambda)}\Big(\frac{\rho}{4}\Big)\leq c\Psi_\lambda^*(\rho)$.
We now use the Young's inequality with exponents $p_0$ and $\frac{p_0}{p_0-1}$ to obtain
\begin{equation*}\begin{split}
\uppercase\expandafter{\romannumeral1}_3&\leq c\Big(v(\rho)+\omega_p(\rho)\log\Big(\frac{1}{\rho}\Big)+\Psi_\lambda^*(\rho)\Big)
(1+|Dl|)^{p_0-1}
\leq c\Psi_\lambda^*(\rho)(1+|Dl|)^{p_0-1},
\end{split}\end{equation*}
where we used \eqref{hybrid excess} in the last step.
The estimates above yield that
\begin{equation}\label{linearI}
\uppercase\expandafter{\romannumeral1}\leq c\Psi_\lambda^*(\rho)(1+|Dl|)^{p_0-1}.
\end{equation}
Next, we consider the estimate for $\uppercase\expandafter{\romannumeral2}$.
To this end, we use the continuity condition \eqref{Az1z2} to obtain
\begin{equation*}\begin{split}
\uppercase\expandafter{\romannumeral2}&\leq
\ \dashint_{Q_{\rho/4}^{(\lambda)}}|(\mu)_{Q_{\rho/4}^{(\lambda)}}|
|(A(z,Du)-A(0,Du))||D\varphi|\,\mathrm{d}z
\\&\leq L\ \dashint_{Q_{\rho/4}^{(\lambda)}}\omega_p(\rho)\Big[(1+|Du|)^{p_0-1}+
(1+|Du|)^{p(z)-1}\Big]
\Big(1+\log(1+|Du|)\Big)\,\mathrm{d}z
\\&\leq L\ \dashint_{Q_{\rho/4}^{(\lambda)}}\omega_p(\rho)(1+|Du|)^{p_2-1}\log(e+|Du|)\,\mathrm{d}z
=:\uppercase\expandafter{\romannumeral2}_1,
\end{split}\end{equation*}
since $\sqrt{\nu}\leq \mu(z)\leq \sqrt{L}$.
Noting that
$\frac{p_0}{p_2-1}\geq1$, we use the H\"older's inequality and Lemma \ref{importanthigher} with $\gamma=\frac{p_0}
{p_2-1}$ and $\theta=\frac{1}{4}$ to obtain
\begin{equation*}\begin{split}
\uppercase\expandafter{\romannumeral2}_1&\leq
L\omega_p(\rho)\Big(\ \dashint_{Q_{\rho/4}^{(\lambda)}}(1+|Du|)^{p_0}\log^{\frac{p_0}{p_2-1}}(e+|Du|)\,\mathrm{d}z\Big)^{\frac{p_2-1}{p_0}}
\\&\leq c\omega_p(\rho)\Big[\log\Big(\frac{1}{\rho}\Big)^{\frac{p_0}{p_2-1}}(1+|Dl|)^{p_0}\Big]^{\frac{p_2-1}{p_0}}
\\&\leq c(1+|Dl|)^{p_2-1}\omega_p(\rho)\log\Big(\frac{1}{\rho}\Big)\leq
c(1+|Dl|)^{p_0-1}\omega_p(\rho)\log\Big(\frac{1}{\rho}\Big),
\end{split}\end{equation*}
where we used Lemma \ref{lambda} (3) in the last step.
From the estimates above, we conclude that
\begin{equation}\label{linearII}
\uppercase\expandafter{\romannumeral2}\leq c \Psi_\lambda^*(\rho) (1+|Dl|)^{p_0-1}.
\end{equation}
Finally, we come to the estimate of $\uppercase\expandafter{\romannumeral3}$. Recalling that $\sqrt{\nu}\leq \mu(z)\leq \sqrt{L}$,
we use
the continuity condition \eqref{DAz1z2} to deduce
\begin{equation*}\begin{split}
|\uppercase\expandafter{\romannumeral3}|&\leq \sqrt{L}\ \dashint_{Q_{\rho/4}^{(\lambda)}}
\Big|
\big\langle\int_0^1[\partial_wA(0,Dl+s(Du-Dl))-\partial_wA(0,Dl)]
\cdot (Du-Dl)\,\mathrm{d}s,D\varphi\big\rangle\Big|\,\mathrm{d}z
\\&\leq c\ \dashint_{Q_{\rho/4}^{(\lambda)}}\int_0^1\omega_a\Big(\frac{s|Du-Dl|}{1+|Dl+s(Du-Dl)|+|Dl|}\Big)
(1+|Dl|+|Du-Dl|)^{p_0-2}\,\mathrm{d}s
|Du-Dl|\,\mathrm{d}z.
\end{split}\end{equation*}
Noting that $s\in(0,1)$, we conclude from the arguments in \cite[page 1808]{M} that
\begin{equation*}
1+|Dl+s(Du-Dl)|+|Dl|\geq\frac{s}{2}(1+|Dl|+|Du-Dl|),
\end{equation*}
and this leads us to
\begin{equation*}\begin{split}
|\uppercase\expandafter{\romannumeral3}|&
\leq c(1+|Dl|)^{p_0-1}\ \dashint_{Q_{\rho/4}^{(\lambda)}}\omega_a\Big(\frac{|Du-Dl|}{1+|Dl|}\Big)
\Big(\frac{|Du-Dl|}{1+|Dl|}+\Big(\frac{|Du-Dl|}{1+|Dl|}\Big)^{p_0-1}\Big)
\,\mathrm{d}z.
\end{split}\end{equation*}
Recalling that
the Caccioppoli's inequality \eqref{Cac} implies $\Phi^{(\lambda)}\Big(\frac{\rho}{4}\Big)\leq c\Psi_\lambda^*(\rho)$,
we now proceed along the lines of the arguments in \cite[page 1808-1809]{M} to conclude that
\begin{equation}\begin{split}\label{linearIII}
|\uppercase\expandafter{\romannumeral3}|&
\leq c(1+|Dl|)^{p_0-1}\Big(\Psi^*_\lambda(\rho)+
\omega_a^{\frac{1}{2}}
\Big(\sqrt{\Psi^*_\lambda(\rho)}
\Big)\sqrt{\Psi^*_\lambda(\rho)}\ \Big).
\end{split}\end{equation}
Combining \eqref{linearI}, \eqref{linearII} and \eqref{linearIII}, we obtain the desired estimate \eqref{linearity}.
The proof of the lemma is now complete.
\end{proof}
With the help of Lemma \ref{linearity lemma} we can now establish a decay estimate
in terms of the hybrid excess functional and the following proposition is our main result in this section.
\begin{proposition}\label{decay proposition} Let $Q_\rho^{(\lambda)}(z_0)\subset\Omega_T$
be a scaled parabolic cylinder on which the intrinsic coupling
$\lambda\leq1+|Dl_{z_0,\ \rho}^{(\lambda)}|\leq M\lambda$ holds
and let $\rho_0>0$ be the radius in Lemma \ref{importanthigher}. Assume that $\Phi^{(\lambda)}(z_0,\rho,l)\leq\frac{1}{16}$
and $\rho<\rho_0$.
For any fixed $\theta<2^{-\frac{1}{2}\gamma_2-6}$, there exists
a constant $\epsilon_1=\epsilon_1(n,N,L,\nu,\gamma_2,E,M,\theta)$ such that
if smallness conditions
\begin{equation}\label{smallness conditions}
\Psi^*_\lambda(z_0,\rho,l_{z_0,\ \rho}^{(\lambda)})\leq \epsilon_1\qquad\text{and}\qquad
\Phi^{(\lambda)}(z_0,\rho,l_{z_0,\ \rho}^{(\lambda)})\leq\frac{1}{16}
\end{equation}
are satisfied, then there exists a
new scaling factor $\lambda_1\in[\frac{1}{2}\lambda,2M\lambda]$ such that
$1+|D l_{z_0,\ \rho}^{(\lambda_1)}|=\lambda_1$ and the following decay estimate holds:
\begin{equation}\label{decay estimate}
\Psi_{\lambda_1}(z_0,\theta\rho,l_{z_0,\ \theta\rho}^{(\lambda_1)})\leq c_*\theta^2
\Psi^*_\lambda(z_0,\rho,l_{z_0,\ \rho}^{(\lambda)}),\end{equation}
where the constant $c_*$ depends only upon  $n$, $N$, $\nu$, $L$, $\gamma_2$, $E$ and $M$.
\end{proposition}
\begin{proof}Without loss of generality, we may take $z_0=0$. Once again, we set $p_0=p(z_0)$,
$p_1=\inf _{Q_\rho^{(\lambda)}(z_0)}p(z)$ and $p_2=\sup _{Q_\rho^{(\lambda)}(z_0)}p(z)$.
For simplicity, we abbreviate
\begin{equation*}\Psi_\lambda(\rho):=\Psi_\lambda(z_0,\rho,l_{z_0,\ \rho}^{(\lambda)})\qquad
\text{and}\qquad\Psi^*_\lambda(\rho):=\Psi^*_\lambda(z_0,\rho,l_{z_0,\ \rho}^{(\lambda)}).\end{equation*}
We first observe that the assumption $\theta<2^{-\frac{1}{2}\gamma_2-6}$
guarantees $Q_{\theta\rho}^{(\lambda_1)}\subset Q_\rho^{(\lambda)}$
for any fixed $\lambda_1\in[\frac{1}{2}\lambda,2M\lambda]$. In order to prove \eqref{decay estimate},
we only need to consider the case $\Psi^*_\lambda(\rho)>0$
due to the quasi-minimality of the affine map $l_{z_0,\ \theta\rho}^{(\lambda_1)}$. This enables us to
define an auxiliary map $v$ via
\begin{equation*}
v(x,t)=\frac{u(x,\lambda^{2-p_0}t)-l_{\rho}^{(\lambda)}}{c_1(1+|Dl_{\rho}^{(\lambda)}|)\sqrt{\Psi^*_\lambda(\rho)}}
\end{equation*}
for $(x,t)\in Q_\rho:= Q_\rho^{(1)}$ and
$c_1\geq1$ is to be determined later. Moreover, we
set $\gamma=\sqrt{\Psi^*_\lambda(\rho)}$.
Initially, we choose $\epsilon_1<1$ and this ensures that
$\gamma<1$.
By a change of variable, we observe from $p_0\geq 2$ that
\begin{equation*}\begin{split}
\dashint_{Q_{\rho/4}}\gamma^{p_0-2}\Big|\frac{v}{\rho}\Big|^{p_0}\,\mathrm{d}z&\leq\frac{1}{\gamma^2c_1^2}
\ \dashint_{Q_{\rho/4}^{(\lambda)}}\Big|\frac{u-l_{\rho}^{(\lambda)}}{\rho(1+|Dl_{\rho}^{(\lambda)}|)}\Big|^{p_0}\,\mathrm{d}z
\leq \frac{c}{c_1^2\Psi^*_\lambda(\rho)}\Psi_\lambda(\rho)\leq\frac{c}{c_1^2},
\end{split}\end{equation*}
where the constant $c$ depending only upon $n$ and $\gamma_2$.
The next thing to do in the proof is to verify that the map $v$ is approximatively $A$-caloric in the sense of \eqref{A caloric}.
From the Caccioppoli's inequality \eqref{Cac}, we deduce
\begin{equation*}\begin{split}
\sup_{t\in\Lambda_{\rho/4}}\ \dashint_{B_{\rho/4}}\Big|\frac{v(\cdot,t)}{\rho}\Big|^2\,\mathrm{d}x
&\leq \frac{\lambda^{2-p_0}}{c_1^2\gamma^2(1+|Dl_{\rho}^{(\lambda)}|)^2}
\sup_{t\in\Lambda_{\rho/4}^{(\lambda)}}\ \dashint_{B_{\rho/4}}\frac{|u(\cdot,t)-l_{\rho}^{(\lambda)}|^2}{\lambda^{2-p_0}
\rho^2}\,\mathrm{d}x
\\&\leq \frac{c_{Cacc}(1+|Dl_{\rho}^{(\lambda)}|)^{p_0-2}}{c_1^2\gamma^2\lambda^{p_0-2}}\Psi^*_\lambda(\rho)
\leq \frac{c(c_{Cacc}, M,\gamma_2)}{c_1^2},
\end{split}\end{equation*}
where $c_{Cacc}$ is the constant in the inequality \eqref{Cac}.
Next, we use the Caccioppoli's inequality \eqref{Cac} again to obtain
\begin{equation*}\begin{split}
\dashint_{Q_{\rho/4}}|Dv|^2+\gamma^{p_0-2}|Dv|^{p_0}\,\mathrm{d}z
&\leq \frac{1}{c_1^2\gamma^2}\Phi^{(\lambda)}\Big(0,\frac{\rho}{4},l_{\rho}^{(\lambda)}\Big)
\leq \frac{c_{Cacc}}{c_1^2\Psi^*_\lambda(\rho)}\Psi_\lambda^*(\rho)=\frac{c_{Cacc}}{c_1^2}.
\end{split}\end{equation*}
Therefore, we conclude from the above estimates that there exists a constant $\bar c$ depending only upon
$n$, $N$, $\nu$, $L$, $\gamma_2$, $E$ and $M$ such that
\begin{equation*}\begin{split}
\sup_{-\frac{1}{16}\rho^2<t<0}\ \dashint_{B_\rho}\left|\frac{v}{\rho/4}\right|^2\,\mathrm {d}x&
+\ \dashint_{Q_{\rho/4}}|Dv|^2\,\mathrm {d}z+\ \dashint_{Q_{\rho/4}}
\gamma^{p_0-2}\left(\ \left|\frac{v}{\rho/4}\right|^{p_0}+|Dv|^{p_0}\right)\,\mathrm {d}z
\leq\frac{\bar c}{c_1^2}\leq1,
\end{split}\end{equation*}
provided that the constant $c_1$ can be chosen so large that $c_1\geq \sqrt{\bar c}$.
Furthermore,
let us now introduce the vector field $A\in \mathbb{R}^N\times\mathbb{R}^{N\times n}$ with constant coefficients
\begin{equation*}A:=(\mu)_{Q_{\rho/4}^{(\lambda)}}\lambda^{2-p_0}
\partial_wA(0,Dl_{\rho}^{(\lambda)}).\end{equation*}
Another step in the proof is to check that $A$ satisfies ellipticity and growth conditions similar to \eqref{linearA}.
We first apply the ellipticity condition \eqref{A}$_1$, $\sqrt{\nu}\leq \mu(z)\leq \sqrt{L}$,
$p_0\geq2$ and $\lambda\leq 1+|Dl_{\rho}^{(\lambda)}|$ to deduce that for any $\tilde w\in\mathbb{R}^{N\times n}$ there holds
\begin{equation*}\begin{split}\big\langle A\tilde w,\tilde w \big \rangle
&=(\mu)_{Q_{\rho/2}^{(\lambda)}}\lambda^{2-p_0}\big\langle \partial_wA(0,Dl_{\rho}^{(\lambda)})\tilde w,\tilde w \big \rangle
\\&\geq\nu \lambda^{2-p_0}(1+|Dl_{\rho}^{(\lambda)}|)^{p_0-2}|\tilde w|^2\geq \nu|\tilde w|^2.
\end{split}\end{equation*}
Moreover, for any $w,\tilde w\in\mathbb{R}^{N\times n}$, we use the growth condition \eqref{A}$_2$, $\sqrt{\nu}\leq \mu(z)\leq \sqrt{L}$,
$p_0\geq2$ and $1+|Dl_{\rho}^{(\lambda)}|\leq M\lambda$
to find that there exists a constant $\tilde c$ depending only upon
$L$, $\gamma_2$ and $M$ such that the following inequality holds:
\begin{equation*}\begin{split}
|\big\langle Aw,\tilde w \big \rangle|&=(\mu)_{Q_{\rho/2}^{(\lambda)}}\lambda^{2-p_0}
|\big\langle \partial_wA(0,Dl_{\rho}^{(\lambda)})\tilde w,\tilde w \big \rangle|
\\&\leq L \lambda^{2-p_0}(1+|Dl_{\rho}^{(\lambda)}|)^{p_0-2}|w||\tilde w|\leq \tilde c|w||\tilde w|.
\end{split}\end{equation*}
To proceed further, we
define a rescaled test function $\tilde \varphi\in C_0^\infty(Q_{\rho/4}^{(\lambda)},\mathbb{R}^N)$ by
$\tilde \varphi(x,\lambda^{2-p_0}t)=\varphi(x,t)$. By the chain rule, we conclude from Lemma \ref{linearity lemma}
with $l=l_{\rho}^{(\lambda)}$
that there exists a constant $\hat c$ depending only upon
$n$, $N$, $\nu$, $L$, $\gamma_2$, $E$ and $M$ such that the inequality
\begin{equation*}\begin{split}
&\Big|\ \dashint_{Q_{\rho/4}}v\cdot\partial_t\varphi
-\big\langle ADv,D\varphi \big \rangle\,\mathrm {d}z\Big|
\\&=\frac{1}{c_1\gamma(1+|Dl_{\rho}^{(\lambda)}|)\lambda^{p_0-2}}
\Big|\ \dashint_{Q_{\rho/4}^{(\lambda)}}(u-l_{\rho}^{(\lambda)})\cdot\partial_t\tilde\varphi
-(\mu)_{Q_{\rho/4}^{(\lambda)}}\big\langle \partial_wA(0,Dl_{\rho}^{(\lambda)})\cdot
(Du-Dl_{\rho}^{(\lambda)}),D\tilde\varphi \big \rangle\,\mathrm {d}z\Big|
\\&\leq \frac{c(1+|Dl_{\rho}^{(\lambda)}|)^{p_0-2}}{c_1\gamma\lambda^{p_0-2}}\Big[\Psi^*_\lambda(\rho)+
\omega_a^{\frac{1}{2}}
\Big(\sqrt{\Psi^*_\lambda(\rho)}
\Big)\sqrt{\Psi^*_\lambda(\rho)}\ \Big]
\sup_{Q_{\rho/4}^{(\lambda)}}|D\tilde\varphi|
\\&\leq \hat c\Big[\sqrt{\Psi^*_\lambda(\rho)}+
\omega_a^{\frac{1}{2}}
\Big(\sqrt{\Psi^*_\lambda(\rho)}
\Big)\ \Big]
\sup_{Q_{\rho/4}^{(\lambda)}}|D\tilde\varphi|
\leq \hat c\Big[\sqrt{\epsilon_1}+
\omega_a^{\frac{1}{2}}
\Big(\sqrt{\epsilon_1}
\Big)\ \Big]
\sup_{Q_{\rho/4}}|D\varphi|
\end{split}\end{equation*}
holds. Our task now is to apply Lemma \ref{caloric} to find an $A$-caloric map $f$ which approximate the
map $v$ in the sense of \eqref{ap}. To this end, we fix $\epsilon\in(0,1)$ which will be
determined in the course of the proof. Moreover, let
$\delta_0=\delta_0(n,N,\gamma_2,\nu,\tilde c,\epsilon)$ be the constant determined by Lemma \ref{caloric} with $\lambda=\nu$
and $\Lambda=\tilde c$. We now impose a condition that $\epsilon_1\leq \tilde\epsilon$ where $\tilde \epsilon>0$ satisfies
\begin{equation}\label{epsilontilde}
\hat c\Big[\sqrt{\tilde \epsilon}+
\omega_a^{\frac{1}{2}}
\Big(\sqrt{\tilde \epsilon}
\Big)\ \Big]\leq \delta_0.
\end{equation}
At this stage, we see that the hypotheses of Lemma \ref{caloric} are fulfilled
and we can apply Lemma \ref{caloric} with $w=v$. Therefore, we conclude the existence of an $A$-caloric map
$f\in C^\infty(Q_{\rho/8},\mathbb{R}^N)$ satisfying
\begin{equation}\label{caloric1}\dashint_{Q_{\rho/8}}\Big|\frac{f}{\rho/8}\Big|^2+|Df|^2\,\mathrm {d}z
+\ \ \dashint_{Q_{\rho/8}}\gamma^{p_0-2}\left(\ \Big|\frac{f}{\rho/8}\Big|^{p_0}+|Df|^{p_0}
\right)\,\mathrm {d}z\leq c(n,\gamma_2,\nu,
L)
\end{equation}
and
\begin{equation}\label{caloric2}\dashint_{Q_{\rho/8}}\Big|\frac{v-f}{\rho/8}\Big|^2\,\mathrm {d}z
+\gamma^{p_0-2}\ \Big|\frac{v-f}{\rho/8}\Big|^{p_0}\,\mathrm {d}z\leq \epsilon.
\end{equation}
We have established the existence of the $A$-caloric map,
and now we further study its decay estimate.
To start with, we set $\tilde \theta=2^{\frac{1}{2}\gamma_2+2}\theta$.
It follows from $\theta<2^{-\frac{1}{2}\gamma_2-6}$
that $\tilde\theta<\frac{1}{4}$. We now apply \eqref{caloric1} and \cite[Lemma 2.7]{DH}
for $s=2$ or $s=p_0$ to deduce that there exists a constant $c_{pa}$ depending only upon
$n$, $N$, $\nu$, $L$ and $\gamma_2$ such that the inequality
\begin{equation*}\begin{split}
\gamma^{s-2}&\Big(\frac{\tilde\theta\rho}{8}\Big)^{-s}\ \dashint_{Q_{\tilde \theta\rho/8}}
|f-(f)_{Q_{\tilde \theta\rho/8}}-
(Df)_{Q_{\tilde \theta\rho/8}}\cdot x|^s\,\mathrm{d}z
\\&\leq c_{pa}\gamma^{s-2}\tilde \theta^s
\Big(\frac{\rho}{8}\Big)^{-s}\ \dashint_{Q_{\rho/8}}
|f-(f)_{Q_{\rho/8}}-
(Df)_{Q_{\rho/8}}\cdot x|^s\,\mathrm{d}z
\\&\leq c_{pa}\tilde \theta^s\Big(\gamma^{s-2}\ \dashint_{Q_{\rho/8}}\Big|\frac{f}{\rho/8}\Big|^s\,\mathrm{d}z+
\gamma^{s-2}\ \dashint_{Q_{\rho/8}}|Df|^{s}\,\mathrm{d}z\Big)\leq c_{pa}\tilde \theta^s
\end{split}\end{equation*}
holds for either $s=2$ or $s=p_0$.
Furthermore, we conclude from the above estimate and \eqref{caloric2} that for $s=2$ or $s=p_0$ there holds
\begin{equation}\begin{split}\label{first decay}
\gamma^{s-2}&\Big(\frac{\tilde\theta\rho}{8}\Big)^{-s}\ \dashint_{Q_{\tilde \theta\rho/8}}
|v-(f)_{Q_{\tilde \theta\rho/8}}-
(Df)_{Q_{\tilde \theta\rho/8}}\cdot x|^s\,\mathrm{d}z
\\&\leq 2^{s-1}\gamma^{s-2}\Big(\frac{\tilde\theta\rho}{8}\Big)^{-s}\ \dashint_{Q_{\tilde \theta\rho/8}}|v-f|^s\,\mathrm{d}z
\\&\qquad+2^{s-1}\gamma^{s-2}\Big(\frac{\tilde\theta\rho}{8}\Big)^{-s}\ \dashint_{Q_{\tilde \theta\rho/8}}
|f-(f)_{Q_{\tilde \theta\rho/8}}-
(Df)_{Q_{\tilde \theta\rho/8}}\cdot x|^s\,\mathrm{d}z
\\&\leq c\tilde \theta^{-(n+2+s)}
\Big(\frac{\rho}{8}\Big)^{-s}\gamma^{s-2}\ \dashint_{Q_{\rho/8}}|v-f|^s\,\mathrm{d}z+c\tilde \theta^s
\leq c\tilde \theta^{-(n+2+s)}\epsilon+c\tilde \theta^s.
\end{split}\end{equation}
At this stage, we choose $\epsilon=\tilde \theta^{n+2+2\gamma_2}$ and therefore $\epsilon$ depends only upon $\theta$
and $\gamma_2$. The choice of $\epsilon$ determines the value of $\delta_0$ in dependence of $n$, $N$, $\gamma_2$, $\nu$, $L$ and $\theta$.
Moreover, the constant $\tilde\epsilon$ is fixed via \eqref{epsilontilde}. We now proceed to estimate \eqref{first decay} and obtain
\begin{equation}\begin{split}\label{v decay}
&\Big(\frac{\tilde\theta\rho}{8}\Big)^{-s}\ \dashint_{Q_{\tilde \theta\rho/8}}
|v-(f)_{Q_{\tilde \theta\rho/8}}-
(Df)_{Q_{\tilde \theta\rho/8}}\cdot x|^s\,\mathrm{d}z\leq c\tilde \theta^s \gamma^{2-s}.
\end{split}\end{equation}
Recalling the definition of the map $v$, we rescale back to $u$ in \eqref{v decay} and this yields that for $s=2$ or $s=p_0$ there holds
\begin{equation*}\begin{split}
\Big(\frac{\tilde\theta\rho}{8}\Big)^{-s}&\ \dashint_{Q_{\tilde \theta\rho/8}^{(\lambda)}}
\Big|u-l_{\rho}^{(\lambda)}-c_1(1+|Dl_{\rho}^{(\lambda)}|)\gamma[
(f)_{Q_{\tilde \theta\rho/8}}-
(Df)_{Q_{\tilde \theta\rho/8}}\cdot x]\Big|^s\,\mathrm{d}z
\\&\leq c(1+|Dl_{\rho}^{(\lambda)}|)^s\gamma^s\tilde \theta^s \gamma^{2-s}\leq c\lambda^s\tilde \theta^s\Psi_\lambda^*(\rho).
\end{split}\end{equation*}
A careful examination of the proof of
\cite[Lemma 2.8]{BDM}
ensures the existence of a constant
$\hat c_1\geq 1$ depending only upon $n$, $N$, $\nu$, $L$, $\gamma_2$ and $M$ such that
the inequality
\begin{equation}\begin{split}\label{main decay}
\Big(\frac{\tilde\theta\rho}{8}\Big)^{-s}&\ \dashint_{Q_{\tilde \theta\rho/8}^{(\lambda)}}
|u-l_{\tilde \theta\rho/8}^{(\lambda)}|^s\,\mathrm{d}z
\leq \hat c_1\lambda^s\tilde \theta^s\Psi_\lambda^*(\rho)
\end{split}\end{equation}
holds for either $s=2$ or $s=p_0$. Our task now is to find a scaling factor $\lambda_1\in[\frac{1}{2}\lambda,2M\lambda]$ such that
$1+|D l_{\rho}^{(\lambda_1)}|=\lambda_1$.
For $\lambda_*\in[\frac{1}{2}\lambda,2M\lambda]$, we have $(\theta\rho)^2\lambda_*^{2-p_0}
\leq(\frac{1}{8}
\tilde\theta\rho)^2\lambda^{2-p_0}$, since $p_0\geq 2$ and $\tilde \theta=2^{\frac{1}{2}\gamma_2+2}\theta$.
This implies that $Q_{\theta\rho}^{(\lambda_*)}\subset
Q_{\tilde \theta\rho/8}^{(\lambda)}\subset Q_{\rho}^{(\lambda)}$. Therefore, we apply \cite[Lemma 2.5]{M}
with $z_0=0$, $\lambda=\lambda_*$, $A_{z_0,\ \rho}^{(\lambda)}=Dl_{\theta\rho}^{(\lambda_*)}$,
$\xi=l_{\rho}^{(\lambda)}(0)$ and $w=Dl_{\rho}^{(\lambda)}$ to conclude that
there exists a constant $\tilde c_1$ depending only upon $n$, $M$ and $\gamma_2$ such that the following inequality holds
\begin{equation*}\begin{split}
|Dl_{\theta\rho}^{(\lambda_*)}-Dl_{\rho}^{(\lambda)}|^2&\leq \frac{n(n+2)}{(\theta\rho)^2}\ \dashint_{Q_{\theta\rho}^{(\lambda_*)}}
|u-l_{\rho}^{(\lambda)}|^2\,\mathrm{d}z
\\&\leq n(n+2)\theta^{-(n+4)}\Big(\frac{\lambda}
{\lambda_*}\Big)^{2-p_0}\ \dashint_{Q_{\rho}^{(\lambda)}}
\Big|\frac{|u-l_{\rho}^{(\lambda)}|}{\rho}\Big|^2\,\mathrm{d}z
\\&\leq \tilde c_1\theta^{-(n+4)}\lambda^2\Psi_\lambda^*(\rho)\leq \tilde c_1\theta^{-(n+4)}\lambda^2\epsilon_1,
\end{split}\end{equation*}
since $1+|Dl_{\rho}^{(\lambda)}|\leq M\lambda$.
At this point, we assume that $\epsilon_1\leq \frac{1}{4}\tilde c_1^{-1}\theta^{n+4}$
and this yields $|Dl_{\theta\rho}^{(\lambda_*)}-Dl_{\rho}^{(\lambda)}|\leq \frac{1}{2}\lambda$. By the triangle inequality, we observe
from $\lambda\leq1+|Dl_{\rho}^{(\lambda)}|\leq M\lambda$ that
for any $\lambda_*\in[\frac{1}{2}\lambda,2M\lambda]$ there holds
\begin{equation}\label{lambdalu}
1+|Dl_{\theta\rho}^{(\lambda_*)}|\leq 2M\lambda\qquad\text{and}\qquad1+|Dl_{\theta\rho}^{(\lambda_*)}|\geq \frac{1}{2}\lambda.
\end{equation}
To proceed further,
we now define the function $f:[\frac{1}{2}\lambda,2M\lambda]\to \mathbb{R}$ by
$f(\lambda^*)=1+|Dl_{\theta\rho}^{(\lambda_*)}|-\lambda^*$.
From \cite[(2.9)]{BDM}, we find that
\begin{equation*}\begin{split}
 Dl_{\theta\rho}^{(\lambda_*)}=\frac{n+2}{(\theta\rho)^2}\ \dashint_{Q_{\theta\rho}^{(\lambda_*)}}
u\otimes x\,\mathrm{d}z
    \end{split}\end{equation*}
and this implies that $f(\lambda^*)$ is continuous.
From \eqref{lambdalu}, we see that
$f(\frac{1}{2}\lambda)\geq0$ and $f(2M\lambda)\leq0$. By continuity of the function $f(\lambda^*)$,
there exists a scaling factor $\lambda_1\in [\frac{1}{2}\lambda,M\lambda]$
such that $f(\lambda_1)=0$ which leads us to $1+|Dl_{\theta\rho}^{(\lambda_1)}|=\lambda_1$. Recalling that
$Q_{\theta\rho}^{(\lambda_1)}\subset
Q_{\tilde \theta\rho/8}^{(\lambda)}$,
then we conclude from \eqref{main decay} and
the proof of \cite[Lemma 2.8]{BDM} that there exists a constant $\tilde c_2$
depending only upon $n$ and $\gamma_2$ such that the inequality
\begin{equation*}\begin{split}
\dashint_{Q_{\theta\rho}^{(\lambda_1)}}
\Big|\frac{u-l_{\theta\rho}^{(\lambda_1)}}{\theta\rho}\Big|^s\,\mathrm{d}z&\leq \tilde c_2\
\dashint_{Q_{\theta\rho}^{(\lambda_1)}}
\Big|\frac{u-l_{\tilde \theta\rho/8}^{(\lambda)}}{\theta\rho}\Big|^s\,\mathrm{d}z
\\& \leq \tilde c_2\Big(\frac{\tilde \theta}
{8\theta}\Big)^{n+2+s}\Big(\frac{\lambda}
{\lambda_1}\Big)^{2-p_0}\Big(\frac{\tilde\theta\rho}{8}\Big)^{-s}\ \dashint_{Q_{\tilde \theta\rho/8}^{(\lambda)}}
|u-l_{\tilde \theta\rho/8}^{(\lambda)}|^s\,\mathrm{d}z
\\&
\leq 2^{(\frac{1}{2}\gamma_2-1)(n+2+\gamma_2)}
\tilde c_2\hat c_1\lambda^s\tilde \theta^s\Psi_\lambda^*(\rho)\leq
2^{5n\gamma_2^2}\tilde c_2\hat c_1\theta^2(1+|Dl_{\theta\rho}^{(\lambda_1)}|)^s\Psi_\lambda^*(\rho)
\end{split}\end{equation*}
holds for either $s=2$ or $s=p_0$.
Therefore, the desired inequality \eqref{decay estimate} follows with the choice $c_*=2^{5n\gamma_2^2}\tilde c_2\hat c_1$,
provided that the smallness condition $\Psi^*_\lambda(\rho)\leq \epsilon_1$ holds with $\epsilon_1=\min\{1,\tilde\epsilon,
\frac{1}{4}\tilde c_1^{-1}\theta^{n+4}\}$. The proof of the proposition is now complete.
\end{proof}
\section{Proof of partial regularity}
In this section, we give the proof of Theorem \ref{main}. Our proof starts from the estimate \eqref{smalless initial 0}
in Proposition \ref{start}
and we will show that the weak solution to the parabolic system \eqref{parabolic} is H\"older continuous in a
neighborhood of any regular point. The strategy is standard and we will make use of an
integral characterization of parabolic H\"older space due to Campanato and DaPrato (cf. \cite[Theorem 2.10]{M}).
Before proving Theorem \ref{main}, we need the following iteration lemma.
\begin{lemma}\label{iteration lemma}Let $z_0\in\Omega_T$ and $M\geq1$.
Moreover, let $\rho_0>0$ be the radius in Lemma \ref{importanthigher} and let
$c_*$ be the constant in \eqref{decay estimate}. Assume that
$0<\theta<\min\{2^{-
\gamma_2-6},(3c_*)^{-\frac{1}{2}}\}$.
Then, there exist a constant $\epsilon_2=\epsilon_2(n,N,L,\nu,\gamma_2,E,M,\theta)>0$
such that the following holds: If there exists a radius $\rho<\rho_0$
such that the following inequalities
 \begin{equation}\label{smalless initial}
\begin{cases}\,&1+|Dl_{z_0,\ \rho}|\leq M,\\
\,&\Psi(z_0,\rho,l_{z_0,\ \rho})\leq\epsilon_2,\\
\,&\Phi(z_0,\rho,l_{z_0,\ \rho})\leq\frac{1}{16},\\
\,& v(\rho)+\omega_p(\rho)\log\Big(\frac{1}{\rho}\Big)\leq \epsilon_2
	\end{cases}
\end{equation}
hold,
then
there exists a sequence $\{\lambda_j\}_{j\in \mathbb{N}}$ such that for all $j\in\mathbb{N}$ there holds
\begin{equation}\label{smalless jstep}
\begin{cases}
\,&1\leq\lambda_j\leq(2M)^j,\\
\,&\lambda_j\leq 1+|Dl_{z_0,\ \theta^j\rho}^{(\lambda_j)}|\leq M\lambda_j,\tag{$I_j$}\\
\,&\Psi_{\lambda_j}(z_0,\theta^j\rho,l_{z_0,\ \theta^j\rho}^{(\lambda_j)})\leq\epsilon_2,\\
\,&\Phi^{(\lambda_j)}(z_0,\theta^j\rho,l_{z_0,\ \theta^j\rho}^{(\lambda_j)})\leq\frac{1}{16}.
	\end{cases}
\end{equation}
\end{lemma}
\begin{proof} We omit for simplification the reference point $z_0$ in our notation and we will prove the lemma by induction.
We fix a constant $\hat\epsilon_1<\epsilon_1$ where $\epsilon_1$ is the constant from Proposition \ref{decay proposition}
and the constant $\hat\epsilon_1$ will be determined in
the course of the proof.
Now, we set $\epsilon_2:=\frac{1}{3}\hat\epsilon_1$.
Initially, the inequalities in
$(I_{0})$
hold true
from \eqref{smalless initial}.

It only suffices to prove that
the implication
$(I_{j})\Rightarrow (I_{j+1})$ holds true for any
$j\geq0$.
For simplicity, we may take
$l_j=l_{\theta^{j}\rho}^{(\lambda_{j})}$ for $j=0,1,\cdots$.
From \eqref{smalless initial}, \eqref{smalless jstep}$_3$ and \eqref{hybrid excess}, we conclude that for any $\rho<\rho_0$
there holds
\begin{equation}\label{psixing}
\Psi_{\lambda_j}^*(\theta^j\rho,l_j)\leq
\Psi_{\lambda_j}(\theta^j\rho,l_j)+v(\rho)+\omega_p(\rho)\log\Big(\frac{1}{\rho}\Big)
\leq 2\epsilon_2<\epsilon_1.
\end{equation}
Keeping in mind that $\rho<\rho_0$ we deduce from \eqref{psixing} and $(I_j)_2$-$(I_j)_4$ that the hypotheses of Proposition
\ref{decay proposition} with $(\rho,\lambda)$ replaced by $(\theta^j\rho,\lambda_j)$ are fulfilled.
Therefore, we conclude that
there exists a scaling factor $\lambda_{j+1}\in[\frac{1}{2}\lambda_j,2M\lambda_j]$ such that $\lambda_{j+1}=1+|Dl_{j+1}|$,
hence that $\lambda_{j+1}\geq1$, and finally that $\lambda_{j+1}\leq 2M\lambda_j\leq (2M)^{j+1}$.
This implies that $(I_{j+1})_1$ and $(I_{j+1})_2$ hold true.
Since $\theta<(3c_*)^{-\frac{1}{2}}$,
we infer from \eqref{decay estimate} and \eqref{psixing} that
\begin{equation*}\begin{split}
\Psi_{\lambda_{j+1}}(\theta^{j+1}\rho,l_{j+1})&\leq c_*\theta^2 \Psi_{\lambda_{j}}^*(\theta^{j}\rho,l_{j})
\leq 3c_*\theta^2\epsilon_2\leq \epsilon_2.
\end{split}
\end{equation*}
This proves that $(I_{j+1})_3$ holds true.
The next thing to do in the proof is to verify that $(I_{j+1})_4$ holds true.
Since $\theta<2^{-
\gamma_2-6}$ and $\lambda_{j+1}\in[\frac{1}{2}\lambda_j,2M\lambda_j]$, we have $\sqrt{\theta}<\frac{1}{4}$ and
$Q^{(\lambda_{j+1})}_{\theta^{j+1}\rho}\subset Q^{(\lambda_{j})}_{\theta^{j+\frac{1}{2}}\rho}
\subset Q^{(\lambda_{j})}_{\theta^{j}\rho}$.
Moreover, we apply
the Caccioppoli inequality \eqref{Cac} with $(\rho,\lambda,l)$ replaced by $(\theta^j\rho,\lambda_j,l_j)$ to deduce that
there exists a constant $\bar c_1$ depending only upon $n$, $N$, $\nu$, $L$, $\gamma_2$, $E$ and $M$ such that
the following inequality holds
\begin{equation}\begin{split}\label{c1bar}
\Phi^{(\lambda_{j+1})}(\theta^{j+1}\rho,l_{j})&=\ \dashint_{Q_{\theta^{j+1}\rho}^{(\lambda_{j+1})}}
 \left(\frac{|Du-Dl_j|}{1+|Dl_j|}\right)^2+\left(\frac{|Du-Dl_j|}{1+|Dl_j|}\right)^{p_0}\,\mathrm {d}z
\\&\leq\theta^{-\frac{1}{2}(n+2)}\Big(\frac{\lambda_j}{\lambda_{j+1}}
\Big)^{2-p_0}\Phi^{(\lambda_{j})}(\theta^{j+\frac{1}{2}}\rho,l_{j})
\\&\leq \theta^{-(n+2)}(2M)^{\gamma_2-2}\Phi^{(\lambda_{j})}(\frac{1}{4}\theta^{j}\rho,l_{j})
\\&\leq  c_{Cacc}\theta^{-(n+2)}(2M)^{\gamma_2-2}\Psi_{\lambda_{j}}^*(\theta^{j}\rho,l_{j})
\leq \bar c_1\theta^{-(n+2)}\hat \epsilon_1,
\end{split}
\end{equation}
where we used \eqref{psixing} in the last step.
On the other hand, we apply \cite[Lemma 2.5]{M} with $(\rho,\lambda,A_{z_0,\ \rho}^{(\lambda)},\xi,w)$
replaced by $(\theta^{j+1}\rho,\lambda_{j+1},Dl_{j+1},l_j(x_0),Dl_j)$
to infer that there exists a constant $\bar c_2\geq1$
depending only upon  $n$, $\gamma_2$ and $M$ such that
the following inequality holds
\begin{equation}\begin{split}\label{c2bar1}
|Dl_{j+1}-Dl_{j}|^2&\leq \frac{n(n+2)}{(\theta^{j+1}\rho)^2}\ \dashint_{Q_{\theta^{j+1}\rho}^{(\lambda_{j+1})}}
|u-l_j|^2\,\mathrm{d}z
\\&\leq \frac{n(n+2)\lambda_j^{2-p_0}}{\theta^{n+4}\lambda_{j+1}^{2-p_0}}
\ \dashint_{Q_{\theta^{j}\rho}^{(\lambda_{j})}}
\frac{|u-l_j|^2}{(\theta^j\rho)^2(1+|Dl_j|)^2}\,\mathrm{d}z\cdot (1+|Dl_j|)^2
\\&\leq n(n+2)(2M)^{\gamma_2-2}\theta^{-n-4}(1+|Dl_j|)^2\Psi_{\lambda_{j}}(\theta^{j}\rho,l_{j})
\\&\leq \bar c_2\theta^{-n-4}(1+|Dl_j|)^2\hat\epsilon_1.
\end{split}
\end{equation}
Moreover, the inequality \eqref{c2bar1} also implies that
\begin{equation}\begin{split}\label{c2bar2}
|Dl_{j+1}-Dl_{j}|^{p_0}\leq \bar c_2^{\frac{\gamma_2}{2}}\theta^{-(n+4)\frac{\gamma_2}{2}}(1+|Dl_j|)^{p_0}\hat\epsilon_1,
\end{split}
\end{equation}
since $p_0\geq 2$ and $\hat\epsilon_1<1$. Recalling that $\lambda_{j+1}\in[\frac{1}{2}\lambda_j,2M\lambda_j]$
and $\lambda_{j+1}=1+|Dl_{j+1}|$,
then we obtain $1+|Dl_j|\geq\lambda_j\geq\frac{1}{2M}\lambda_{j+1}=\frac{1}{2M}(1+|Dl_{j+1}|)$ and $1+|Dl_{j+1}|=\lambda_{j+1}
\geq\frac{1}{2}\lambda_j\geq \frac{1}{2M}(1+|Dl_{j}|)$. At this stage, we apply \eqref{c1bar}-\eqref{c2bar2} to conclude with
\begin{equation*}\begin{split}
\Phi^{(\lambda_{j+1})}&(z_0,\theta^{j+1}\rho,l_{j+1})\\&\leq 2^{\gamma_2-1}(2M)^{\gamma_2}
\Phi^{(\lambda_{j+1})}(z_0,\theta^{j+1}\rho,l_{j})
+2^{\gamma_2-1}\frac{|Dl_j-Dl_{j+1}|^{p_0}}{(1+|Dl_{j+1}|)^{p_0}}
+2\frac{|Dl_j-Dl_{j+1}|^2}{(1+|Dl_{j+1}|)^2}
\\&\leq  2^{\gamma_2-1}(2M)^{\gamma_2}\left(
\Phi^{(\lambda_{j+1})}(z_0,\theta^{j+1}\rho,l_{j})
+\frac{|Dl_j-Dl_{j+1}|^{p_0}}{(1+|Dl_{j}|)^{p_0}}
+\frac{|Dl_j-Dl_{j+1}|^2}{(1+|Dl_{j}|)^2}\right)
\\&\leq 2^{\gamma_2-1}(2M)^{\gamma_2}(2\bar c_2^{\frac{\gamma_2}{2}}+\bar c_1)\theta^{-(n+4)\frac{\gamma_2}{2}}
\hat\epsilon_1.
\end{split}
\end{equation*}
Finally, we choose
\begin{equation}\label{epsilon2}\hat\epsilon_1=\min\Big\{\frac{1}{2}\epsilon_1,
\frac{1}{16}[2^{\gamma_2-1}(2M)^{\gamma_2}(2\bar c_2^{\frac{\gamma_2}{2}}+\bar c_1)]^{-1}\theta^{(n+4)\frac{\gamma_2}{2}}\Big\}\end{equation}
and therefore the constant $\epsilon_2=\frac{1}{3}\hat\epsilon_1$ depends only upon
$n$, $N$, $\nu$, $L$, $\gamma_2$, $E$, $M$ and $\theta$.
Moreover, we have established the inequality $(I_{j+1})_4$ for this
choice of $\hat\epsilon_1$. The proof of the lemma is now complete.
\end{proof}
\begin{proof}[Proof of Theorem \ref{main}]
We fix a point $\mathfrak z_0\in \Omega_T\setminus(\Sigma_1\cup\Sigma_2)$ and
let $\alpha\in (0,1)$ be a fixed constant. From now on, we show that $u$ is H\"older continuous with H\"older exponent
$\alpha$ on a neighborhood of $\mathfrak z_0$.
From Proposition \ref{start}, we can find a constant $M>0$ and
a constant $\hat\epsilon=\hat\epsilon(n,N,\nu,L,E,\gamma_2,M)>0$
such that
the following is true: For any $\epsilon<\hat\epsilon$ there exists a radius $\rho=\rho(n,N,\nu,L,E,\gamma_2,M,\epsilon)<\rho_0$ such that
for any $z_0\in Q_{\rho/8}(\mathfrak z_0)$ there holds
 \begin{equation}\label{coincide}
\begin{cases}\,&1+|Dl_{z_0,\ \rho}|\leq M,\\
\,&\Psi(z_0,\rho,l_{z_0,\ \rho})\leq\epsilon,\\
\,&\Phi(z_0,\rho,l_{z_0,\ \rho})\leq\frac{1}{16},
\\\,&v(\rho)+\omega_p(\rho)\log\Big(\frac{1}{\rho}\Big)\leq \epsilon.
	\end{cases}
\end{equation}
At this point, we fix
 \begin{equation}\label{thetafinal}\theta=\frac{1}{2}\min\left\{2^{-
\gamma_2-6},(3c_*)^{-\frac{1}{2}},(2M)^{-\frac{4+(\gamma_2-2)(n+2+2\alpha)}{4(1-\alpha)}}\right\}\end{equation}
where $c_*$ is the constant in \eqref{decay estimate}. This also fixes $\epsilon_2$ via \eqref{epsilon2}.
Moreover, we
fix $\epsilon=\min\{\hat\epsilon,\epsilon_2\}$ where $\epsilon_2$ is the constant in Lemma \ref{iteration lemma}.
Note that this particular choice of $\epsilon$ determines
$\rho=\rho(\epsilon)$ from Proposition \ref{start} and therefore \eqref{coincide} holds for such $\epsilon$ and $\rho$.
Since $\epsilon\leq\epsilon_2$, we see that \eqref{smalless initial} holds true and the hypotheses of Lemma \ref{iteration lemma}
are fulfilled.
Our task now is to show that $u\in C_{\mathrm{loc}}^{0;\alpha,\alpha/2}(Q_{\rho/8}(\mathfrak z_0),\mathbb{R}^N)$ by using the iteration 
scheme from Lemma \ref{iteration lemma}.
To this end, we use the Campanato characterization for the parabolic H\"older space. More precisely, we shall show that
\begin{equation*}
\sup_{z_0\in Q_{\rho/8}(\mathfrak z_0)}
\sup_{r>0}\frac{1}{|Q_r(z_0)|^{1+\frac{2\alpha}{n+2}}}\int_{Q_r(z_0)}|u-(u)_{Q_r(z_0)}|^2\,\mathrm{d}z<+\infty.\end{equation*}
By a localization argument, our problem reduces to show that
there exists a constant $c=c(n,N,\nu,L,E,\gamma_2,\theta,\alpha,\rho,M)>0$ such that
for any $z_0\in Q_{\rho/8}(\mathfrak z_0)$ and $r\in(0,\rho]$ there holds
\begin{equation}\label{final estimate}\int_{Q_r(z_0)}|u-(u)_{Q_r(z_0)}|^2\,\mathrm{d}z\leq cr^{n+2+2\alpha}.\end{equation}
The proof of \eqref{final estimate} follows in a similar manner as the arguments in \cite[page 1817-1819]{M} and
we just sketch the proof. Once again, we set
$l_j=l_{z_0,\ \theta^{j}\rho}^{(\lambda_{j})}$ for $j=0,1,\cdots$.
We first show that
\begin{equation}\label{jestimate}\dashint_{Q_{\theta^j\rho}^{(\lambda_j)}(z_0)}|u-l_j(x_0)|^2\,\mathrm{d}z\leq
(2M)^{2(j+1)}(\theta^j\rho)^2\end{equation}
holds for any $j\in\mathbb{N}$.
According to \eqref{smalless jstep}$_1$-\eqref{smalless jstep}$_3$, it follows that
\begin{equation*}\begin{split}
\dashint_{Q_{\theta^j\rho}^{(\lambda_j)}(z_0)}|u-l_j(x_0)|^2\,\mathrm{d}z
&=\ \dashint_{Q_{\theta^j\rho}^{(\lambda_j)}(z_0)}|u-l_j+Dl_j\cdot(x-x_0)|^2\,\mathrm{d}z
\\&\leq 2\ \dashint_{Q_{\theta^j\rho}^{(\lambda_j)}(z_0)}|u-l_j|^2\,\mathrm{d}z+
2|Dl_j|^2\ \dashint_{Q_{\theta^j\rho}^{(\lambda_j)}(z_0)}|x-x_0|^2\,\mathrm{d}z
\\&\leq 2(\theta^j\rho)^2(1+|Dl_j|)^2\Psi_{\lambda_{j}}^*(z_0,\theta^{j}\rho,l_{j})+2(\theta^j\rho)^2|Dl_j|^2
\\&\leq 4(\theta^j\rho)^2(1+|Dl_j|)^2\leq
(2M)^{2(j+1)}(\theta^j\rho)^2,
\end{split}\end{equation*}
which proves \eqref{jestimate}.
Therefore, we conclude from \eqref{jestimate} and $l_j(x_0)=(u)_{Q_{\theta^j\rho}^{(\lambda_j)}(z_0)}$ that for any $j\geq0$
there holds
\begin{equation*}\begin{split}
\int_{Q_{\theta^j\rho}^{(\lambda_j)}(z_0)}|u-(u)_{Q_{\theta^j\rho}^{(\lambda_j)}(z_0)}|^2\,\mathrm{d}z
\leq (2M)^{2(j+1)}(\theta^j\rho)^2|Q_{\theta^j\rho}^{(\lambda_j)}(z_0)|=\alpha_n(2M)^{2(j+1)}(\theta^j\rho)^{n+4}\lambda_j^{2-p_0},
\end{split}\end{equation*}
where $\alpha_n$ is the volume of the unit ball in $\mathbb{R}^n$.
Next, we set $\hat\theta=(2M)^{\frac{2-\gamma_2}{2}}\theta$ and this implies
$Q_{\hat\theta^j\rho}(z_0)\subset Q_{\theta^j\rho}^{(\lambda_j)}(z_0)$.
Consequently, we infer from \eqref{thetafinal} that
\begin{equation*}\begin{split}
\int_{Q_{\hat\theta^j\rho}(z_0)}|u-(u)_{Q_{\hat\theta^j\rho}(z_0)}|^2\,\mathrm{d}z&\leq
2\int_{Q_{\hat\theta^j\rho}(z_0)}|u-(u)_{Q_{\theta^j\rho}^{(\lambda_j)}(z_0)}|^2\,\mathrm{d}z\leq
2\int_{Q_{\theta^j\rho}^{(\lambda_j)}(z_0)}|u-(u)_{Q_{\theta^j\rho}^{(\lambda_j)}(z_0)}|^2\,\mathrm{d}z
\\&\leq 2\alpha_n(2M)^2\rho^{n+4}\hat\theta^{j(n+2+2\alpha)}\Big[(2M)^{\frac{4+(\gamma_2-2)(n+2+2\alpha)}{2}}\theta^{2-2\alpha}\Big]^j
\\&\leq 2\alpha_n(2M)^2\rho^{n+4}\hat\theta^{j(n+2+2\alpha)}.
\end{split}\end{equation*}
Finally, for any fixed $r\in(0,\rho]$, there exists an integer $j\in\mathbb{N}$ such that $\hat\theta^{j+1}\rho<r\leq \hat\theta^j\rho$.
Therefore, we conclude
from the above estimate that
\begin{equation*}\begin{split}
\int_{Q_r(z_0)}|u-(u)_{Q_r(z_0)}|^2\,\mathrm{d}z&\leq2
\int_{Q_r(z_0)}|u-(u)_{Q_{\hat\theta^j\rho}(z_0)}|^2\,\mathrm{d}z
\leq2
\int_{Q_{\hat\theta^j\rho}(z_0)}|u-(u)_{Q_{\hat\theta^j\rho}(z_0)}|^2\,\mathrm{d}z
\\&\leq 4\alpha_n(2M)^2\rho^{n+4}\hat\theta^{j(n+2+2\alpha)}
\\&\leq 4\alpha_n(2M)^2\hat\theta^{-(n+2+2\alpha)}\rho^{n+4}\left(\frac{r}{\rho}\right)^{n+2+2\alpha}\leq c(n,\gamma_2,M,\alpha,\theta,\rho)
r^{n+2+2\alpha},
\end{split}\end{equation*}
which proves the desired estimate \eqref{final estimate} for
any fixed $\alpha\in(0,1)$. This finishes the proof of Theorem \ref{main}.
\end{proof}
\begin{remark}We finally remark that the method in this paper could be improved to address the problem for the full range $p(z)>\frac{2n}{n+2}$.
In the case that $\frac{2n}{n+2}<p(z_0)<2$ for a fixed point $z_0\in\Omega_T$, we could consider the scaled parabolic cylinder in the following form:
$$Q^{(\lambda)}_r(z_0):=B_{\lambda^{\frac{2-p_0}{2}}r}(x_0)\times\Lambda_r(t_0)$$
where
$p_0=p(z_0)$,
$B_{\lambda^{\frac{2-p_0}{2}}r}(x_0)=\{x\in\mathbb{R}^n:|x-x_0|\leq \lambda^{\frac{2-p_0}{2}}r\}$ and
$\Lambda_r(z_0)=(t_0-r^2,t_0)$. We leave the problem of
the singular range $\frac{2n}{n+2}<p(z)<2$ for future study.
\end{remark}
\appendix
\section{Estimate in $L^p\log^\gamma L$ space}
For the sake of completeness we state in Proposition \ref{LlogL} an embedding result $L^{p+\sigma}\hookrightarrow L^p\log^\gamma L$
from \cite[Lemma 8.6]{IV} which we used in the proof of Lemma \ref{importanthigher}. For completeness sake we also include the proof
of this result here.
\begin{proposition}\label{LlogL}
Let $m\geq2$, $\bar\gamma>1$ and $\Omega\subset \mathbb{R}^m$ be a measurable set.
Let $\gamma_2\geq1$ and $1\leq\gamma\leq\bar\gamma$. Assume that $p\geq1$ and $f\in L^p\log L^\gamma(\Omega)$.
For any $\sigma\in(0,\gamma)$, there exists a constant $c$ depending only upon $\bar\gamma$ and $\sigma$ such that
\begin{equation}\begin{split}\label{LlogLestimate}\Big(\ \ \dashint_{\Omega}|f|^p\log^\gamma \Big(e+\frac{|f|}{\Big(\ \dashint_{\Omega}|f|^p\,
\mathrm {d}x\Big)^{\frac{1}{p}}}\Big)\,\mathrm {d}x\Big)^{\frac{1}{p}}\leq c\Big(\ \ \dashint_{\Omega}|f|^{p+\sigma}\,
\mathrm {d}x\Big)^{\frac{1}{p+\sigma}}.\end{split}\end{equation}
\end{proposition}
\begin{proof} Our proof is due to Iwaniec and Verde \cite[Lemma 8.6]{IV}.
We define the left-hand side of \eqref{LlogLestimate} by
\begin{equation*}\begin{split}
[f]_{L^p\log^\gamma L}:=\Big(\ \ \dashint_{\Omega}|f|^p\log^\gamma \Big(e+\frac{|f|}{\Big(\ \dashint_{\Omega}|f|^p\,\mathrm
{d}x\Big)^{\frac{1}{p}}}\Big)\,\mathrm {d}x\Big)^{\frac{1}{p}}
\end{split}\end{equation*}
and introduce the Luxemburg norm
\begin{equation}\begin{split}\label{Luxemburg}\|f\|_{L^p\log^\gamma L}=\inf\left\{k>0:\ \ \dashint_\Omega|f|^p\log^\gamma\Big(e+\frac{|f|}
{k}\Big)\,\mathrm {d}x\leq k^p\right\}.
\end{split}\end{equation}
For a fixed $\alpha\in(0,1)$, it is easy to verify that the inequality $\log(1+x)\leq \frac{1}{\alpha} x^\alpha$ holds for every $x\geq0$.
This gives
$\log^\gamma(1+x)\leq \Big(\frac{\gamma}{\sigma}\Big)^\gamma x^\sigma$. It follows from \eqref{Luxemburg} that
\begin{equation*}\begin{split}\|f\|_{L^p\log^\gamma L}\leq c
\Big(\ \ \dashint_{\Omega}|f|^{p+\sigma}\,\mathrm {d}x\Big)^{\frac{1}{p+\sigma}},\end{split}\end{equation*}
where the constant $c$ depends only upon $\sigma$ and $\bar\gamma$.
It remains to prove that $[f]_{L^p\log^\gamma L}\leq c\|f\|_{L^p\log^\gamma L}$.
To this end, we set $K=\|f\|_{L^p\log^\gamma L}$
and deduce
\begin{equation}\begin{split}\label{lplogl}[f]_{L^p\log^\gamma L}^p&
\leq  c_\gamma\ \ \dashint_{\Omega}|f|^p\log^\gamma \Big(e+\frac{|f|}{K}\Big)\,\mathrm {d}x
+c_\gamma
\ \ \dashint_{\Omega}|f|^p\log^\gamma \Big(e+\frac{K}{\Big(\ \dashint_{\Omega}|f|^p\,\mathrm {d}x\Big)^{\frac{1}{p}}}\Big)\,\mathrm {d}x
\\&=cK^p+c\log^\gamma\ \Big(e+\frac{K}{\ \Big(\ \dashint_{\Omega}|f|^p\,\mathrm {d}x\Big)^{\frac{1}{p}}}\Big)
\ \ \dashint_{\Omega}|f|^p\,\mathrm {d}x.
\end{split}\end{equation}
In the case $\ \dashint_{\Omega}|f|^p\,\mathrm {d}x\geq K^p$, we infer from \eqref{lplogl} that
\begin{equation*}\begin{split}[f]_{L^p\log^\gamma L}\leq c(\bar\gamma,\sigma)
\Big(\ \dashint_{\Omega}|f|^p\,\mathrm {d}x\Big)^{\frac{1}{p}}
\end{split}\end{equation*}
and \eqref{LlogLestimate} follows by the H\"older's inequality. In the case $\ \dashint_{\Omega}|f|^p\,\mathrm {d}x< K^p$, we observe that
 the estimate $\log^\gamma(e+x)\leq c(\gamma)x$ holds for any $x\geq1$. It follows from \eqref{lplogl} that
 $[f]_{L^p\log^\gamma L}^p\leq c(\bar\gamma)K^p$.
We have thus proved the proposition.
\end{proof}
\section{$A$-caloric approximation}
In this section, we present a version of the $A$-caloric approximation that is compatible with the intrinsic geometry of non-standard growth.
Our proof follows in a similar manner as the proofs of \cite[Lemma 3.2]{DMS} and \cite[Lemma 2.8]{DH}, and
we just sketch the proof.
We follow the terminology used in \cite[Chapter 3]{DMS} and introduce the definition of $A$-caloric map.
\begin{definition}Let $\lambda,\Lambda>0$ be fixed constants
and $A:\mathbb{R}^{n\times N}\to \mathbb{R}^{n\times N}$ be a bilinear form which satisfies
\begin{equation}\begin{split}\label{linearA}\lambda|\tilde w|^2\leq \big\langle A\tilde w,\tilde w \big \rangle,\qquad
|\big\langle Aw,\tilde w \big \rangle|\leq\Lambda|w||\tilde w|,\qquad w,\tilde w\in\mathbb{R}^{N\times n}.
\end{split}\end{equation}
A map $f\in L^2(t_0-\rho^2,t_0;W^{1,2}(B_\rho(x_0)),\mathbb{R}^N) $ is called $A$-caloric in the parabolic cylinder
$Q_\rho(z_0)$ if it satisfies
\begin{equation*}\int_{Q_\rho(z_0)}\left[f\cdot\partial_t \varphi-\big\langle ADf, D\varphi\big\rangle\right]
\mathrm {d}z=0,\qquad\text{whenever}\quad \varphi\in C_0^\infty(Q_\rho(z_0);\mathbb{R}^N).\end{equation*}
\end{definition}
We
are now in a position to state our main result of this section.
\begin{lemma}\label{caloric}
Let $\epsilon>0$ and $0<\lambda\leq\Lambda$ be fixed.
Then, for any $p\in[2,\gamma_2]$ there exists $\delta_0=\delta_0(n,N,\gamma_2,\lambda,L,\epsilon)\in(0,1]$
with the following property: Whenever $\gamma\in(0,1]$ and $A$ is a bilinear form on $\mathbb{R}^{N\times n}$ satisfying
\eqref{linearA}
and whenever
$w\in L^p(t_0-\rho^2,t_0;W^{1,p}(B_\rho(x_0)),\mathbb{R}^N)$
is a map satisfying
\begin{equation*}\sup_{t_0-\rho^2<t<t_0}\ \dashint_{B_\rho(x_0)}\left|\frac{w}{\rho}\right|^2\,\mathrm {d}x
+\ \dashint_{Q_\rho(z_0)}|Dw|^2\,\mathrm {d}z+\ \dashint_{Q_\rho(z_0)}
\gamma^{p-2}\left(\ \left|\frac{w}{\rho}\right|^p+|Dw|^p\right)\,\mathrm {d}z
\leq1
\end{equation*}
which is approximatively $A$-caloric in the sense that
\begin{equation}\label{A caloric}\Big|\ \dashint_{Q_\rho(z_0)}w\cdot\partial_t\varphi
-\big\langle ADw,D\varphi \big \rangle\,\mathrm {d}z\Big|\leq\delta\sup_{Q_\rho(z_0)}|D\varphi|\end{equation}
for every $\varphi\in C_0^\infty(Q_\rho(z_0),\mathbb{R}^N)$, where $\delta\leq\delta_0(n,N,\gamma_2,\lambda,\Lambda,\epsilon)$,
then there exists an $A$-caloric map $f\in C^\infty(Q_{\rho/2}(z_0),\mathbb{R}^N)$ satisfies
\begin{equation*}\dashint_{Q_{\rho/2}(z_0)}\Big|\frac{f}{\rho/2}\Big|^2+|Df|^2\,\mathrm {d}z
+\ \ \dashint_{Q_{\rho/2}(z_0)}\gamma^{p-2}\left(\ \Big|\frac{f}{\rho/2}\Big|^p+|Df|^p\right)\,\mathrm {d}z\leq c(n,\gamma_2,\lambda,
\Lambda)
\end{equation*}
and
\begin{equation}\label{ap}\dashint_{Q_{\rho/2}(z_0)}\Big|\frac{w-f}{\rho/2}\Big|^2\,\mathrm {d}z
+\gamma^{p-2}\ \Big|\frac{w-f}{\rho/2}\Big|^p\,\mathrm {d}z\leq \epsilon.
\end{equation}
\end{lemma}
\begin{proof}
Our approach is along the same lines as the one used in \cite[Lemma 3.2]{DMS} and \cite[Lemma 2.8]{DH}.
The lemma will be proved by contradiction and the proof
is divided into several steps. To start with, we first observe that
it suffices to prove the lemma in the case $\rho=1$ and $z_0=0$ by a scaling argument.

Step 1: If the statement was false then
we could find an $\epsilon>0$, $\gamma_k\in(0,1]$, $p_k\in[2,\gamma_2]$, bilinear forms
$A_k$ satisfying \eqref{linearA}, $w_k\in L^{p_k}(-1,0;W^{1,p_k}(B_1),\mathbb{R}^N)$ such that for every $k\in\mathbb{N}$ the map $w_k$ is
approximately $A_k$-caloric in the sense that
\begin{equation*}\Big|\ \dashint_{Q_1}w_k\cdot\partial_t\varphi
-\big\langle A_kDw_k,D\varphi \big \rangle\,\mathrm {d}z\Big|\leq\frac{1}{k}\sup_{Q_1}|D\varphi|\end{equation*}
for all $\varphi\in C_0^\infty(Q_1,\mathbb{R}^N)$ and satisfies
\begin{equation}\label{contradiction wk}\sup_{-1<t<0}\ \dashint_{B_1}\left|w_k(\cdot,t)\right|^2\,\mathrm {d}x
+\ \dashint_{Q_1}|Dw_k|^2\,\mathrm {d}z+\ \dashint_{Q_1}
\gamma_k^{p_k-2}\left(|w_k|^{p_k}+|Dw_k|^{p_k}\right)\,\mathrm {d}z
\leq1,
\end{equation}
while for all $A_k$-caloric maps $f\in C^\infty(Q_{1/2},\mathbb{R}^N)$
satisfying
\begin{equation}\label{B1assumption}\dashint_{Q_{1/2}}4|f|^2+|Df|^2\,\mathrm {d}z
+\ \ \dashint_{Q_{1/2}}\gamma_k^{p_k-2}\left(2^{p_k}|f|^{p_k}+|Df|^{p_k}\right)\,\mathrm {d}z\leq c_*
\end{equation}
with a constant $c_*=c_*(n,N,\gamma_2,\lambda,\Lambda)$,
there holds
\begin{equation}\label{B1result}\dashint_{Q_{1/2}}4|w_k-f|^2\,\mathrm {d}z
+2^{p_k}\gamma_k^{p_k-2} |w_k-f|^{p_k}\,\mathrm {d}z> \epsilon.
\end{equation}
In the remainder of the proof, we will construct a sequence of $A_k$-caloric maps $v_k$ such that \eqref{B1assumption} holds for a constant
$c_*>0$, but \eqref{B1result} fails to hold, which leads to a contradiction.

Step 2: We discuss uniform bounds and weak convergence for $w_k$. To this end,
we set $\tilde w_k=\gamma_k^{(p_k-2)/p_k}w_k$ and it follows from \eqref{contradiction wk} that
\begin{equation*}\dashint_{Q_1}|\tilde w_k|^{p_k}+|D\tilde w_k|^{p_k}\,\mathrm {d}z\leq 1.\end{equation*}
At this point, we choose $q\geq2$ with $q\leq p_k\leq q(1+\frac{2}{n})$ holds for all $k\geq 1$. By the H\"older's inequality, we have
\begin{equation}\label{wk}\dashint_{Q_1}|\tilde w_k|^q+|D\tilde w_k|^q\,\mathrm {d}z\leq 1.\end{equation}
In view of $p_k\in[2,\gamma_2]$ and $\gamma_k\in (0,1]$,
we see that $\gamma_k^{\frac{p_k-2}{p_k}}
\leq 1$ for all $k\in\mathbb{N}$. In the case $p_k\to 2$ and $\gamma_k\to0$, the limitation of $\gamma_k^{\frac{p_k-2}{p_k}}$
may not exists, but
the boundedness of $\gamma_k^{\frac{p_k-2}{p_k}}
\leq 1$ ensures the existence of a not relabeled-subsequence $\{\gamma_k^{\frac{p_k-2}{p_k}}\}_{k=1}^{+\infty}$
such
that $\gamma_k^{\frac{p_k-2}{p_k}}\to \mu$ for a constant $\mu\in[0,1]$.
Moreover, from \eqref{contradiction wk} and \eqref{wk}, we infer the existence of maps
$w\in L^2(-1,0;W^{1,2}(B_1,\mathbb{R}^N))$, $\tilde w\in L^q(-1,0;W^{1,q}(B_1,\mathbb{R}^N))$, a bilinear form $A$
and constants $(\gamma,p,\mu)$ such that
 \begin{equation}\label{convergence}
	\begin{cases}
	\ w_k\rightharpoonup w\qquad &\text{weakly\ in}\ L^2(Q_1,\mathbb{R}^N),   \\
	\ Dw_k\rightharpoonup Dw\qquad &\text{weakly\ in}\ L^2(Q_1,\mathbb{R}^{N\times n}),\\
\ \tilde w_k\rightharpoonup \tilde w\qquad &\text{weakly\ in}\ L^q(Q_1,\mathbb{R}^N),   \\
	\ D\tilde w_k\rightharpoonup D\tilde w\quad &\text{weakly\ in}\ L^q(Q_1,\mathbb{R}^{N\times n}),\\
\ A_k\to A\qquad&\text{as\ bilinear\ forms\ in}\ \mathbb{R}^{N\times n},
\\ \ \gamma_k\to \gamma \qquad&\text{in}\ [0,1],
\\ \ p_k\to p \qquad&\text{in}\ [2,\gamma_2],
\\ \ \gamma_k^{\frac{p_k-2}{p_k}}\to \mu \qquad&\text{in}\ [0,1].
	\end{cases}
\end{equation}
We observe from \eqref{convergence}$_1$ and \eqref{convergence}$_8$ that
$\tilde w_k\rightharpoonup \mu w$ weakly in $L^2(-1,0;W^{1,2}(B_1,\mathbb{R}^N)).$
 On the other hand, \eqref{convergence}$_3$ and \eqref{convergence}$_4$ imply
 $\tilde w_k\rightharpoonup \tilde w$ weakly in $L^q(-1,0;W^{1,q}(B_1,\mathbb{R}^N))$. It is clear that $\tilde w=\mu w$.
 By lower semicontinuity, we infer from \eqref{contradiction wk} and \eqref{wk} that
\begin{equation}\label{w}\dashint_{Q_1}|w|^2+|Dw|^2+|\tilde w|^q+|D\tilde w|^q\,\mathrm {d}z\leq 1.\end{equation}
Furthermore, we follow the arguments in \cite[page 28]{DMS} to infer that the limit map $w$ is an $A$-caloric map. More precisely, we have
\begin{equation*}\int_{Q_1}\left[w\cdot\partial_t \varphi-\big\langle ADw, D\varphi\big\rangle\right]
\mathrm {d}z=0,\end{equation*}
for any $\varphi\in C_0^\infty(Q_1,\mathbb{R}^N)$. This also implies that $w\in C^\infty(Q_1,\mathbb{R}^N)$.

Step 3: We aim to establish the strong convergence of $w_k$ by using compactness method. To start with,
we follow the arguments in \cite[page 28-29]{DMS} to deduce that for any $h\in(0,1)$ there holds
\begin{equation*}\int_{-1}^{-h}\|w_k(\cdot,t)-w_k(\cdot,t+h)\|_{W^{-l,2}(B_1,\mathbb{R}^N)}^2
\mathrm {d}t\leq c(n,L,l)\Big(h+\frac{1}{k^2}\Big),\end{equation*}
where $l>\frac{n+2}{2}$ is fixed. This inequality together with a compactness result (cf. \cite[Theorem 2.5]{DMS}) implies that
$w_k\to w$ strongly in $L^2(Q_1,\mathbb{R}^N)$. Moreover, since
\begin{equation*}\begin{split}\|\tilde w_k-\tilde w\|_{L^2(Q_1,\mathbb{R}^N)}
\leq \|w_k-w\|_{L^2(Q_1,\mathbb{R}^N)}+\|w\|_{L^2(Q_1,\mathbb{R}^N)}(\gamma_k^{\frac{p_k-2}{p_k}}-\mu),\end{split}\end{equation*}
we have $\tilde w_k\to \tilde w$ strongly in $L^2(Q_1,\mathbb{R}^N)$. Next, we claim that
\begin{equation}\label{claim}\tilde w_k\to\tilde w\qquad \text{strongly\ in}\ L^\sigma(Q_{1/2},\mathbb{R}^N),\end{equation}
for some $\sigma>q$. To prove the claim \eqref{claim}, we choose $\sigma>0$ such that $q<p<\sigma<q(1+\frac{2}{n})$ and set
$U_k=\tilde w_k-\tilde w$. By the Gagliardo-Nirenberg inequality (cf. \cite[Lemma 2.5]{DH}), we deduce
\begin{equation*}\begin{split}\dashint_{Q_{1/2}}|U_k|^\sigma\,\mathrm {d}z\leq &
c\ \ \dashint_{-\frac{1}{2}}^0\Big(\ \dashint_{B_{1/2}}|U_k|^q+|DU_k|^q\,\mathrm {d}x\Big)^\gamma
\cdot\Big(\ \dashint_{B_{1/2}}|U_k|^2\,\mathrm {d}x\Big)^{\frac{1}{2}(\sigma-q\gamma)}\,\mathrm {d}t,
\end{split}\end{equation*}
where $\gamma=\frac{\sigma-2}{q-2+\frac{2q}{n}}$. From \eqref{wk}, \eqref{w}, \eqref{contradiction wk}
and the H\"older's inequality, we follow the arguments in \cite[page 217]{DH} to obtain
\begin{equation*}\begin{split}\dashint_{Q_{1/2}}|U_k|^\sigma\,\mathrm {d}z&\leq
c\Big(\ \ \dashint_{Q_{1/2}}|U_k|^q+|DU_k|^q\,\mathrm {d}z\Big)^\gamma
\cdot\Big(\ \
\dashint_{-\frac{1}{2}}^0\Big(\ \dashint_{B_{1/2}}|U_k|^2\,\mathrm {d}x\Big)^{\frac{\sigma-q\gamma}
{2(1-\gamma)}}\,\mathrm {d}t\Big)^{1-\gamma}
\\&\leq c\Big(\sup_{-\frac{1}{2}<t<0}\ \dashint_{B_{1/2}}\left|w_k(\cdot,t)\right|^2\,\mathrm {d}x\Big)^{\frac{\sigma-q\gamma}{2}-1+\gamma}
\|U_k\|_{L^2\left(Q_{1/2},\mathbb{R}^N\right)}^{2(1-\gamma)}
\\&\leq c
\Big(\ \ \dashint_{Q_{1/2}}|\tilde w_k-\tilde w|^2\,\mathrm {d}z\Big)^{1-\gamma}.
\end{split}\end{equation*}
Recalling that
$\tilde w_k\to \tilde w$ strongly in $L^2(Q_1,\mathbb{R}^N)$, this yields $U_k\to 0$
strongly in $L^\sigma(Q_{1/2},\mathbb{R}^N)$
and therefore \eqref{claim} holds true.

Step 4: To derive the contradiction, we now construct a sequence of $A_k$-caloric maps $v_k$ as indicated in Step 1.
To this aim we denote by
$v_k\in C^0(-\frac{9}{16},0;L^2(B_\frac{3}{4},\mathbb{R}^N))\cap L^2(-\frac{9}{16},0;W^{1,2}(B_\frac{3}{4},\mathbb{R}^N))$ the
unique weak solution
of the initial-boundary value problem:
 \begin{equation*}
	\begin{cases}
	\ \int_{Q_{3/4}}\left[v_k\cdot\partial_t \varphi-\big\langle A_kDv_k, D\varphi\big\rangle\right]
\mathrm {d}z=0,\quad\text{for\ any}\ \varphi\in C_0^\infty(Q_{3/4},\mathbb{R}^N)
\\
	\qquad \qquad \qquad \qquad \qquad \qquad v_k=w, \quad\text{on}\quad \partial_PQ_{3/4}.
	\end{cases}
\end{equation*}
Clearly, $v_k\in C^\infty(Q_{3/4},\mathbb{R}^N)$ and
we infer from the arguments in \cite[page 29-30]{DMS} that the following convergence holds
\begin{equation}\label{sc}
	\begin{cases}
	\ v_k\to w\qquad &\text{strongly\ in}\ L^2(Q_{3/4},\mathbb{R}^N),   \\
	\ Dv_k\to Dw\qquad &\text{strongly\ in}\ L^2(Q_{3/4},\mathbb{R}^{N\times n}).
	\end{cases}
\end{equation}
Using this together with \eqref{convergence} we obtain $\|w_k-v_k\|_{L^2(Q_{1/2})}\to 0$.
Recalling that $w$ is an $A$-caloric map, we conclude
from \cite[page 39]{B}, \eqref{w} and \eqref{sc} that
\begin{equation}\begin{split}\label{supremum}\sup_{Q_{1/2}}\left\{\ |v_k|,\ |Dv_k|,\ |w|,\ |Dw|\ \right\}^2
&\leq c\ \ \dashint_{Q_{3/4}}|w|^2+|Dw|^2+ |v_k|^2+|Dv_k|^2\,\mathrm {d}z
\\&\leq c\Big(\|w\|_{L^2(Q_1,\mathbb{R}^N)}^2+\|Dw\|_{L^2(Q_1,\mathbb{R}^{n\times N})}^2\Big)
\leq c(\lambda,\Lambda),
\end{split}
\end{equation}
for $k$ sufficiently large. For $p<\sigma<q(1+\frac{2}{n})$,
the estimate \eqref{supremum} yields that
\begin{equation}\label{sc1}
	\begin{cases}
	\ v_k\to w\qquad &\text{strongly\ in}\ L^\sigma(Q_{1/2},\mathbb{R}^N),   \\
	\ \gamma_k^{\frac{p_k-2}{p_k}}v_k\to \tilde w\qquad &\text{strongly\ in}\ L^\sigma(Q_{1/2},\mathbb{R}^N).
	\end{cases}
\end{equation}
Moreover, we conclude from \eqref{claim} and \eqref{sc1}$_2$ that
\begin{equation*}\begin{split}\dashint_{Q_{1/2}}\gamma_k^{p_k-2}|w_k-v_k|^{p_k}\,\mathrm {d}z
\leq c\Big(\ \dashint_{Q_{1/2}}|\tilde w_k-\tilde w|^\sigma\,\mathrm {d}z\Big)^{\frac{2}{\sigma}}
+c\Big(\ \dashint_{Q_{1/2}}|\tilde w-\gamma_k^{\frac{p_k-2}{p_k}}v_k|^\sigma\,\mathrm {d}z\Big)^{\frac{2}{\sigma}}\to0.
\end{split}\end{equation*}
This together with \eqref{sc1}$_1$ yields
\begin{equation}\begin{split}\label{con}\lim_{k\to\infty}\ \ \dashint_{Q_{1/2}}4|w_k-v_k|^2+2^{p_k}
\gamma_k^{p_k-2}|w_k-v_k|^{p_k}\,\mathrm {d}z=0,
\end{split}\end{equation}
since
$w_k\to w$ strongly in $L^2(Q_1,\mathbb{R}^N)$.
Finally, recalling that $\gamma_k^{p_k-2}\leq1$, we use \eqref{supremum} to deduce that
\begin{equation*}\begin{split}\varlimsup_{k\to\infty}\ \ &\dashint_{Q_{1/2}}4|v_k|^2+|Dv_k|^2+\gamma_k^{p_k-2}
(2^{p_k}|v_k|^{p_k}+|Dv_k|^{p_k})
\,\mathrm {d}z\leq \hat c(n,N,\gamma_2,\lambda,\Lambda).
\end{split}\end{equation*}
Therefore, the limit \eqref{con} is contrary to \eqref{B1result}
with $c_*\geq \hat c(n,N,\gamma_2,\lambda,\Lambda)$. This leads to a contradiction for $k$ sufficiently large.
We have thus proved the lemma.
\end{proof}
\bibliographystyle{abbrv}

\end{document}